\documentclass[11pt,hidelinks]{article}

\usepackage{orcidlink}

\title{Near-critical and finite-size scaling for \\
high-dimensional lattice trees and animals}

\author{
Yucheng Liu\,\orcidlink{0000-0002-1917-8330}\thanks{Department of Mathematics,
	University of British Columbia,
	Vancouver, BC, Canada V6T 1Z2.
	Liu: \href{mailto:yliu135@math.ubc.ca}{yliu135@math.ubc.ca}.
	Slade: \href{mailto:slade@math.ubc.ca}{slade@math.ubc.ca}.
	}
\and
Gordon Slade\,\orcidlink{0000-0001-9389-9497}$^*$
}
\date{\vspace{-5ex}} 

\usepackage[greek,english]{babel}

\usepackage{amsmath, amssymb, amscd, amsthm, amsfonts}
\usepackage{graphicx} 
\usepackage{hyperref} 

\usepackage{booktabs}
\usepackage{xcolor}
\usepackage{ dsfont, bm}
\usepackage[title]{appendix}
\usepackage{comment}
\usepackage{enumerate}
\usepackage{enumitem}
\usepackage{cite}

\usepackage[textwidth=480pt,textheight=650pt,centering]{geometry} 

\usepackage{tikz}
\tikzset{every picture/.style={line width=0.75pt}}

\theoremstyle{plain}
\newtheorem{theorem}{Theorem}[section]
\newtheorem{lemma}[theorem]{Lemma}
\newtheorem{conjecture}[theorem]{Conjecture}
\newtheorem{proposition}[theorem]{Proposition}
\newtheorem{corollary}[theorem]{Corollary}

\newtheorem{definition}[theorem]{Definition}

\numberwithin{equation}{section}

\newcommand{\ie}{i.e.}

\newcommand{\eps}{\varepsilon}

\newcommand{\Z}{\mathbb{Z}}

\newcommand{\R}{\mathbb{R}}
\newcommand{\C}{\mathbb{C}}

\newcommand{\T}{\mathbb{T}}

\newcommand{\del}{\partial}
\newcommand{\grad}{\nabla}
\newcommand{\inv}{^{-1}}
\renewcommand{\(}{\left(}
\renewcommand{\)}{\right)}
\newcommand{\half}{\frac{1}{2}}

\newcommand{\1}{\mathds{1}}

\newcommand{\nl}{\nonumber \\}

\renewcommand{\Re}{\mathrm{Re}\,}

\providecommand{\abs}[1]{\lvert#1\rvert}
\providecommand{\bigabs}[1]{\big\lvert#1\big\rvert}
\providecommand{\Bigabs}[1]{\Big\lvert#1\Big\rvert}
\providecommand{\biggabs}[1]{\bigg\lvert#1\bigg\rvert}

\providecommand{\norm}[1]{\lVert#1\rVert}
\providecommand{\bignorm}[1]{\big\lVert#1\big\rVert}
\providecommand{\Bignorm}[1]{\Big\lVert#1\Big\rVert}
\providecommand{\biggnorm}[1]{\bigg\lVert#1\bigg\rVert}

\newcommand{\Ann}{B}

\renewcommand{\mp}{m(p)}
\newcommand{\btil}{\tilde{b}}
\newcommand{\chip}{{\chi(p)}}
\newcommand{\chim}{\chi \supm}

\newcommand{\muz}{{\mu_z}}
\newcommand{\mup}{{\mu_p}}
\newcommand{\supm}{^{(m)}}

\newcommand{\supmp}{^{(\mp)}}
\newcommand{\supmS}{^{(m_S(z))}}
\newcommand{\supM}{^{(M)}}
\newcommand{\supN}{^{(N)}}
\newcommand{\supNm}{^{(N,m)}}
\newcommand{\supT}{^{ \mathbb{T} }}
\newcommand{\supk}[1]{^{(#1)}}

\newcommand{\supzerom}{^{(0,m)}}

\usepackage{mathrsfs}

\newcommand{\nnb}{\nonumber\\}

\newcommand{\veee}[1]{|\!|\!|#1|\!|\!|}
\newcommand{\xvee}{\veee{x}}
\providecommand{\nnnorm}[1]{\veee #1}
\newcommand{\const}{\mathrm{const}}

\newcommand{\crit}{_{p_c}}
\newcommand{\gp}{ g_p }
\newcommand{\gpc}{ g\crit }

\newcommand{\BL}{{B_L}}
\newcommand{\KIR}{K_{\rm IR}}

\newcommand{\Rd}{{\mathbb R^d}}
\newcommand{\Zd}{\mathbb Z^d}
\newcommand{\Td}{\mathbb T^d}

\newcommand{\dc}{d_{\mathrm{c}}}
\newcommand{\TA}{\mathcal{C}}

\newcommand{\LTPiFour}{
\begin{tikzpicture}[x=0.75pt,y=0.75pt,yscale=-1,xscale=1]

\draw    (195,100) -- (230,100) ;
\draw [shift={(195,100)}, rotate = -15] [color={rgb, 255:red, 0; green, 0; blue, 0 }  ][line width=0.75]    (0,5.59) -- (0,-5.59)(-5.03,5.59) -- (-5.03,-5.59)   ;
\draw    (190,140) -- (230,140) ;
\draw    (270,140) -- (310,140) ;
\draw [shift={(290,140)}, rotate = 0] [color={rgb, 255:red, 0; green, 0; blue, 0 }  ][fill={rgb, 255:red, 0; green, 0; blue, 0 }  ][line width=0.75]      (0, 0) circle [x radius= 3.35, y radius= 3.35]   ;
\draw    (235,140) -- (270,140) ;
\draw [shift={(235,140)}, rotate = 15] [color={rgb, 255:red, 0; green, 0; blue, 0 }  ][line width=0.75]    (0,5.59) -- (0,-5.59)(-5.03,5.59) -- (-5.03,-5.59)   ;
\draw    (275,100) -- (310,100) ;
\draw [shift={(275,100)}, rotate = -15] [color={rgb, 255:red, 0; green, 0; blue, 0 }  ][line width=0.75]    (0,5.59) -- (0,-5.59)(-5.03,5.59) -- (-5.03,-5.59)   ;
\draw    (315,140) -- (350,140) ;
\draw [shift={(315,140)}, rotate = 15] [color={rgb, 255:red, 0; green, 0; blue, 0 }  ][line width=0.75]    (0,5.59) -- (0,-5.59)(-5.03,5.59) -- (-5.03,-5.59)   ;
\draw    (230,100) -- (270,100) ;
\draw [shift={(250,100)}, rotate = 0] [color={rgb, 255:red, 0; green, 0; blue, 0 }  ][fill={rgb, 255:red, 0; green, 0; blue, 0 }  ][line width=0.75]      (0, 0) circle [x radius= 3.35, y radius= 3.35]   ;
\draw    (310,100) -- (350,100) ;
\draw    (230,100) -- (230,140) ;
\draw    (190,100) -- (190,140) ;
\filldraw[fill=white] (190,100) circle (0pt) node[left]{$0$};
\draw    (270,100) -- (270,140) ;
\draw    (310,100) -- (310,140) ;
\draw    (350,100) -- (350,140) ;
\filldraw[fill=white] (350,140) circle (0pt) node[right]{$x$};
\end{tikzpicture}
}

\newcommand{\LAPiFour}{
\begin{tikzpicture}[x=0.75pt,y=0.75pt,yscale=-1,xscale=1]

\draw    (195,100) -- (230,100) ;
\draw [shift={(195,100)}, rotate = -15] [color={rgb, 255:red, 0; green, 0; blue, 0 }  ][line width=0.75]    (0,5.59) -- (0,-5.59)(-5.03,5.59) -- (-5.03,-5.59)   ;
\draw    (190,140) -- (230,140) ;
\draw [shift={(210,140)}, rotate = 0] [color={rgb, 255:red, 0; green, 0; blue, 0 }  ][fill={rgb, 255:red, 0; green, 0; blue, 0 }  ][line width=0.75]      (0, 0) circle [x radius= 3.35, y radius= 3.35]   ;
\draw    (270,140) -- (310,140) ;
\draw [shift={(290,140)}, rotate = 0] [color={rgb, 255:red, 0; green, 0; blue, 0 }  ][fill={rgb, 255:red, 0; green, 0; blue, 0 }  ][line width=0.75]      (0, 0) circle [x radius= 3.35, y radius= 3.35]   ;
\draw    (235,140) -- (270,140) ;
\draw [shift={(235,140)}, rotate = 15] [color={rgb, 255:red, 0; green, 0; blue, 0 }  ][line width=0.75]    (0,5.59) -- (0,-5.59)(-5.03,5.59) -- (-5.03,-5.59)   ;
\draw    (275,100) -- (310,100) ;
\draw [shift={(275,100)}, rotate = -15] [color={rgb, 255:red, 0; green, 0; blue, 0 }  ][line width=0.75]    (0,5.59) -- (0,-5.59)(-5.03,5.59) -- (-5.03,-5.59)   ;
\draw    (315,140) -- (350,140) ;
\draw [shift={(315,140)}, rotate = 15] [color={rgb, 255:red, 0; green, 0; blue, 0 }  ][line width=0.75]    (0,5.59) -- (0,-5.59)(-5.03,5.59) -- (-5.03,-5.59)   ;
\draw    (230,100) -- (270,100) ;
\draw [shift={(250,100)}, rotate = 0] [color={rgb, 255:red, 0; green, 0; blue, 0 }  ][fill={rgb, 255:red, 0; green, 0; blue, 0 }  ][line width=0.75]      (0, 0) circle [x radius= 3.35, y radius= 3.35]   ;
\draw    (310,100) -- (350,100) ;
\draw    (230,100) -- (230,140) ;
\draw    (190,100) -- (190,140) -- (170,120) -- cycle ;
\filldraw[fill=white] (170,120) circle (0pt) node[left]{$0$};
\draw    (270,100) -- (270,140) ;
\draw    (310,100) -- (310,140) ;
\draw    (350,100) -- (350,140) ;
\filldraw[fill=white] (350,140) circle (0pt) node[right]{$x$};
\end{tikzpicture}
}

\begin{document}
\maketitle

\begin{abstract}
We consider spread-out models of lattice trees and lattice animals
on $\mathbb Z^d$, for $d$ above the upper critical dimension $d_{\mathrm c}=8$.
We define a correlation length and prove that it diverges
as $(p_c-p)^{-1/4}$ at the critical point $p_c$.
Using this, we prove that
the near-critical two-point function
is bounded above by $C|x|^{-(d-2)}\exp[-c(p_c-p)^{1/4}|x|]$.
We apply the near-critical bound to study lattice trees and lattice
animals on a discrete $d$-dimensional torus (with $d > d_{\mathrm c}$) of volume $V$.
For $p_c-p$ of order $V^{-1/2}$, we prove that the torus susceptibility is
of order $V^{1/4}$, and that the torus two-point function behaves as
$|x|^{-(d-2)} + V^{-3/4}$ and thus has a plateau of size $V^{-3/4}$.
The proofs require significant extensions of previous results obtained using
the lace expansion.
\end{abstract}




\section{Introduction and results}

\subsection{Introduction}

Lattice trees and lattice animals are combinatorial  models for the
 critical behaviour of branched polymers \cite{Jans15}, and
 are both believed to belong to the same universality class.
 From a mathematical point of view, the
analysis of their scaling limits and critical exponents in low dimensions
is at least as difficult
as it is for the notoriously intractable self-avoiding walk, and the theory
remains to be discovered.  This is despite the fact that an alternate model
of branched polymers does have a remarkable exact solution
in dimensions two and three \cite{BI03a,PS81}.

In dimensions larger than the upper critical dimension $\dc=8$
(first predicted in \cite{LI79}),
there has been significant progress.  This includes asymptotics for the
critical point \cite{BBR10,GP23,KS24,MS13}, proof of existence of mean-field
critical exponents
$\gamma = \frac 12$, $\nu_2 = \frac 14$, $\eta=0$ \cite{HS90b,HS92c,HHS03},
and proof of
super-Brownian scaling limits \cite{DS98,Holm08,Holm16,HP20,CFHP23}.
The lace expansion developed in \cite{HS90b} has been the primary tool.

Our purpose here is to extend the high-dimensional results in three ways,
for spread-out lattice trees and lattice animals in dimensions $d >\dc =8$,
at and near the critical point $p_c$:
\begin{itemize}
\item
We obtain an asymptotic formula $\xi(p) \asymp
(p_c-p)^{-1/4}$ for a correlation length that is believed to coincide with the exponential rate of decay of the two-point
function.
\item
We obtain a uniform upper bound
$C|x|^{-(d-2)}\exp[-c(p_c-p)^{1/4}|x|]$
for the near-critical two-point function.
\item
We apply the above results to study lattice trees and lattice animals on
the $d$-dimensional discrete torus.
On the torus of volume $V$, we prove that for $p$ at a distance
$V^{-1/2}$ below $p_c$, the torus susceptibility has size $V^{1/4}$ and the
two-point function behaves as $|x|^{-(d-2)}+V^{-3/4}$.  The constant term $V^{-3/4}$
is the torus ``plateau.''
The torus results form part of the general theory of high-dimensional finite-size
scaling described in \cite{LPS25-universal}.
\end{itemize}
The proofs apply and extend results for lattice trees and lattice animals
obtained using
the lace expansion in \cite{HS90b,HHS03}, and adapt methods recently applied to self-avoiding
walk \cite{Slad23_wsaw,Liu24}, percolation \cite{HMS23}, and the Ising model \cite{LPS25-Ising}
above their upper critical dimensions.

\smallskip \noindent \textbf{Notation.}
We write $a \vee b = \max \{ a , b \}$ and $a \wedge b = \min \{ a , b \}$.
We write $f \lesssim g$ to mean that there is a $C>0$ such that $f \le Cg$,
write $f \asymp g$ to mean that $f \lesssim g$ and $g \lesssim f$,
and write $f\sim g$ to mean that $\lim f/g=1$.
To avoid dividing by zero, with $|x|$ the Euclidean norm of $x\in \R^d$, we define
\begin{equation}
\label{eq:xvee1}
	\xvee = \max\{|x| , 1\}.
\end{equation}
Note that \eqref{eq:xvee1} does not define a norm on $\R^d$.

\subsection{Lattice trees and lattice animals}

Let $\mathbb{G}=(\mathbb{V},\mathbb{E})$ be a graph with vertex set
$\mathbb{V}$ and edge set $\mathbb{E}$.
A lattice animal is a finite connected subgraph of $\mathbb{G}$,
and a lattice tree is an acyclic lattice animal.
We restrict attention to $\mathbb{G}$ with vertex set $\mathbb{V}=\Z^d$ and edge
set $\mathbb{E}=\{\{x,y\}: 0< \|y-x\|_\infty \le L\}$;
each vertex therefore has degree $\Omega = (2L+1)^d - 1$.
The parameter $L$ will be taken to be large.
This defines the \emph{spread-out} models of lattice trees and lattice animals,
which are conjectured to be in the same universality class as each other and
as the nearest-neighbour
models (for which $\mathbb{E}=\{\{x,y\}: \|y-x\|_1 =1\}$).

We consider \emph{bond} trees and animals, as opposed to \emph{site}.
The number of bonds (a.k.a.\ edges) in a tree $T$ or an animal $A$ is denoted $\abs T, \abs A$.
The \emph{two-point functions} for lattice trees and lattice animals
are defined, for $p\ge 0$ and $x \in \Z^d$, as the generating functions
\begin{equation}
\label{eq:2ptfcn}
	G_p^t(x) = \sum_{T \ni 0,x} \Bigl(\frac{p}{\Omega} \Big)^{|T|} , \qquad
	G_p^a(x) = \sum_{A \ni 0,x} \Bigl(\frac{p}{\Omega} \Big)^{|A|} ,
\end{equation}
where the sums are over all trees or animals containing $0$ and $x$.
The \emph{one-point functions} are defined by
\begin{equation} \label{eq:gpdef}
    g_p^t = G_p^t(0)
    = \sum_{T \ni 0} \Bigl(\frac{p}{\Omega} \Big)^{|T|} ,
    \qquad g_p^a = G_p^a(0)
    = \sum_{A \ni 0} \Bigl(\frac{p}{\Omega} \Big)^{|A|} .
\end{equation}
To avoid repeating separate formulas for lattice trees and lattice animals,
we often omit the superscripts $t,a$
when a formula applies to both models.
With this convention,
the \emph{susceptibility} and the \emph{correlation length of order $2$}
are defined as
\begin{equation} \label{eq:chi}
	\chi (p) = \sum_{x \in \Z^d} G_p (x) ,  \qquad	
	\xi_2 (p) = \biggl(\frac{1}{\chi  (p)}\sum_{x \in \Z^d}|x|^2 G_p (x)\bigg)^{1/2}.
\end{equation}

Let $t_n$ denote the number of $n$-bond lattice trees or $n$-bond lattice animals \emph{modulo translation}.
A simple concatenation argument shows that
$t_nt_m \le t_{n+m+1}$, from which it
follows that
$t_n^{1/n}$ approaches a positive limit as $n \to\infty$.  We define
the \emph{critical points} $p_c^t,p_c^a$  (depending on $d,L$) by
\begin{equation} \label{eq:tn_subadd}
    \lim_{n\to\infty} t_n^{1/n} =\frac{\Omega}{p_c}.
\end{equation}
Also, $t_{n} \le (\Omega/p_c)^{n}$ for all $n$ \cite[Section~2.1]{Jans15}.
The series
in \eqref{eq:2ptfcn}--\eqref{eq:chi} all
converge for positive $p$ less than $p_c^t,p_c^a$.
Both critical points
converge to $e\inv$
as $L\to \infty$ \cite{Penr94,KS24}.
A standard differential inequality \cite{BFG86,TH87} implies that, in all dimensions,
\begin{equation}
\label{eq:chilb-general-d}
    \chi(p) \gtrsim \frac{1}{(1-p/p_c)^{1/2}}.
\end{equation}

In dimensions $d > 8$ and when $L$ is sufficiently large,
much more has been proved using the lace expansion.
It is proved in
\cite{HS90b,HHS03} that
\begin{align}
\label{eq:gamma_half}
 \chi (p) &\asymp \frac{1}{(1-p/p_c)^{1/2}},
	\\
\label{eq:nu2_quarter}
\xi_2 (p) &\asymp \frac{1}{(1-p/p_c)^{1/4}}     ,
	\\
\label{eq:eta_zero}
 G\crit (x) & \sim \frac{B}{|x|^{d-2}},
\end{align}
with $B>0$.
In terms of the critical exponents $\gamma$ for the susceptibility, $\nu_2$
for the correlation length of order $2$, and $\eta$ for the critical two-point
function,
\eqref{eq:gamma_half}--\eqref{eq:eta_zero} are statements that
\begin{equation}
    \gamma = \frac 12, \qquad \nu_2 = \frac 14, \qquad \eta =0.
\end{equation}
In addition, it is proved in \cite{HS90b} that the one-point functions obey
$    1 \le g_{p_c}  \le 4 $.
For lattice trees,
the upper and lower bounds in \eqref{eq:gamma_half}--\eqref{eq:nu2_quarter}
are improved in \cite{HS92c}
to asymptotic formulas, which are used to prove that
$t_n^t \sim \const (\Omega/p_c)^{n}n^{-5/2}$.  An asymptotic formula for
the mean radius of gyration of $n$-bond lattice trees is also proved in \cite{HS92c}.

\subsection{The correlation length}

For $m\ge 0$,
we define the \emph{tilted susceptibility} by
\begin{equation}
\label{eq:chimdef}
    \chim(p) = \sum_{x\in \Z^d}G_p(x)e^{mx_1},
\end{equation}
and we define the \emph{correlation length} $\xi(p)$ and \emph{mass} $m(p)$ by
\begin{align}
\label{eq:mtilt}
    \xi(p)^{-1} = \mp = \sup \Bigl \{ m \ge 0 :
	 \chim(p) < \infty 	\Bigr \} .
\end{align}
We believe, but do not prove, that
the definition \eqref{eq:mtilt} agrees with the more common definition
\begin{equation}
\label{eq:mdef}
    \tilde\xi(p)^{-1} = \tilde m(p) =   \liminf_{n \to \infty}\frac{-\log G_p(ne_1)}{n},
\end{equation}
where $e_1 = (1, 0, \dots, 0)\in \Zd$.
It follows from \eqref{eq:mtilt} and $\Zd$-symmetry that
\begin{align} \label{eq:G_exp_decay}
G_p(x) \le \chim(p) e^{-m \norm x_\infty}
	\qquad (x\in \Zd)
\end{align}
whenever $0 \le m < \mp$, and therefore $\mp \le \tilde m(p)$.

For models that satisfy FKG or GKS type inequalities,
e.g., for percolation and the Ising model,
$\tilde m(p)$ provides an off-axis exponential bound on $G_p(x)$,
from which $\mp = \tilde m(p)$ can be deduced.
Also, the fact that $\tilde m(p) \to 0$ as $p \to p_c$ can sometimes
be deduced from the off-axis bound (e.g., \cite[Corollary~A.4(b)]{Hara90} for percolation).
However,
we do not know a proof of any of these facts for lattice trees and animals,
and we work with $\mp$ instead.

To see that $\mp > 0$ when $p < p_c$,
we recall that $t_n$ denotes
the number of $n$-bond lattice trees/animals modulo translation.
Since an $n$-bond tree contains exactly $n+1$ vertices
and an $n$-bond animal contains at most $n+1$ vertices,
with $\TA$ denoting either tree or animal,
the susceptibilities obey
\begin{equation}
\label{eq:chi-tn}
    \chi(p)
    =
    \sum_{x\in\Z^d} \sum_{\TA \ni 0,x} \Big(\frac{p}{\Omega} \Big)^{|\TA|}
    \le
    \sum_{n=0}^\infty t_n(n+1)^2 \Big(\frac{p}{\Omega} \Big)^{n}.
\end{equation}
Similarly, since
$e^{mx_1} \le e^{mnL}$ if $x$ is in an $n$-bond tree or animal,
and since $t_n \le (\Omega/p_c)^{n}$, we have
\begin{equation}
\label{eq:chi-tn-2}
    \chim(p)
    \le
    \sum_{n=0}^\infty t_n(n+1)^2 \Big(\frac{p e^{mL}}{\Omega} \Big)^{n}
    \le
    \sum_{n=0}^\infty (n+1)^2 \Big(\frac{p e^{mL}}{p_c} \Big)^{n} ,
\end{equation}
which is finite if $p e^{mL} < p_c$.  It follows that
$m(p) \ge L^{-1} \log(p_c / p) > 0$ if $p<p_c$.

\subsection{Results}

\subsubsection{Results for the infinite lattice}

Our first result is the following theorem which shows that,
above the upper critical dimension,
the mass $m(p)$ of \eqref{eq:mtilt}
does go to zero at the critical point, and it does so
in a precise way.
Theorem~\ref{thm:mass} is a statement that the critical exponent $\nu$
is given by $\nu =\frac 14$.

\begin{theorem}
\label{thm:mass}
Let $d>8$
and $L \ge L_0$ with $L_0$ sufficiently large.
For lattice trees, there is a constant $a_L>0$ such that
$m(p) \sim a_L (p_c-p)^{1/4}$.
For lattice animals, there are $L$-dependent constants such that
$m(p)\asymp (p_c-p)^{1/4}$.
\end{theorem}

The proof of Theorem~\ref{thm:mass} employs a technique that was introduced
in the context of percolation \cite{Hara90} and then later adapted
to self-avoiding walk \cite{HS92a}.  An exponential tilt of the two-point function
is used to prove that $\chi(p)$ is asymptotically equivalent to $\mp^{-2}$.
Given this asymptotic relation, Theorem~\ref{thm:mass} follows from \eqref{eq:gamma_half} and its improvement to an asymptotic relation for lattice trees in \cite{HS92c}.
If \eqref{eq:gamma_half} is also improved to an asymptotic relation for lattice animals, our proof will yield an asymptotic relation also.

The next theorem extends
the critical behaviour of $G\crit$ in
\eqref{eq:eta_zero} to a near-critical upper
bound for the two-point function.

\begin{theorem}
\label{thm:near_critical}
Let $d>8$
and $L \ge L_0$ with $L_0$ sufficiently large.
For both lattice trees and lattice animals,
there are constants $c \in (0,1)$ (independent of $L$)
and $C_L, \delta_L > 0$ such that
\begin{equation} \label{eq:near_critical}
    G_p(x) \le \frac{C_L}{|x|^{d-2}} e^{-cm(p)|x|}
\end{equation}
for all $p \in [p_c - \delta_L, p_c)$ and for all $x \in \Z^d$.
\end{theorem}

For $p<p_c$, the two-point function is expected to exhibit Ornstein--Zernike
decay of the form $m(p)^{(d-3)/2} |x|_p^{-(d-1)/2} e^{-m(p)|x|_p}$, with
a $p$-dependent norm (\emph{not} the $\ell_p$ norm) that approaches
the $\ell_2$ norm as $p \to p_c$ \cite{MS22}.
This would imply that \eqref{eq:near_critical} cannot remain true
 for all $x$, with fixed $p$, unless $c<1$.

Analogues of
Theorem~\ref{thm:near_critical}
have been proved for self-avoiding walk for $d>4$
\cite{Slad23_wsaw,Liu24,DP25a}, for the Ising model for $d>4$ \cite{DP24},
and for percolation for $d>6$ \cite{HMS23,DP25b,Pani24_thesis}.
The proof of Theorem~\ref{thm:near_critical}
follows the method of \cite{Slad23_wsaw} and uses
an analogous theorem for the two-point function of the spread-out random walk,
which we prove using the deconvolution strategy of \cite{LS24b}.

\subsubsection{Results for the torus}
\label{sec:torus-results}

Fix an $L \ge L_0$.
For an integer $r \ge 2L+1$, we define $\Lambda_r$ to be the discrete
box $[ - \frac r 2, \frac r 2 )^d \cap \Z^d$ of volume $V = r^d$.
We write $\T_r^d = (\Z/r\Z)^d$ for the $d$-dimensional discrete torus of period $r$,
with the spread-out edge set inherited from projecting the $\Zd$ edge set
$\mathbb{E}=\{\{x,y\}: 0< \|y-x\|_\infty \le L\}$ to the torus.
We identify a point $x\in\T_r^d$ with its representative in
$\Lambda_r$.
We are interested in estimates that are uniform in large $r$.
Constants throughout Section~\ref{sec:torus-results} may depend on $L$.

The theory of high-dimensional finite-size scaling has received considerable
attention in the physics literature, including the important role of boundary conditions.
A small sample of relevant papers is \cite{BKW12,WY14,GEZGD17,BEHK22,DGGZ22,FMPPS23,LPS25-universal}.
Using Theorems~\ref{thm:mass}--\ref{thm:near_critical}, the following theorem
and its corollary
establish the relevant scaling for (spread-out) lattice trees and lattice animals
in dimensions $d>\dc=8$.
We write $G_p^\T(x)$ for the two-point function
for lattice trees or lattice animals on the torus $\T_r^d$,
defined analogously to \eqref{eq:2ptfcn}.
Let $\delta_L$ be the constant of Theorem~\ref{thm:near_critical}.

\begin{theorem}
\label{thm:torus}
Let $d>8$ and $L \ge L_0$.
For both lattice trees and lattice animals on $\T_r^d$,
there are constants $c,C>0$ such that
for all $p \in [p_c - \delta_L, p_c)$ and all $x\in \T_r^d$,
\begin{align}
\label{eq:GTub}
G_p \supT (x)
	&\le G_p(x) + \frac{C}{m(p)^2 r^d},
	\\
\label{eq:GTlb}
G_p \supT(x)
	&\ge \bigg(1- \frac{C}{m(p)^8r^d} \biggr)
	\biggl( G_p(x) + \frac {c} {m(p)^2 r^d} \bigg)
\qquad (p_c - p \le c r^{-4}).
\end{align}
\end{theorem}

We will show in Lemma~\ref{lem:Glb}, as a straightforward consequence of \eqref{eq:eta_zero}
and a standard differential inequality, that there is a $c_0>0$ such that
\begin{equation}
\label{eq:Glb-intro}
    G_p(x) \ge \frac{c_0}{ \nnnorm x^{d-2}}
    \qquad
    (|x|\le c_0(p_c-p)^{-1/4} ) .
\end{equation}
Using this, we have the following corollary of Theorem~\ref{thm:torus},
expressed in terms of the volume $V=r^d$.

\begin{corollary}
\label{cor:plateau}
Let $d>8$ and $L \ge L_0$.
There exists $C_1 > 0$ such that for $r$ sufficiently large
and $p_* = p_c - C_1 V^{-1/2}$,
\begin{equation} \label{eq:Gplateau}
G_{p_*} \supT(x) \asymp \frac{1}{\nnnorm x^{d-2}} + \frac{1}{V^{3/4}}
	\qquad (x\in \T_r^d)
\end{equation}
and
\begin{equation}
    \chi^{\T}(p_*) \asymp V^{1/4}.
\end{equation}
\end{corollary}

\begin{proof}
Let $p_* = p_c - C_1 V^{-1/2}$, with the constant $C_1$ to be chosen later.
We restrict to $r$ sufficiently large so that $p_* \ge p_c - \delta_L$ and Theorem~\ref{thm:torus} applies at $p_*$.
We note that $p_c - p_* = C_1 V^{-1/2} = C_1 r^{-d/2}$ is $o(r^{-4})$ since $d > 8$, so the lower bound \eqref{eq:GTlb} applies when $r$ is large.

By Theorem~\ref{thm:mass}, there is an $a>0$ such that
$m(p_*)\ge a (p_c-p_*)^{1/4} = aC_1^{1/4}V^{-1/8}$,
with a matching upper bound.
The plateau term $m(p_*)^{-2} r^{-d}$ in \eqref{eq:GTub}--\eqref{eq:GTlb} is therefore bounded above and below by multiples of $C_1^{-1/2} V^{-3/4}$.
Combining \eqref{eq:GTub} with
$G_{p_*}(x) \lesssim \nnnorm x^{-(d-2)}$ from \eqref{eq:eta_zero},
we get the upper bound in \eqref{eq:Gplateau}.

For the lower bound,
if we take $C_1 \ge a^{-4}\sqrt{2C}$, then the coefficient in
\eqref{eq:GTlb} obeys
\begin{equation} \label{eq:plateau_pf1}
    1- \frac{C}{m(p_*)^8 r^d} \ge 1 - \frac{C}{a^8C_1^2} \ge \frac 12.
\end{equation}
Also,
since $|x|\le r \sqrt d$ for all $x\in \T_r^d$, the restriction on $x$
in \eqref{eq:Glb-intro} is satisfied throughout $\T_r^d$ if
$r \sqrt d \le c_0 C_1^{-1/4}V^{1/8} = c_0 C_1^{-1/4}r^{d/8}$,
which holds for all large $r$ as $d > 8$.
The desired lower bound then follows by inserting \eqref{eq:plateau_pf1}
and \eqref{eq:Glb-intro} into \eqref{eq:GTlb}.

Finally, summation over $x$ in the torus gives
\begin{equation}
    \chi^{\T}(p_*)
    \asymp
    \sum_{x\in\Lambda_r}\Big(\frac{1}{\nnnorm x^{d-2}} + \frac{1}{V^{3/4}} \Big)
    \asymp r^2 + V^{1/4},
\end{equation}
and the term $V^{1/4} = r^{d/4}$ dominates since $d>8$.
This concludes the proof.
\end{proof}

We believe that Corollary~\ref{cor:plateau} remains valid as long as $|p_c-p|$ is
of order $V^{-1/2}$, and we refer to $V^{-1/2}$ as the size of the \emph{scaling window}.
In Corollary~\ref{cor:plateau}, the size of the window, the plateau, and the susceptibility
are respectively instances of the more general formulas
$V^{-\frac{2}{\gamma \dc}}$, $V^{\frac{2}{\dc}-1}$, and $V^{\frac{2}{\dc}}$
\cite{LPS25-universal}, where $\dc$ is the upper critical dimension.
These formulas have been proven to hold for self-avoiding walk with $\dc=4$
\cite{Slad23_wsaw,Liu24}, the Ising model with $\dc=4$ \cite{LPS25-Ising},
the hierarchical $|\varphi|^4$ model with $\dc=4$ \cite{MPS23,PS25},
and percolation with $\dc=6$ \cite{HMS23}.
Our results show that lattice trees and lattice animals also fit into this general theory, which is presented in \cite{LPS25-universal} (see also \cite{BEHK22} for the physics
perspective).
We do not expect Theorem~\ref{thm:torus} to hold when the periodic boundary
condition is replaced by free boundary conditions.  Instead, with free boundary
conditions we have conjectured in \cite{LPS25-universal} that the asymptotic relations of
Corollary~\ref{cor:plateau} hold for values of $p$ in a window of the same width $V^{-1/2}$
but centred at a \emph{pseudocritical point} $p_{r,c}$ that is shifted
larger than $p_c$ by an amount of order $r^{-1/\nu}=r^{-4}$.  Note that this
shift is larger than $V^{-1/2} = r^{-d/2}$ in dimensions $d > \dc = 8$,
so these windows for free and periodic boundary
conditions do not overlap.

Inspired by the results of \cite{MPS23,PS25} for the hierarchical $|\varphi|^4$
model in dimensions $d \ge 4$, we have the following conjecture for the amplitude
for the two-point function plateau and the susceptibility throughout the critical
window $p=p_c+sV^{-1/2}$, $s \in \R$.
Let
\begin{equation}
    I_0(s) = \int_0^\infty e^{-\frac 14 t^4 + \frac 12 s t^2} dt
    \qquad
    (s \in \R).
\end{equation}
Via the change of variables $t = x^4/4$, $I_0 (s)$ can be rewritten in terms of
the \emph{Fax\'en integral}
\begin{equation}
    \mathrm{Fi} (\alpha, \beta ; y) = \int_0^{\infty} e^{-t + y t^{\alpha}} t^{\beta - 1} dt
    \qquad
    (0 \le \alpha <1 ,  \;\;\beta>0)
\end{equation}
as
$I_0 (s) = 2^{-\frac{3}{2}} \mathrm{Fi} ({\textstyle \frac{1}{2}, \frac{1}{4} ; s })$.
Its asymptotic behaviour can then be read from
\cite[Ex.~7.3, p.~84]{Olve97}:
\begin{align}
    I_0 (s)
    \sim
    \begin{cases}
    \sqrt{\frac{\pi}{2}} |s|^{-1/2} & (s \to -\infty)
    \\
    \sqrt{\pi} s^{-1/2}e^{s^2/4} & (s \to +\infty).
    \end{cases}
\end{align}

\begin{conjecture}
\label{conj:profile}
Let $d>8$.
For lattice trees or lattice
animals on $\T_r^d$, there are positive constants $\lambda_1,\lambda_2$ such that,
for all $x\in \T_r^d$ and all $s \in \R$,
\begin{equation}
    G_{p_c +sV^{-1/2}}^\T(x) - G_{p_c}(x)
    \sim
    \lambda_1 I_0(\lambda_2 s) \frac{1}{V^{3/4}}
\end{equation}
as $V \to \infty$.
By summation over all $x \in \T_r^d$,
\begin{equation}
\label{eq:chi-conj}
    \chi^\T(p_c+sV^{-1/2}) \sim  \lambda_1 I_0(\lambda_2 s) V^{1/4}.
\end{equation}
\end{conjecture}

In \cite{LS25_profile}, we show that \eqref{eq:chi-conj} holds
for  trees and connected subgraphs on the complete graph on $V$ vertices.
The universal behaviour of the two-point function on the complete graph
is also established in \cite{LS25_profile}, including a proof that the
profile $I_0$ and the power $V^{-3/4}$
also occur in the scaling of the two-point
function at a pair of distinct vertices.
As motivated in \cite{MPS23,PS25},
Conjecture~\ref{conj:profile} is founded on the (generally unproven)
principle that the complete graph profile applies universally at and above
the upper critical dimension.

\subsection{Organisation}

We prove
Theorem~\ref{thm:torus} assuming Theorems~\ref{thm:mass}--\ref{thm:near_critical}
in Section~\ref{sec:torus}.
General facts needed for our proof that $m(p) \to 0$ as $p\to p_c$ are presented in Section~\ref{sec:chim}.
We prove Theorem~\ref{thm:mass} in Section~\ref{sec:mass} and
Theorem~\ref{thm:near_critical} in Section~\ref{sec:near_critical}.
Some proofs are deferred to two appendices.
Appendix~\ref{app:diagram} is concerned with diagrammatic estimates
for the lace expansion under
an exponential tilt.
Appendix~\ref{app:S} proves the analogue of Theorems~\ref{thm:mass}--\ref{thm:near_critical}
for the spread-out random walk.

We need to take the spread-out parameter $L$ large on several occasions.
The constant $L_0$ appearing in the main results is taken to be the maximum of
the constants $L_1,L_2,L_3$ that occur below.

\section{Finite-size scaling: proof of Theorem~\ref{thm:torus}}
\label{sec:torus}

In this section, we prove
Theorem~\ref{thm:torus} assuming Theorems~\ref{thm:mass}--\ref{thm:near_critical}.
In Section~\ref{sec:lift}, we formalise the notion of \emph{lift} that we use to
compare the torus models with their $\Z^d$ counterparts.
Lower bounds needed for the proof of Theorem~\ref{thm:torus}
are presented in Section~\ref{sec:Glb}, and upper bounds are presented in Section~\ref{sec:Gammaub}.
In particular, in Section~\ref{sec:Glb} we show that the asymptotic formula
\eqref{eq:eta_zero} for $G_{p_c}(x)$ implies an lower bound on $G_p(x)$.
Finally, in Section~\ref{sec:pf-torus} we prove Theorem~\ref{thm:torus}.

We assume throughout this section that $d >8$ and $L \ge L_0$ are fixed,
where $L_0$ is from Theorems~\ref{thm:mass}--\ref{thm:near_critical}.
All constants in this section are permitted to depend on $d,L$.

We write $\pi:\Z^d \to \Lambda_r$ for the torus projection, i.e.,
if $x \in \Lambda_r$ and $u \in \Z^d$
then $\pi(x+ru)=x$.
If $\pi(x)=\pi(x')$ then we say $x'$ and $x$ are \emph{(torus) equivalent}
and we write $x' \equiv x$; this occurs if and only if
$x'=x+ru$ for some $u \in \Z^d$.

\subsection{Lift from the torus to the infinite lattice}
\label{sec:lift}

We compare a lattice tree or animal on the torus with its \emph{lift}
(unwrapping) to $\Z^d$, defined as follows.

\begin{definition} \label{def:lift}
Let $\T_r^d = (\Z/r\Z)^d$ be the $d$-dimensional discrete torus of period $r$,
with the spread-out edge set inherited from projecting the $\Zd$ edge set
$\mathbb{E}=\{\{x,y\}: 0< \|y-x\|_\infty \le L\}$ to the torus.
Assume that $r \ge 2L+1$.
\begin{enumerate}
\item \emph{(Lift of a walk.)}
Let $\omega=(\omega_0,\omega_1,\dots,\omega_n)$ be a torus walk
starting from $\omega_0 = 0$ and taking steps from $\mathbb{E}$.
The \emph{lift} of $\omega$ is
the $\Z^d$ walk $\bar \omega$ defined by $\bar\omega(0)=0$ and
\begin{equation}
\label{e:liftdef}
\bar \omega(k) - \bar \omega(k-1) = \omega(k) -  \omega(k-1)
	\qquad (1\leq k \leq n),
\end{equation}
using the identification of $\T^d_r$ with $\Lambda_r = [ - \frac r 2, \frac r 2 )^d \cap \Z^d$.
The restriction $r \ge 2L+1$ ensures that the lift is a bijection
from torus walks starting from the origin
to $\Zd$ walks starting from the origin.

\item \emph{(Lift of a lattice tree.)}
Let $T$ be a torus lattice tree containing $0$,
so that for any $x\in T$,
there is a unique self-avoiding walk in the tree from
$0$ to $x$.  The \emph{lift} of $T$ is defined to be the union over $x \in T$
of the lifts of those torus walks to $\Zd$.
This defines an injection from
torus lattice trees containing the origin
to $\Zd$ lattice trees containing the origin.
It is not a surjection because lifted trees cannot contain
any pair of $\Zd$ vertices that project to the same torus vertex.

\item \emph{(Lift of a lattice animal.)}
We use the lexicographic order on $[-L,L]^d \cap \Zd \setminus \{ 0 \}$
and on $\Lambda_r$, and we use the order on $\Lambda_r$ to induce an order on
the torus vertices.
Let $A$ be a torus animal containing $0$.
We first grow a spanning tree in $A$ recursively as follows,
starting from $\{0\}$.
Given the current tree,
we choose its first
vertex which is incident to an edge in $A$ that would not create a cycle when added to the tree,
and we add the first such edge to the tree.
The process terminates
in a spanning tree $T$ for $A$ when no further edges can be added.
We lift $T$ to a $\Zd$ tree $\hat T$.
For the remaining excess edges in $A \setminus T$,
we write them in the form $\{u,v\}$
with $u$ less than $v$ in the torus vertex order,
and we lift the edge
 to $\{ \hat u, \hat u + (v-u) \}$
where $\hat u\in \hat T$ projects to $u$.
The lifting is injective because we can recover the torus animal by projection.
See Figure~\ref{figure:LTlift} for an example.

\end{enumerate}
\end{definition}

\begin{figure}[h]
\center{
\parbox{0.2\linewidth}{\includegraphics[scale = 0.8]{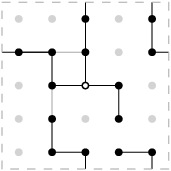}}
\qquad
\parbox{0.2\linewidth}{\includegraphics[scale = 0.8]{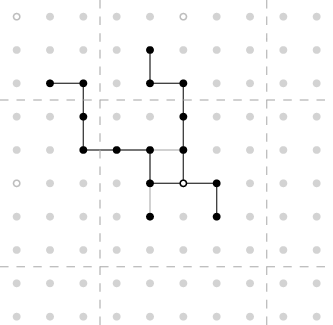}}
\caption{The lift of a lattice animal from $\T_5^2$ to $\Z^2$.
A spanning tree and its lift are shown in black.
The two excess edges are shown in grey.
Note that the lift of an animal can break torus cycles.
}
\label{figure:LTlift}
}
\end{figure}

\begin{lemma} \label{lem:LA_subtree}
Let $A$ be a $\Zd$ animal containing $0$, such that
$A$ does not contain any pair of torus equivalent points.
Let $\pi(A)$ be the projection of $A$ onto the torus.
Then $A$ is the lift of $\pi(A)$.
\end{lemma}

\begin{proof}
By Definition~\ref{def:lift},
the lift of $\pi(A)$ is defined by first lifting a spaning tree $T$ of $\pi(A)$.
We first show that $T$ lifts to a subtree of $A$,
by coupling the construction of $T$ in $\pi(A)$
with the construction of a subtree $\tilde T$ of $A$ that will coincide with $\hat T$.
To construct $\tilde T$,
we use the same ordering of $[-L,L]^d \cap \Zd \setminus \{ 0 \}$
and of $\Lambda_r$ as in Definition~\ref{def:lift},
and we order vertices of $\Zd = \Lambda_r + r\Z^d$ via the lexicographical ordering on the product space $\Lambda_r \times \Zd$.
We then run the same spanning tree algorithm starting from $\{0\}$ in $A$, and we stop when the construction of $T$ in $\pi(A)$ terminates.
By the chosen ordering of $\Zd$ vertices,
and using that $A$ does not contain any pair of torus equivalent points,
at each time when a torus edge is added to $T$,
some $\Zd$ edge in $A$ that projects to that torus edge will be added to $\tilde T$.
This ensures that the lift $\hat T$ of $T$ satisfies $\hat T  = \tilde T$.
In particular, $\hat T$ is a subtree of $A$, as claimed.

We now consider the lifting of edges in $\pi(A) \setminus T$.
Let $\norm A$ denote the number of vertices in $A$.
Since
$A$ does not contain any pair of torus equivalent points,
we have $\norm{ \pi(A) } = \norm A$, so the lift $\hat T = \tilde T$ of $T$
satisfies
\begin{equation}
\norm{ \hat T } = \norm T = \norm { \pi(A) } = \norm A ,
\end{equation}
and $\hat T$ must be a \emph{spanning} tree of $A$.
When $\{\check u ,\check v \} \in \pi(A) \setminus T$ is lifted,
we know there is a unique edge $\{u,v\} \in A$ that projects to $\{\check u ,\check v\}$,
and we know that $u,v  \in \hat T$ because $\hat T$ is spanning.
Since $\norm{u-v}_\infty \le L$,
no matter which of the vertices $\check u, \check v$ is smaller,
the edge is always lifted to $\{u,v\} \in A$.
This concludes the proof that $\pi(A)$ lifts to $A$.
\end{proof}

In discussions that apply equally to lattice trees and lattice animals,
we refer to them as \emph{polymers}
and write $\TA$ for a generic polymer.
We define the functions $\psi_p$ and $E_p$ on $\T_r^d$ by
\begin{align}
    \psi_p(x) =
    \sum_{ \substack{ x' \equiv x \\ x' \ne x } } G_p(x') ,
\qquad
\label{eq:Epdef}
    E_p(x) = \sum_{x' \equiv x} \sum_{ y\in\Z^d}
    \sum_{ \substack{ y' \equiv y \\  y'\neq y} }
    \sum_{\TA \ni 0,x',y,y'} \Big(\frac{p}{\Omega} \Big)^{|\TA|} .
\end{align}

\begin{proposition} \label{prop:GGam}
For all $d \ge 1$, $p \ge 0$, and $x \in \T_r^d$,
\begin{align} \label{eq:Gamma-G}
\psi_p(x) - E_p(x)
\le G_p \supT (x) - G_p(x)
\le \psi_p(x) .
\end{align}
\end{proposition}

\begin{proof}
The upper bound is due to injectivity of lifting.
Since the lift of a torus polymer containing $x\in \T_r^d$ must contain a $\Zd$ point $x' \equiv x$,
we get the desired upper bound
\begin{align}
G_p\supT(x) \le G_p(x) + \psi_p(x) .
\end{align}

For the lower bound,
we use inclusion--exclusion and exclude from $G_p(x) + \psi_p(x)$
the contribution due to $\Zd$ polymers that are not the lift of a torus polymer.
For lattice trees,
the exclusion term consists exactly of trees that contain $0,x'$ for some $x'
\equiv x$ as well as a pair of distinct equivalent points $y,y'$,
and this gives rise to the subtracted term $E_p(x)$ in \eqref{eq:Gamma-G}.
For lattice animals,
Lemma~\ref{lem:LA_subtree} implies that
any $\Zd$ animal that does not contain any pair of torus equivalent points
is the lift of some torus animal (namely, its torus projection),
so the exclusion term is again bounded by $E_p(x)$. 
This gives the lower bound of \eqref{eq:Gamma-G} and concludes the proof.
\end{proof}

\subsection{Lower bounds on $G_p$ and $\psi_p$}
\label{sec:Glb}

The $\Zd$ convolution of absolutely summable functions $f,g:\Z^d\to \C$
is defined by
\begin{equation}
\label{eq:Zdconv}
    (f*g)(x)=\sum_{y\in \Z^d} f(y)g(x-y)
     \qquad (x \in \Z^d).
\end{equation}
We will use the fact (see \cite[Proposition~1.7]{HHS03}) that if
$|f_i(x)| \lesssim \nnnorm{x}^{-(d-2)}$ for $i=1,\ldots, k$, and if $d>2k$, then
\begin{equation}
\label{eq:prop1.7}
    |(f_1* \cdots * f_k)(x)| \lesssim \frac{1}{\nnnorm{x}^{d-2k}}.
\end{equation}

\begin{lemma} \label{lem:Glb}
Let $d > 8$ and $L \ge L_0$.
Then there is a constant $c_0 > 0$ such that
\begin{equation} \label{eq:Glb}
G_p(x)  \ge \frac{c_0}{ \xvee^{d-2}}
	 \qquad (\norm x_\infty \le c_0 (p_c - p)^{-1/4}).
\end{equation}
\end{lemma}

\begin{proof}
Let $\rho = p_c-p$.
A standard differential inequality (the upper skeleton inequality of \cite{HS90b}),
together with the upper bound on $\chip$ in \eqref{eq:gamma_half}, implies that
\begin{equation}
p \frac{\mathrm d}{\mathrm d p}[G_p(x)]
\le \frac \Omega 2 \chip (G_p*G_p)(x)
\lesssim  \rho^{-1/2} (G_{p_c}*G_{p_c})(x).
\end{equation}
By integrating the differential inequality, and by
using the critical behaviour of $G\crit$ in \eqref{eq:eta_zero}
and the convolution estimate \eqref{eq:prop1.7} with $k=2$,
we find that
\begin{align}
G_p(x) \ge G\crit(x) - (G\crit(x) - G_p(x))
\ge \frac {C_1} { \nnnorm x ^{d-2} } -  \rho^{1/2} \frac {C_2}{ \nnnorm x ^{d-4} }
\end{align}
for some $C_1, C_2 > 0$.
For $0 < \abs x \le \eps (p_c - p)^{-1/4} = \eps \rho^{-1/4}$,
this gives $G_p(x) \ge (C_1 - \eps^2 C_2) \nnnorm x^{-(d-2)}$.
Since $G_p(0) \ge 1$ by definition, the claim follows
with $c_0=\eps$
by choosing $\eps$ sufficiently small and using the equivalence of norms on $\Rd$.
\end{proof}

\begin{lemma} \label{lem:psi_lb}
Let $d >8$, $L \ge L_0$,
and $c_3=\frac {16}{81} c_0^4$,
where $c_0$ is the constant of Lemma~\ref{lem:Glb}.
Then
\begin{equation}
\psi_p(x) \gtrsim   \frac{1}{\mp^2r^d}
\end{equation}
for all $r\ge 2L+1$, all $p \in [p_c - c_3 r^{-4}, p_c)$, and all $x\in \T_r^d$.
\end{lemma}

\begin{proof}
We apply Lemma~\ref{lem:Glb} to $x +ru$ with
$x\in \Lambda_r$ and $\norm u_\infty \le M$, with a well-chosen $M \ge 1$.
Since \eqref{eq:Glb} holds when $\|x+ru\|_\infty \le c_0 (p_c - p)^{-1/4}$,
and since $\norm x_\infty \le \frac r 2$,
it suffices to have
\begin{equation}\label{eq:proof1.11}
    M \ge  1, \qquad \frac r2 + rM \le c_0 (p_c - p)^{-1/4}.
\end{equation}
We try $M=M(p,r) = \mu c_0 (p_c - p)^{-1/4} r\inv$.
The two inequalities of \eqref{eq:proof1.11} then become
\begin{equation}
    r^4 (p_c - p) \le   \mu^4 c_0^4,
    \qquad r^4 (p_c - p) \le 2^4 (1-\mu)^4 c_0^4 .
\end{equation}
The choice $\mu=2/3$ makes the two right-hand sides equal, and
in that case both inequalities assert
that $r^4 (p_c - p) \le \frac {16}{81} c_0^4$.
Thus, if we choose $c_3=\frac {16}{81} c_0^4$ and require $p \in [p_c - c_3 r^{-4}, p_c )$, then
by using \eqref{eq:Glb} in the definition of $\psi_p(x)$
we obtain
\begin{align}
\psi_p(x)
\gtrsim \sum_{1 \le \norm u_\infty \le M} \frac 1 {\norm {x+ru}_\infty^{d-2}}
\ge \sum_{1 \le \norm u_\infty \le M} \frac 1 { ( \frac 3 2 r \norm {u}_\infty )^{d-2}}
\asymp \frac {M^2} {r^{d-2}}.
\end{align}
By the definition of $M$ and the fact that
$\mp^2 \asymp (p_c- p)^{1/2}$ (by Theorem~\ref{thm:mass}), the right-hand
side is bounded from below by a multiple of $\mp^{-2} r^{-d} $.  This completes the proof.
\end{proof}

\subsection{Upper bound on $E_p$}
\label{sec:Gammaub}

We estimate $E_p$ using the torus $\star$ convolution, defined
for functions $f,g : \T_r^d \to \C$ (periodic functions on $\Z^d$) by
\begin{equation}
\label{eq:starconv}
    (f\star g)(x) = \sum_{y\in \T_r^d} f(y)g(x-y)
    \qquad (x \in \T_r^d).
\end{equation}
We also use the $\Zd$ convolution $f*g$ from \eqref{eq:Zdconv},
so it is important to mark the difference between $\star$ and $*$.
Let
\begin{equation}
    \Gamma_p(x) = G_p(x) + \psi_p(x) = \sum_{u\in \Z^d} G_p(x+ru).
\end{equation}
We write $\Gamma_p^{\star k}$ for the $k$-fold star convolution of $\Gamma_p$ with itself, and similarly for $G_p^{*k}$.
The following lemma provides a bound on $\Gamma_p^{\star k}$.  In its statement
$\delta_L$ is the constant of Theorem~\ref{thm:near_critical}.

\begin{lemma} \label{lem:integrals}
Let $d>8$ and $L \ge L_0$.
There is a constant $C>0$ such that for all $r \ge 2L+1$, $p \in [p_c - \delta_L, p_c)$, and $x\in \T_r^d$,
\begin{equation}
\Gamma_p^{\star k}(x) \le G_p^{*k}(x) + \frac C{ \mp^{2k} r^d }
	\qquad (k=1,2,3,4).
\end{equation}
\end{lemma}

\begin{proof}
The proof is based on the observation that
\begin{align}
\label{eq:Gammastarbd}
\Gamma_p^{\star k}(x) = \sum_{u\in \Zd} G_p^{*k}(x+ru) ,
\end{align}
and Theorem~\ref{thm:near_critical} is used to bound the terms with nonzero $u$.
The details for $k=1,2,3$ (both for the identity \eqref{eq:Gammastarbd}
and its upper bound) are given in \cite[Lemma~3.6]{MS23}.
The same method applies for $k=4$.
\end{proof}

\begin{figure}
\begin{center}
	\includegraphics[scale = 0.7]{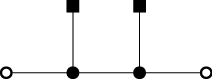}
	\qquad\qquad
	\includegraphics[scale = 0.7]{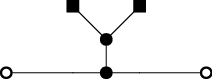}
\caption{$\Z^d$ configurations contributing to $E_{p}(x)$.
Lines represent $G_p$, hollow vertices are $0,x'$,
box vertices are $y,y'$,
and filled vertices are summed over $\mathbb Z^d$.
The vertices $y,y'$ are distinct and torus equivalent.
}
\label{figure:Zd_diagrams}
\end{center}
\end{figure}

\begin{figure}
\begin{center}
	\includegraphics[scale = 0.7]{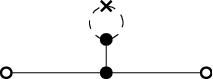}
	\qquad
	\includegraphics[scale = 0.7]{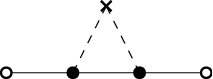}
	\qquad
	\includegraphics[scale = 0.7]{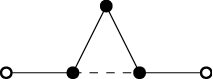}
\caption{Torus configurations contributing to $E_p(x)$.
Lines represent $\Gamma_p$, hollow vertices are $0,x$,
and filled vertices are summed over the torus.
The dashed lines with $\times$ represent
$\Gamma_p^{\star 2}-G_p^{*2}\lesssim \mp^{-4}r^{-d}$,
and the last dashed line represents
$\Gamma_p - G_{p} \lesssim \mp^{-2}r^{-d}$.
}
\label{figure:torus_diagrams}
\end{center}
\end{figure}

The next lemma bounds $E_p$ in terms of $\Gamma_p$.

\begin{lemma} \label{lem:lower}
Let $d>8$ and $L \ge L_0$.
There is a constant such that for all $r \ge 2L+1$, $p \in [p_c - \delta_L, p_c)$, and $x\in \T_r^d$,
\begin{equation}
E_p(x) \lesssim \frac 1 {\mp^6 r^d} \Gamma^{\star 2}( x )
	+ \frac 1 {\mp^4 r^d} \Gamma^{\star 3}( x )
	+ \frac 1 {\mp^2 r^d} \Gamma^{\star 4}( x)  .
\end{equation}
\end{lemma}

\begin{proof}
The function $E_p(x)$ is defined in \eqref{eq:Epdef}
to be the sum over polymers that contain $0$, a point $x' \equiv x$, and a pair of distinct equivalent points $y,y'$.
This can happen with two topologies, as depicted in Figure~\ref{figure:Zd_diagrams}.

For an upper bound, we ignore the interaction between different lines in the diagrams,
by using the lattice animal version of the BK inequality \cite[Lemma~2.1]{HS90b}.
This produces $E_p(x) \le E_p \supk 1(x) + E_p \supk 2(x)$, where
\begin{align}
E_p \supk 1(x)
&= \sum_{x' \equiv x} \sum_{s,t\in\Z^d}
	G_p(s) G_p(t) G_p(x'-t - s)
	\sum_{y\in\Z^d} \sum_{u \neq 0} G_p(y - s) G_p(y - ru - t - s) , \\
E_p \supk 2(x)
&= \sum_{x' \equiv x} \sum_{s\in\Z^d} G_p(s) G_p(x'-s)
	\sum_{t\in\Z^d} G_p(t-s)
	\sum_{y\in\Z^d} \sum_{u \neq 0} G_p(y-t) G_p(y-ru-t).
\end{align}
We estimate each term.
For simplicity, we omit the subscript $p$ and write $m = \mp$ below.
We also write $\check t = \pi(t)$, $\check s = \pi(s)$, and $\check x = x$ for the projections onto the torus.

For $E \supk 1$, we first  rewrite the sums over $y,u$ as
\begin{align}
\sum_{y\in\Z^d} \sum_{u \neq 0} G(y - s) G(y - ru - s - t)
= \sum_{u\ne 0} (G*G) (t +ru) ,
\end{align}
and observe that,
by considering whether or not $t = \check t$
and using Lemma~\ref{lem:integrals} (and \eqref{eq:Gammastarbd}),
we have
\begin{align}
G(t) \sum_{u\ne 0} (G*G)(t +ru)
\lesssim \1_{ t = \check t } G( t ) \frac 1 { m^4 r^d }
	+ \1_{ t \ne \check t } G(t) (\Gamma \star \Gamma)(t) .
\end{align}
The sum over $x'$ gives a factor $\Gamma(\check x - \check t - \check s)$.
Altogether, we obtain the upper bound
\begin{align}
E \supk 1(x)
&\lesssim \sum_{s,t\in\Z^d}
	G(s) \Gamma(\check x - \check t - \check s)
	\Big( G( t ) \frac 1 { m^4 r^d }
	+ \1_{ t \ne \check t } G(t) (\Gamma \star \Gamma)(t)  \Big) \nl
&= \frac 1 {m^4 r^d} \Gamma^{\star 3}( \check x )
	+ \sum_{s\in \Zd} \sum_{\check t \in \T_r^d} G(s)
	\Gamma(\check x - \check t - \check s) (\Gamma \star \Gamma)(\check t)
	\sum_{w\ne 0} G(\check t + rw) 	\nl
&\lesssim \frac 1 {m^4 r^d} \Gamma^{\star 3}( \check x )
	+ \frac 1 {m^2 r^d} \Gamma^{\star 4}(\check x)  ,
\end{align}
where we have again used Lemma~\ref{lem:integrals} to bound the sum over $w \ne 0$ in the last inequality.

For $E \supk 2$, we use the same strategy.
The sums over $y,u$ equal
$\sum_{u\ne 0} (G*G)(ru) \lesssim m^{-4}r^{-d}$
by Lemma~\ref{lem:integrals},
and the sum over $x'$ gives $\Gamma( \check x - \check s)$.
We obtain
\begin{align}
E_p \supk 2(x)
&\lesssim \frac 1 {m^4 r^d} \sum_{s\in\Z^d} G(s) \Gamma(\check x-\check s)
	\sum_{t\in\Z^d} G(t-s)
= \frac \chi {m^4 r^d} ( \Gamma \star \Gamma )(\check x) .
\end{align}
This completes the proof since
$\chi \asymp m^{-2}$ by \eqref{eq:gamma_half} and Theorem~\ref{thm:mass}.

The three contributions to the upper bound on $E_p(x)$ are depicted in
Figure~\ref{figure:torus_diagrams}.
\end{proof}

\subsection{Proof of Theorem~\ref{thm:torus}}
\label{sec:pf-torus}

We first prove the plateau upper bound, and then turn to the more substantial lower bound.

\begin{proof}[Proof of \eqref{eq:GTub}]
This follows by combining the upper bound $G_p^{\T}(x)\le G_p(x)+\psi_p(x)$ from
Proposition~\ref{prop:GGam} with the upper bound $\psi_p(x) \lesssim   \mp^{-2} r^{-d}$ from the $k=1$ case of Lemma~\ref{lem:integrals}.
\end{proof}

\begin{proof}[Proof of \eqref{eq:GTlb}]
Let $p \in [p_c - \delta_L, p_c)$ with $p \ge p_c - c_3 r^{-4}$,
where $c_3$ is the constant of Lemma~\ref{lem:psi_lb}.
By the lower bounds of Proposition~\ref{prop:GGam} and Lemma~\ref{lem:psi_lb},
there is a $c_\psi>0$ such that
\begin{equation} \label{eq:lb_pf}
G_p \supT (x) \ge G_p(x) + \frac {c_\psi} {\mp^2 r^d} - E_p(x) ,
\end{equation}
so it suffices to prove that
\begin{equation} \label{eq:lb_claim}
E_p(x) \lesssim \frac { 1 } { \mp^{8} r^d }
	\bigg( G_p(x) + \frac 1 { \mp^{2} r^d } \bigg) .
\end{equation}

Since $\norm{ x }_\infty \le \frac r 2$ for all $x\in \T_r^d$,
by choosing a smaller constant $c_3$, we can assume
the lower bound of Lemma~\ref{lem:Glb} holds throughout $\T_r^d$.
Using Theorem~\ref{thm:mass}, we can also assume $\mp \le ( r \sqrt d)\inv$.
Then, by Lemma~\ref{lem:lower}, Lemma~\ref{lem:integrals},
and by the convolution estimate \eqref{eq:prop1.7},
we have
\begin{align}
E_p(x) \lesssim \sum_{k=2}^4 \frac 1 { \mp^{2(5-k)}r^d } \Gamma^{\star k} (x)
\lesssim \sum_{k=2}^4 \frac 1 { \mp^{2(5-k)}r^d }
	\bigg( \frac 1 {\nnnorm x^{d-2k} } + \frac 1 { \mp^{2k} r^d } 	\bigg) .
\end{align}
Now, since $\abs x \le r \sqrt d \le \mp\inv$, we obtain
\begin{align}
E_p(x) \lesssim \sum_{k=2}^4 \frac { 1 } { \mp^{2(5-k)} r^d }
	\bigg( \frac {\mp^{-2(k-1)}} {\nnnorm x^{d-2} }+\frac1{ \mp^{2k} r^d }\bigg)
=  \frac { 3 } { \mp^{8} r^d }
	\bigg( \frac 1 {\nnnorm x^{d-2} } + \frac 1 { \mp^{2} r^d } \bigg) ,
\end{align}
and the desired \eqref{eq:lb_claim}
follows from Lemma~\ref{lem:Glb}.
\end{proof}

\section{The tilted susceptibility}
\label{sec:chim}

In this section, we present useful preliminaries concerning the
tilted susceptibility $\chim$ defined in \eqref{eq:chimdef}.
These results follow from geometric considerations and apply in all dimensions.
They will be used in our proof that $\lim_{p\to p_c^-}m(p)=0$, which is
implicit in the statement of Theorem~\ref{thm:mass}.

\subsection{Divergence of the tilted susceptibility}

Recall that $\Omega = (2L+1)^d - 1$. Let
\begin{equation} \label{eq:def_D}
D(x) = \frac 1 \Omega \1\{ 0 < \norm x_\infty \le L \}.
\end{equation}
The following is an inequality in the spirit of the Aizenman--Simon inequality \cite{AS80}.

\begin{lemma} \label{lem:AS}
Let $d\ge1$.
Let $\Lambda \subset \Zd$ be such that $0 \in \Lambda$ and $x \not \in \Lambda$. Then
\begin{align}
\label{eq:AS}
G_p(x) \le \sum_{\substack{y\in \Lambda \\ z \not \in \Lambda}} G_p(y) pD(z-y) G_p(x-z) .
\end{align}
\end{lemma}

\begin{proof}
In a lattice tree or lattice animal containing $0$ and $x$,
fix a self-avoiding path connecting $0$ and $x$.
Since $0\in \Lambda$ and $x \not \in \Lambda$, there is a first bond $\{ y, z \}$ in this path such that $y \in \Lambda$ and $z \not \in \Lambda$.
For a lattice tree,
there is only one such path, and $y,z$ divide the tree into three components:
a subtree containing $0,y$, the bond $\{ y,z \}$, and a subtree containing $z,x$.
Ignoring the mutual avoidance between the three components then produces the desired inequality.
For a lattice animal,
neither the path nor the division is unique in general,
but the lattice animal version of the BK
inequality \cite[Lemma~2.1]{HS90b} still yields the desired inequality.
\end{proof}

\begin{corollary}
\label{cor:mtilt}
For $p \in (0,p_c)$, $\chi \supmp(p) = \infty$.
\end{corollary}

\begin{proof}
This is proved in \cite[Corollary~A.4(c)]{Hara90}
to follow from
Lemma~\ref{lem:AS}.
The essential idea is that if $\chi \supmp (p) < \infty$,
then Lemma~\ref{lem:AS} implies that $G_p(x) e^{\mp x_1}$ still has exponential decay,
which in turn implies that $\chi\supk{\mp + \eps}(p) < \infty$ for some small $\eps > 0$.  This contradicts the definition of $\mp$.
\end{proof}

\subsection{Lower bound on the tilted susceptibility}

The next lemma is a lower bound on $\chi \supm$,
inspired by the method of \cite{TH87} for the correlation length of order $2$.
It expresses the fact that, due to the exponential weight,
the tilted susceptibility
diverges much faster than the untilted susceptibility.

\begin{lemma} \label{lem:chim}
Let $d\ge 1$, $m>0$, and $s > 1$.
Then there is a constant $c_{m,s} > 0$ for which
\begin{align}
\chi \supm (p) \ge c_{m,s} \chip ^ s
	\qquad ( p < p_c).
\end{align}
\end{lemma}

\begin{proof}
Let $\TA$ denote either tree or animal.
We begin by rewriting $\chi \supm$
using $\Zd$-symmetry and the definition of $G_p$.
We also decompose by the number of edges in $\TA$, to get
\begin{align}
\chi \supm (p)
&= \frac 1 d \sum_{x\in \Zd} G_p(x) \sum_{j=1}^d \cosh mx_j
\nl &
= \frac 1 d \sum_{x\in \Zd}    \sum_{\TA \ni 0,x} \Big(\frac p \Omega \Big)^{\abs \TA}
	\sum_{j=1}^d \cosh mx_j
= \frac 1 d \sum_{n=0}^\infty \Big(\frac p \Omega \Big)^n    \sum_{ \substack{ \abs \TA = n \\ \TA \ni 0 } }
	\bigg( \sum_{x \in \TA} \sum_{j=1}^d \cosh mx_j  \bigg).
\end{align}
An $n$-bond lattice tree contains $n+1$ vertices, and an $n$-bond lattice animal
contains at least order  $n+1$
vertices.
In either case, there must be a vertex
$x\in \TA$ with $\norm x_\infty \gtrsim n^{1/d}$.
Therefore, for some $c>0$ the double sum in parentheses obeys
\begin{align}
\sum_{x \in \TA} \sum_{j=1}^d \cosh mx_j
\ge \cosh ( c m n^{1/d} ) .
\end{align}
With $t_n$ the number of $n$-bond lattice trees or animals modulo translation,
this implies that
\begin{align} \label{eq:chim_pf}
\chi \supm (p)
\ge \frac 1 d \sum_{n=0}^\infty \Big(\frac p \Omega \Big)^n  \cosh ( c m n^{1/d} )
	\sum_{ \substack{ \abs \TA = n \\ \TA \ni 0 } } 1
\gtrsim   \sum_{n=0}^\infty \Big(\frac p \Omega \Big)^n \cosh ( c m n^{1/d} ) (n+1) t_n
.
\end{align}

Given $s > 1$, there is an $r \in (1,\infty)$ such that $r \inv + s \inv = 1$.
We recall \eqref{eq:chi-tn} and apply H\"older's inequality with
\begin{equation}
\begin{aligned}
f(n) &= \Big[ \Big(\frac p \Omega \Big)^n \cosh(cmn^{1/d}) (n+1) t_n \Big]^{1/s},
\\
g(n) &= \Big[ \Big(\frac p \Omega \Big)^n [\cosh(cmn^{1/d})]^{1-r} (n+1)^{1+r} t_n \Big]^{1/r},
\end{aligned}
\end{equation}
to see that the untilted susceptibility obeys
\begin{align} \label{eq:chim_pf2}
\chip \le  \sum_{n=0}^\infty \Big(\frac p \Omega \Big)^n (n+1)^2 t_n
= \sum_{n=0}^\infty f(n)g(n)
\le \norm f_s \norm g_r .
\end{align}
By \eqref{eq:chim_pf}, the norm of $f$ satisfies
\begin{align}
\norm f_s ^s
= \sum_{n=0}^\infty \Big(\frac p \Omega \Big)^n \cosh(cmn^{1/d}) (n+1) t_n
\lesssim \chi \supm (p) .
\end{align}
Since $t_n \le (\Omega/ p_c)^{n}$ (as noted near \eqref{eq:tn_subadd}),
and since $m>0$ and $r > 1$,
the norm of $g$ obeys
\begin{align}
\norm g _r ^ r
= \sum_{n=0}^\infty \Big(\frac p \Omega \Big)^n [\cosh(cmn^{1/d})]^{1-r} (n+1)^{1+r} t_n
\le \sum_{n=0}^\infty \frac{ (n+1)^{1+r} }{ [ \cosh(cmn^{1/d}) ]^{r-1} }
< \infty.
\end{align}
With \eqref{eq:chim_pf2}, this gives
$\chip \lesssim \chim (p) ^{1/s}$,
and by raising to the $s$th power we obtain the desired result.
\end{proof}

\subsection{A tilted differential inequality}

We next present a tilted version of a standard differential inequality for $\chi$.
It involves tilted square diagrams, defined by
\begin{align}
\label{eq:squaremax}
\square_p \supm
= \max \Bigl \{ ( G_p \supm * G_p \supm * G_p  * G_p) (0) - \gp^4 ,\,
	( G_p \supm * G_p * G_p  * G_p) (0) - \gp^4 	\Bigr \} .
\end{align}

\begin{lemma} \label{lem:chim_diff}
Let $d\ge 1$ and $m\ge 0$.
If $\chim (p) < \infty$, then
\begin{align} \label{eq:chim_diff}
\bigg ( \frac 1 {\gp^2} - 3 \square_p \supm \bigg ) \chi ( \chi \supm )^2 - \chim
\le  p\frac{\mathrm d  \chi \supm }{\mathrm d p}
\le \frac \Omega 2 \chi ( \chi \supm )^2 .
\end{align}
\end{lemma}

\begin{proof}[Proof sketch]
The proof requires only minor modification of the proof of the untilted differential inequality \cite[(1.19)]{HS90b} (see also \cite{BFG86,TH87} and \cite[(7.28)]{Slad06}) in order
to accommodate the tilt.
The untilted proof is based on skeleton inequalities for the three-point function $G_p(0,x,y)$.
When the tilted diagrams are summed, the exponential weight $e^{mx_1}$
can be distributed  along a path from $0$ to $x$.
In the lower bound, we encounter a square diagram with one or two tilted two-point
functions, which we bound using the maximum in \eqref{eq:squaremax}.
\end{proof}

\section{Asymptotic behaviour of the mass: proof of Theorem~\ref{thm:mass}}
\label{sec:mass}

In this section, we prove Theorem~\ref{thm:mass}.
We begin in Section~\ref{sec:lace} by recalling some previous results
concerning the lace expansion.
In Sections~\ref{sec:pf-mass-sub}--\ref{sec:first_boot}, we prove Theorem~\ref{thm:mass}.
The proof of Lemma~\ref{lem:Pi_moments}, which involves tilted diagrammatic
estimates, is deferred to Appendix~\ref{app:diagram}.

Let $\T^d=(\R/2\pi\Z)^d$ denote the continuum torus,
which we identify with $(-\pi,\pi]^d \subset \R^d$.
We use the Fourier transform, defined for
an absolutely summable function $f:\Z^d \to \C$ by
\begin{equation}
    \hat f(k) = \sum_{x\in \Z^d} f(x) e^{ik\cdot x}
    \qquad
    ( k \in \T^d ).
\end{equation}
We assume throughout this section that $d>8$ and that $L$ is large, and we
use the $\Z^d$ convolution defined in \eqref{eq:Zdconv}.

\subsection{Lace expansion}
\label{sec:lace}

By the lace expansion,
$G_p$ satisfies the convolution equation
\begin{align}
\label{eq:Gconv}
G_p = \gp \delta + \Pi_p + p D * ( \gp \delta + \Pi_p ) * G_p
\qquad (0 \le p \le p_c),
\end{align}
where $\delta(x) = \delta_{0,x}$ is the Kronecker delta,
$D$ is defined in \eqref{eq:def_D}
and $\Pi_p$ is given explicitly as an infinite series.
For a derivation of the lace expansion,
see \cite[Section~8.1]{Slad06}.
The validity of \eqref{eq:Gconv} at $p_c$ is established in \cite{HHS03},
and $g_p$ and $\Pi_p$ are left-continuous at $p_c$.
It is convenient to rewrite \eqref{eq:Gconv} as
\begin{align} \label{eq:FG}
F_p * G_p = h_p
\qquad \text{with} \qquad
h_p = \gp \delta + \Pi_p, \quad
F_p = \delta - p D * h_p .
\end{align}
After taking the Fourier transform, \eqref{eq:FG} leads to
\begin{align} \label{eq:Gk}
\hat G_p (k) = \frac{ \hat h_p (k) }{ \hat F_p (k) }
= \frac{ \gp + \hat \Pi_p (k) }{\hat F_p (k)} .
\end{align}

The $x$-space analysis of \cite{HHS03} gives decay estimates of the functions $G_p, \Pi_p$ and asymptotics of the critical $G\crit$ (also see \cite{LS24b}).
In particular,
for any choice of a (small) constant $\eta > 0$,
it follows from
\cite[Propositions~2.2 and 1.8(b)]{HHS03}
that there are constants $L_1 = L_1(\eta)$ and
\begin{equation} \label{eq:beta}
    \beta = \const\, L^{-1 + \eta}
\end{equation}
such that, for all $L \ge L_1$ and
all $p \in [0, p_c]$,
\begin{align} \label{eq:untilted_decay}
\frac{G_p(x)}{\gp}
	\le \delta_{0,x} + \frac{ O(\beta) }{ \nnnorm x^{d-2} }, 	\qquad
\abs{ \Pi_p(x) }
	\le O(\beta) \delta_{0,x} + \frac{ O(\beta^2) }{ \nnnorm x^{2d-6} } .
\end{align}
(In \cite{HHS03}, our $G_p, \Pi_p$ are denoted $\rho_p, \psi_p$ respectively.)
We also assume $\beta$ is small enough so that $\norm{ \hat \Pi_p }_\infty \le \norm{ \Pi_p }_1 \le O(\beta)\le \gpc$.

With $L_1$ taken larger if necessary,
we also have the Fourier space \emph{infrared bound}
\begin{align} \label{eq:FIR}
\hat F_p(k) - \hat F_p(0) \ge \KIR (L^2 \abs k^2 \wedge 1)
	\qquad (k \in \T^d,\ p \in [\tfrac 1 2 p_c, p_c]),
\end{align}
with an $L$-independent constant $\KIR >0$.
Versions of \eqref{eq:FIR} have appeared in previous work, but
for completeness we prove it here.
The proof uses the elementary random walk infrared bound
\begin{align} \label{eq:AIR}
\hat D(0) - \hat D(k)
&\gtrsim  L^2 \abs k^2 \wedge 1
\qquad (k \in \T^d),
\end{align}
which is proved in \cite[Appendix~A]{HS02}.  Of course, $\hat D(0)=1$.

\begin{proof}[Proof of \eqref{eq:FIR}]
By definition, $\hat F_p = 1 - p \hat D \hat h_p$ and $\hat h_p = \gp + \hat \Pi_p$, so
\begin{align}
\label{eq:IRpf0}
\hat F_p (k) - \hat F_p(0)
= p \gp [ \hat D(0) - \hat D(k) ]
+ p [ \hat D(0) \hat \Pi_p(0) - \hat D(k) \hat \Pi_p(k) ] .
\end{align}
Since $p \ge \half p_c$
is bounded away from zero,
and since $\gp \ge 1$, the first term on the right-hand side is
bounded below by a multiple of $ L^2 \abs k^2 \wedge 1$ by \eqref{eq:AIR}.
Also, by
using the second moments of $D$ and $\Pi_p$ and Taylor expansion,
the second term can be bounded by
\begin{align}
\label{eq:IRpf1}
\bigabs{ \hat D(0) \hat \Pi_p(0) - \hat D(k) \hat \Pi_p(k) }
&= \bigabs{ [ \hat D(0) - \hat D(k) ] \hat \Pi_p(0)
	+ \hat D(k) [ \hat \Pi_p(0) - \hat \Pi_p(k) ]  } \nl
&\lesssim (L^2 \abs k^2 \wedge 1) \beta
	+ 1 ( \beta \abs k^2 \wedge \beta) .
\end{align}
Since $\beta$ is small, the right-hand side of \eqref{eq:IRpf1} is small compared
to the lower bound of order $ L^2 \abs k^2 \wedge 1$ on the first term of \eqref{eq:IRpf0},
so the infrared bound \eqref{eq:FIR} holds.
\end{proof}

\subsection{Proof of Theorem~\ref{thm:mass}}
\label{sec:pf-mass-sub}

In this section, we prove Theorem~\ref{thm:mass}.
The proof uses
a new estimate on moments of
an exponentially tilted version of $\Pi_p(x)$.
We need to take the spread-out parameter $L$ to be larger.
We restrict to $L \ge L_2$ from now on,
with the constant $L_2$ produced by Lemma~\ref{lem:Pi_moments}.

For $m\ge 0$ and a function $f : \Z^d \to \R$, we define its \emph{exponential tilt} by
\begin{equation}
f\supm(x) = f(x)e^{mx_1} 	\qquad (x = (x_1, \dots, x_d)) .
\end{equation}
If $f$ is an even function then
\begin{equation}
\sum_{x\in \Zd} f \supm (x)
= \sum_{x\in \Zd} f(x) \cosh(mx_1) .
\end{equation}
Since the transition probability $D(x)$ has range $L$,
for any $a \ge 0$ and any $m \ge0$, we have
\begin{equation} \label{eq:A_moments}
\sum_{x\in \Zd} \abs x^a D^{(m)}(x)
\lesssim L^a e^{mL}.
\end{equation}

The next two propositions are proved in Section~\ref{sec:first_boot}.

\begin{proposition} \label{prop:mp}
Let $d > 8$ and $L\ge L_2$
with $L_2$ sufficiently large.
Then $\mp \searrow 0$ as $p \to p_c^-$.
\end{proposition}

\begin{proposition} \label{prop:Pi_moments}
Let $d > 8$,
$a \in [ 0, \frac{ d-4} 2 ),$\footnote{
The condition $ a < \half (d-4) = 2 + \half (d-8)$ is due to an $L^2$ norm used in the proof of Lemma~\ref{lem:Pi_moments} and is not expected
to be sharp. We expect moments to be finite for all $a < d-6 = 2 + (d-8)$.}
and $L\ge L_2$
with $L_2$ sufficiently large.
There exists $\delta_1 = \delta_1(L) > 0$,
and a constant $K_a$ independent of $L$, such that
\begin{align} \label{eq:Pi_moments_prop}
\sum_{x\in \Z^d} \abs x^a  \abs{\Pi_p \supm(x)}   \le K_a
\end{align}
uniformly in $p \in [p_c - \delta_1, p_c)$ and in $m \in [0, \mp]$.
\end{proposition}

Recall that $h_p = \gp \delta + \Pi_p$ and $F_p = \delta - p D * h_p$.
It follows from Proposition~\ref{prop:Pi_moments} with $a=0$, together with the dominated convergence theorem, that
the function $\hat \Pi_p \supm (0) = \sum_{x\in \Zd} \Pi_p \supm(x)$
is continuous in $m \in [0, \mp]$, and hence so is $\hat F_p \supm (0)$.
Tilted moments of $h_p$ and $F_p$ can also be estimated.
For example, using $\abs x^a \le 2^a ( \abs {x-y}^a + \abs y^a )$
and \eqref{eq:A_moments},
\begin{align} \label{eq:F_moments}
\sum_{x\in \Zd} \abs x^a \abs{ F_p\supm (x) }
&\le \1_{a=0} + p \sum_{x\in \Zd} \abs x^a (D \supm * h_p \supm)(x)  \nl
&\le \1_{a=0} + 2^a p_c  \(
	\bignorm{ \abs x^a D \supm(x) }_1 \norm{ h_p \supm }_1
	+ \norm{ D \supm }_1  \bignorm{ \abs x^a h_p \supm (x) }_1	\) 	\nl
&\lesssim \1_{a=0} + L^a e^{mL} ( \gp + K_0 ) + e^{mL} ( \gp \1_{a=0} + K_a ) 	\nl
&\lesssim  L^a e^{mL} ,
\end{align}
with the same restrictions on $a, L, p, m$ as in Proposition~\ref{prop:Pi_moments}.

\begin{proof}[Proof of Theorem~\ref{thm:mass}
assuming Propositions~\ref{prop:mp}--\ref{prop:Pi_moments}]
For $m < \mp$, by exponentially tilting equation \eqref{eq:Gconv}
and then taking the Fourier transform,
\begin{equation} \label{eq:Gm0}
\chi \supm(p)
= \hat G_p \supm (0)
= \frac{ \hat h_p \supm (0) }{\hat F_p\supm (0)}
= \frac{ \gp + \hat \Pi_p \supm (0) }{\hat F_p\supm (0)} .
\end{equation}
We take the limit $m \to \mp^-$ in the above equation and apply
Corollary~\ref{cor:mtilt}
to see that $\hat F_p \supmp (0) = 0$.
Then, we can write
\begin{equation} \label{eq:chi_rewrite}
\frac 1 { \chi(p) }
= \frac{ \hat F_p(0) } { \hat h_p (0) }
= \frac 1 { \hat h_p (0) } \big( \hat F_p(0) - \hat F_p \supmp (0) \big).
\end{equation}
We
use symmetry, the elementary fact that $\cosh t = 1 + \frac 12 t^2 +O(|t|^{2+\eps}\cosh t)$ (for any $\eps \in [0,2]$), and \eqref{eq:F_moments} with $a = 2 + \eps$
for some $\eps \in (0,1\wedge \frac 12 (d-8))$, to see that
\begin{align} \label{eq:m_asymp_pf}
    \hat F_p (0) - \hat F_p\supmp(0)
    &= - \sum_{x\in \Z^d} F_p(x) \bigl(\cosh (\mp x_1) - 1\bigr)
    \nnb & =
    \half \mp^2 \del_1^2  \hat F_p(0)
	+ O(\mp^{2+\eps})  \sum_{x\in \Z^d} \abs {x_1}^{2+\eps} \abs{ F_p \supmp(x) }
\nl
&= \mp^2 L^2 \bigg( \frac{ \del_1^2 \hat F_p(0) }{2 L^2}
	+ O(\mp^\eps L^\eps e^{\mp L})  \bigg) .
\end{align}
The derivative on the right-hand side obeys
 $\del_1^2 \hat F_p(0) \asymp L^2$ uniformly in $p$:
the lower bound is due to the infrared bound \eqref{eq:FIR}
and the upper bound is due to \eqref{eq:F_moments} with $a = 2$ and $m=0$.
Recall that $g_p$ and $\Pi_p$ are left-continuous at $p_c$.
By
dominated convergence, $\del_1^2 \hat F_p(0)$ is also left-continuous at $p_c$,
and therefore also $\del_1^2 \hat F_{p_c}(0) \asymp L^2$.
In particular, $\del_1^2 \hat F_{p_c}(0) >0$.
Hence, since $\mp \to 0$ as $p \to p_c^-$ by Proposition~\ref{prop:mp},
it follows from \eqref{eq:chi_rewrite} and \eqref{eq:m_asymp_pf} that
\begin{align} \label{eq:mp_asymp}
\frac 1 \chip  \sim  \frac { \del_1^2 \hat F_p(0) } { 2 \hat h_p (0) } \mp^2
\sim  \frac { \del_1^2 \hat F\crit(0) } { 2 \hat h\crit (0) } \mp^2
	\qquad \text{as $p \to p_c^-$.}
\end{align}
This gives the desired result, since we know from \eqref{eq:gamma_half}
that $\chi(p) \asymp (p_c-p)^{-1/2}$.
The asymptotic formula for lattice trees follows from the improvement of \eqref{eq:gamma_half} to an asymptotic relation in \cite{HS92c}.
\end{proof}

\subsection{Proof of Propositions~\ref{prop:mp} and \ref{prop:Pi_moments}}
\label{sec:first_boot}

The proof of Propositions~\ref{prop:mp}--\ref{prop:Pi_moments}
uses a bootstrap argument varying $m$, as was
first done for percolation in \cite{Hara90}.
From \eqref{eq:Gk},
the infrared bound \eqref{eq:FIR},
and the
fact that $\norm{ \hat \Pi_p }_\infty \le O(\beta)$,
we have
\begin{align} \label{eq:Gk_bound}
\abs{ \hat G_p (k) } = \biggabs{ \frac{ \gp + \hat \Pi_p(k) } { \hat F_p(k) } }
\le \frac { 2\gpc } \KIR \frac 1 { L^2 \abs k^2 \wedge 1 }
	\qquad (p \in [\tfrac12 p_c, p_c)) .
\end{align}
For $p \in [\tfrac12 p_c, p_c)$,
we define the \emph{bootstrap function} $b: [0, \mp) \to \R$ by
\begin{equation}
b(m) = \sup_{k \in \T^d} \bigl\{ \gpc\inv \KIR  (L^2 \abs k^2 \wedge 1)  \abs{ \hat G_p \supm (k) } \bigr\} ,
\end{equation}
so $b(0) \le 2$ by \eqref{eq:Gk_bound}.
It follows from
the finiteness of $\chim$ below $\mp$ that
the family $\{ \hat G_p \supm(k) \}_{k\in \Td}$ is equicontinuous
in $m \in [0, \mp)$, which implies
that $b(\cdot)$ is continuous \cite[Lemma~5.13]{Slad06}.
We will prove the implication
\begin{equation}
b(m ) \le 4
\; \implies \;
b(m)  \le 3
\end{equation}
when $L$ is sufficiently large and $m L$ is sufficiently small.
Since $b(\cdot)$ is continuous, this implies that $b(m) \le 3$ for all small $m$.
The precise statement is given in Proposition~\ref{prop:boot_done}.

\subsubsection{The bootstrap argument}

We first state an estimate
on moments of $\Pi_p\supm$ under the bootstrap hypothesis that $b(m) \le 4$.
In the process, we also obtain an estimate on
the tilted square diagram $\square_p \supm$ defined in \eqref{eq:squaremax}.

\begin{lemma} \label{lem:Pi_moments}
Let $d > 8$, $p \in [\tfrac12 p_c, p_c)$, and $m<m(p)$.
There is an $L_2$ such that for all $L \ge L_2$,
if $b(m) \le 4$ and $mL \le 1$ then
\begin{equation}
\label{eq:squaresmall}
\square \supm_p \lesssim L^{-d+1} \le \frac 1 {100},
\end{equation}
and, for any $a \in [ 0, \frac{ d-4} 2 )$ there is a constant $K_a$ independent of $L$ such that
\begin{align} \label{eq:Pi_moments}
\sum_{x\in \Z^d} \abs x^a  \abs{\Pi_p \supm(x)}   \le K_a .
\end{align}
Moreover, if $b(m)\le 4$
uniformly in
$m<m(p)$
and if $\mp L \le 1$, then the bounds
\eqref{eq:squaresmall}--\eqref{eq:Pi_moments} hold also at $m=m(p)$.
\end{lemma}

The proof of Lemma~\ref{lem:Pi_moments}
is built upon diagrammatic estimates for $\Pi_p$ obtained in \cite{HHS03}
and is deferred to Appendix~\ref{app:diagram}.
Using Lemma~\ref{lem:Pi_moments}, conditional on the bootstrap hypothesis
we have the following tilted infrared bound.

\begin{corollary} \label{cor:tilted_infrared}
Let $d > 8$, $p \in [\tfrac12 p_c, p_c)$, and $m<m(p)$.
There exists $\eps \in (0,1)$ and $\kappa >0$ such that for all $L \ge L_2$,
if $b(m) \le 4$ and $mL \le 1$ then
\begin{align} \label{eq:tilted_infrared}
\Re [ \hat F_p \supm (k) - \hat F_p \supm (0) ]
\ge ( \KIR - \kappa m^\eps L^\eps   ) (L^2 \abs k^2 \wedge 1) .
\end{align}
\end{corollary}

\begin{proof}
Recall that $F_p = \delta - p D * h_p$ and $h_p = \gp\delta + \Pi_p$.
As shown in \eqref{eq:F_moments},
from the estimates on $\Pi_p\supm$ supplied by Lemma~\ref{lem:Pi_moments} and using $e^{mL} \le e$,
we have
\begin{align} \label{eq:F_moments2}
\sum_{x\in \Zd} \abs x^a \abs{ F_p\supm (x) }  \lesssim  L^a
	\qquad ( a \in [0, \tfrac{d-4}2) ) .
\end{align}
Since $d > 8$, the above holds with $a = 2 + \eps$ for some $\eps \in (0,1)$.

To prove \eqref{eq:tilted_infrared}, we first
use the fact that $F_p$ is real and symmetric  to write
\begin{equation}
\Re [ \hat F_p\supm(k) - \hat F_p\supm (0) ]
=  \sum_{x\in \Z^d} ( \cos (k\cdot x) - 1) F_p(x) \cosh mx_1 .
\end{equation}
We then add and subtract
$\hat F_p(k) - \hat F_p(0) =  \sum_{x\in \Z^d} ( \cos (k\cdot x) - 1 ) F_p(x)$
to get
\begin{equation}
\label{eq:ReFm}
\Re [ \hat F_p\supm(k) - \hat F_p\supm (0) ]
= \hat F_p(k) - \hat F_p(0) + \sum_{x\in \Z^d} ( \cos (k\cdot x) - 1 ) F_p(x) (\cosh (mx_1) - 1).
\end{equation}
Due to the untilted infrared bound in \eqref{eq:FIR},
it suffices to prove that
\begin{equation} \label{eq:tilted_infrared_pf1}
\Bigabs{ \sum_{x\in \Z^d} ( \cos (k\cdot x) - 1 ) F_p(x) (\cosh (mx_1) - 1) }
\le O(m^\eps L^\eps) ( L^2 \abs k^2 \wedge 1 ) .
\end{equation}
For $\abs k \ge L\inv$,
we use
$\abs{ \cosh (mx_1) - 1 } \lesssim \abs{ mx_1}^\eps \cosh mx_1$
(for $\eps \le 2$) and \eqref{eq:F_moments2}
to bound the left-hand side of \eqref{eq:tilted_infrared_pf1} by
\begin{equation}
\sum_{x\in \Z^d} 2 \abs {F_p(x) } O( \abs{ m x_1 } ^\eps) \cosh mx_1
\lesssim m^\eps \sum_{x\in \Z^d} \abs x^{\eps} \abs {F_p\supm(x) }
\lesssim m^\eps L^\eps ,
\end{equation}
as claimed.
For $\abs k < L\inv$, we additionally use $\abs{  \cos (k\cdot x) - 1 } \lesssim \abs k^2 \abs x^2$ to bound the left-hand side of \eqref{eq:tilted_infrared_pf1} by
\begin{equation}
\sum_{x\in \Z^d} O(\abs k^2 \abs x^2) \abs {F_p(x) } O( \abs{ m x_1 } ^\eps) \cosh mx_1
\lesssim m^\eps \abs k^2 \sum_{x\in \Z^d} \abs x^{2+\eps} \abs {F_p \supm (x) }
\lesssim m^\eps \abs k^2 L^{2+\eps}  .
\end{equation}
This gives \eqref{eq:tilted_infrared_pf1} and completes the proof.
\end{proof}

The next proposition completes the bootstrap argument.

\begin{proposition} \label{prop:boot_done}
Let $d > 8$.
There exists $\eta \in (0,1]$ such that
for all $L \ge L_2$, and for all $m < \mp$ with $mL \le \eta$,
we have $b(m)  \le 3$.
\end{proposition}

\begin{proof}
We first derive an estimate on $\hat G_p \supm$
assuming $L \ge L_2$, $m < \mp$, $b(m)\le 4$, and $mL \le 1$.
To begin, we show that $\hat F_p \supm(0) > 0$ when $m$ is small.
As in \eqref{eq:Gm0},
\begin{align}
\hat F_p \supm(0)
= \frac { \hat h_p \supm (0) } {\chi\supm(p)}
= \frac { \hat h_p (0) 	- [\hat \Pi_p (0) - \hat \Pi_p \supm (0)] } {\chi\supm(p)}
	\qquad (m < \mp) .
\end{align}
The denominator is positive, and $\hat h_p(0) = \gp + \hat \Pi_p(0)$ is bounded away from zero since $\gp\ge 1$ and $|\hat \Pi_p(0)| \le O(\beta)$.
Also, the inequality $\abs{ e^{mx_1} -1 } \lesssim \abs{ mx_1 }^\eps e^{mx_1}$ for $\eps \in (0,1]$ gives
\begin{align} \label{eq:Pim_diff}
\norm{ \hat \Pi_p \supm - \hat \Pi_p }_\infty
= \sup_{k\in \Td} \Bigabs{ \sum_{x\in \Zd} \Pi_p(x) ( e^{mx_1} - 1 ) e^{ik\cdot x} }
\lesssim  \sum_{x\in \Zd} \abs{\Pi_p(x)} |mx_1|^\eps e^{mx_1}
\lesssim m^\eps K_\eps ,
\end{align}
with $K_\eps$ from Lemma~\ref{lem:Pi_moments}.
Since $\eps > 0$, this proves that $\hat F_p \supm(0) > 0$ for small $m$.

Therefore, we have
\begin{align}
\abs{ \hat F_p\supm (k) }
\ge \Re \hat F_p\supm (k)
\ge \Re [ \hat F_p\supm(k) - \hat F_p\supm (0) ]
\ge \KIR\supm (L^2 \abs k^2 \wedge 1 ),
\end{align}
where $\KIR\supm = \KIR - \kappa m^\eps L^\eps$ is from Corollary~\ref{cor:tilted_infrared}.
By taking $mL$ small, we may assume that $\KIR\supm \ge \frac 12 \KIR$.
It follows that
\begin{align}
\abs{ \hat G_p \supm (k)}
= \biggabs{ \frac{ \gp + \hat \Pi_p \supm (k) }{ \hat F_p \supm (k) } }
\le  \frac{ \norm { \gp + \hat \Pi_p }_\infty }{ \Re [ \hat F_p\supm(k) - \hat F_p\supm (0) ]  }
	+ \frac{ \norm { \hat \Pi_p \supm - \hat \Pi_p }_\infty }{ \KIR \supm (L^2 \abs k^2 \wedge 1)  } .
\end{align}
The second term on the right-hand side
is $O(m^\eps) (L^2 \abs k^2 \wedge 1) \inv$ by \eqref{eq:Pim_diff}.
For the first term,
we use $\norm{ \hat \Pi_p }_\infty \le O(\beta) \le \gpc$ to bound the numerator,
and we compare the denominator to its $m=0$ version. This gives
\begin{align}
\label{eq:big-ratio}
\frac{ \norm { \gp + \hat \Pi_p }_\infty }{ \Re [ \hat F_p\supm(k) - \hat F_p\supm (0) ]  }
&\le 2\gpc \bigg( \frac 1{ \hat F_p(k) - \hat F_p(0) }
	+ \frac{ [\hat F_p(k) - \hat F_p(0) ] - \Re [ \hat F_p\supm(k) - \hat F_p\supm (0) ] }
		{  [\hat F_p(k) - \hat F_p(0) ] \Re [ \hat F_p\supm(k) - \hat F_p\supm (0) ] } \bigg)	
.
\end{align}
The combination of \eqref{eq:ReFm} and \eqref{eq:tilted_infrared_pf1} shows
that the numerator of the last term on the right-hand side is at most
$O(m^\eps L^\eps) (L^2 \abs k^2 \wedge 1)$.  With the tilted infrared bound
of \eqref{eq:tilted_infrared}, we see that the right-hand side of
\eqref{eq:big-ratio} is bounded above by
\begin{align}
\frac {2\gpc} \KIR \frac 1 { L^2 \abs k^2 \wedge 1 }
	+ \frac{ O(m^\eps L^\eps) (L^2 \abs k^2 \wedge 1) }{ \KIR \KIR \supm (L^2 \abs k^2 \wedge1)^2 }
= \bigg( \frac {2\gpc} \KIR  + O(m^\eps L^\eps)  \bigg) \frac 1 { L^2 \abs k^2 \wedge 1 } .
\end{align}
Altogether, we obtain
\begin{equation}
\label{eq:Gm_after}
\abs{ \hat G_p \supm (k) }
\le \bigg( \frac {2\gpc} \KIR  + O(m^\eps L^\eps) + O(m^\eps)
\biggr) \frac 1 { L^2 \abs k^2 \wedge 1 },
\end{equation}
assuming $b(m)\le 4$ and $mL$ is small.

To conclude the proof,
we pick $\eta$ small so that
the coefficient on the right-hand side of \eqref{eq:Gm_after}
is less than $3\gpc / \KIR$ for all $mL \le \eta$.
This gives the implication
\begin{align}
b(m ) \le 4
\; \implies \;
b(m)  \le 3
\end{align}
for $m \in [0, \mp)$ satisfying $mL \le \eta$.
But since $b(0) \le 2$ from \eqref{eq:Gk_bound},
and since the function $b(\cdot)$ is continuous,
it must be the case that $b(m) \le 3$ for all $m \in [0, \mp)$ with $mL \le \eta$, as desired.
\end{proof}

\subsubsection{Proof of Propositions~\ref{prop:mp} and \ref{prop:Pi_moments}}

\begin{proof}[Proof of Proposition~\ref{prop:mp}]
Let $L \ge L_2$.
We prove by contradiction that $\mp \searrow 0$ as $p \to p_c^-$.
If it is not the case that $\mp \searrow 0$, then $ \mp \searrow m_0 > 0$.
Let $\mu = \half \min \{ m_0, \eta / L \} > 0$,
which satisfies $\mu < \mp$ for all $p < p_c$,
so $\chi^{(\mu)}(p) < \infty$.
Also, by Proposition~\ref{prop:boot_done}, $b(\mu) \le 3$,
so from the bound on the tilted square diagram
of Lemma~\ref{lem:Pi_moments}
(which is uniform in $p$), we obtain
\begin{align}
\label{eq:sqbdL}
\sup_{p \in [ \half p_c, p_c) } \square_p^{(\mu)}
\le \frac 1{100}.
\end{align}

We now show that the above
contradicts the
conclusions of Section~\ref{sec:chim}.
Using Lemma~\ref{lem:chim_diff} with $m = \mu$,
and inserting $\gp \le 4$ and \eqref{eq:sqbdL},
we find that
the coefficient of $\chi (\chi^{(\mu)})^2$ on the
left-hand side of the differential inequality \eqref{eq:chim_diff} remains positive as $p \to p_c^-$, so by dividing by $(\chi^{(\mu)})^2$ we get
\begin{align}
\label{eq:de-mu}
\frac { \mathrm d }{ \mathrm dp } \biggl( \frac { -1 }{ \chi^{(\mu)} } \bigg)
\gtrsim \chi - \frac{1}{\chi^{(\mu)}} .
\end{align}
Since $\chi^{(\mu)} (p) \ge 1$ and since $\chi (p) \gtrsim (p_c - p)^{-1/2}$ by \eqref{eq:chilb-general-d}, for $p$ near $p_c$ the right-hand side of
\eqref{eq:de-mu} is at least a multiple of $(p_c-p)^{-1/2}$.
Integration
from $p$ to $p_c$ therefore yields $\chi^{(\mu)}(p) \lesssim (p_c - p)^{-1/2}$ as $p \to p_c^-$.
On the other hand,
Lemma~\ref{lem:chim} with $m=\mu$ implies that, for any $s>1$,
\begin{align}
\chi^{(\mu)}(p) \gtrsim \chip^{s} \gtrsim (p_c - p)^{-s/2}
	\qquad (p<p_c).
\end{align}
Any choice of $s>1$ yields a contradiction.
Therefore, it must be the case that $\mp \searrow 0$.
\end{proof}

\begin{proof}[Proof of Proposition~\ref{prop:Pi_moments}]
Let $L \ge L_2$. Since $\mp \searrow 0$ as $p \to p_c^-$ by Proposition~\ref{prop:mp},
we can pick $\delta_1$ small so that $\mp L \le \eta$ for all $p \in [p_c - \delta_1, p_c)$.
We then obtain $b(m) \le 3$ for all $m \in [0,\mp)$ by Proposition~\ref{prop:boot_done},
and the desired moment estimate \eqref{eq:Pi_moments_prop} follows
for these values of $m$ by Lemma~\ref{lem:Pi_moments}.
The uniform estimate also implies
\eqref{eq:Pi_moments_prop} at $m = \mp$,
by the final assertion of Lemma~\ref{lem:Pi_moments}.
\end{proof}

\section{Near-critical bound: proof of Theorem~\ref{thm:near_critical}}
\label{sec:near_critical}

In this section, we prove Theorem~\ref{thm:near_critical}.
The proof uses a much simpler analogue of Theorem~\ref{thm:near_critical}
for the spread-out random walk, which is stated in Proposition~\ref{prop:so-nn}
and proved in Appendix~\ref{app:S}.
We employ a second bootstrap argument on the tilted two-point function---this
time in $x$-space---via a combination of ideas from \cite{LS24b,Liu24,Slad23_wsaw}.
The bootstrap is encapsulated in Proposition~\ref{prop:boot}.
In Section~\ref{sec:boot}, we reduce the proof of Theorem~\ref{thm:near_critical}
to Propositions~\ref{prop:so-nn}--\ref{prop:boot}.
The proof of Proposition~\ref{prop:boot} is
given in Section~\ref{sec:boot-pf} subject to Proposition~\ref{prop:f},
which itself is proved in Section~\ref{sec:f}.

\subsection{Proof of Theorem~\ref{thm:near_critical}}
\label{sec:boot}

For $d>2$ and $\mu \in [0,1]$,
let $S_\mu$ be the solution of the convolution equation
\begin{equation} \label{eq:S}
    (\delta - \mu D)*S_\mu = \delta
\end{equation}
given by
\begin{equation}
S_\mu(x)= \int_{\T^d}\frac{e^{-ik\cdot x}}{1-\mu\hat D(k)} \frac{dk}{(2\pi)^d}
\qquad (x\in\Z^d).
\end{equation}
As in the definition of the mass $m(p)$ in \eqref{eq:mtilt}, we define
\begin{align}
\label{eq:mSdef}
    m_S(\mu)
    =
    \sup \Bigl \{ m \ge 0 :
	 \sum_{x\in\Z^d} S_\mu(x) e^{mx_1} < \infty 	\Big \} .
\end{align}
The mass $m_S(\mu)$ coincides with the rate $\tilde m_S(\mu)$
of exponential decay of $S_\mu$,
defined as in \eqref{eq:mdef}.\footnote
{The random walk two-point function satisfies
$S_\mu(y) \ge \frac 1 {S_\mu(0)} S_\mu(x) S_\mu(y-x)$ for all $x,y$,
which implies the off-axis control
$S_\mu(x) \le S_\mu(0) \exp\{ - \tilde m_S(\mu) \norm x_\infty \}$.
This is exactly as in the proof of \cite[Theorem~A.2(a)]{MS93},
which applies equally to the spread-out case.
It follows that $\sum_{x\in\Z^d} S_\mu(x) e^{mx_1} < \infty$
for all $m < \tilde m_S(\mu)$,
so $m_S(\mu) \ge \tilde m_S(\mu)$.
The reverse inequality holds in general.}
In particular, if $\mu < 1$ then $m_S(\mu)>0$.
A minor modification
of the proof of Corollary~\ref{cor:mtilt} shows that
\begin{equation}
\label{eq:chiSinf}
    \sum_{x\in\Z^d} S_\mu^{(m_S(\mu))}(x)   =\infty.
\end{equation}
Let $\sigma_L^2=\sum_{x\in\Z^d}|x|^2D(x)$ denote the variance of $D$.

\begin{proposition} \label{prop:so-nn}
Let $d>2$,
$\eta >0$,
$L \ge L_2$ (with $L_2$ sufficiently large),
and $\mu\in [\half,1)$.
Then
\begin{align}
m_S(\mu)^2 \sim \frac{2d}{\sigma_L^2} (1-\mu)
	\qquad \text{as $\mu\to 1$,}
\end{align}
and there are constants $K_S > 0$, $a_S \in (0,1)$ (independent of $L$)
and $\delta_2 = \delta_2(L) > 0$, such that
\begin{equation}
S_\mu(x) \le \delta_{0,x} +
	\frac {K_S}{ L^{1-\eta} \nnnorm x^{d-2} } e^{- a_S m_S(\mu) \norm x_\infty }
	\qquad (x\in \Zd)
\end{equation}
uniformly in $\mu \in [1 - \delta_2, 1)$.
\end{proposition}

The proof of Proposition~\ref{prop:so-nn} is presented in Appendix~\ref{app:S}.
We apply Proposition~\ref{prop:so-nn} with the same $\eta$
as in $\beta = \const\, L^{-1+\eta}$
from \eqref{eq:beta}.
In view of \eqref{eq:untilted_decay}, by taking a larger constant $K_S$ we can also assume
\begin{equation} \label{eq:untilted_G}
\frac{G_p(x)}{\gp}
\le \delta_{0,x} + \frac{ K_S }{ L^{1-\eta} \nnnorm x^{d-2} }
	\qquad ( p \in [0, p_c]).
\end{equation}
We also assume
that $L \ge L_2$ and $\delta_2 \le \delta_1$ from Proposition~\ref{prop:Pi_moments},
so we can apply all results from Section~\ref{sec:mass}.
For $p \in [p_c - \delta_2, p_c)$,
we define another \emph{bootstrap function} $\btil :[0, \mp) \to \R$ by
\begin{align}
\btil(m) = \sup_{x\ne 0} \Bigl\{
	\frac{ G_p(x) e^{mx_1} }{ \gp K_S L^{-1+\eta} \abs x^{-(d-2)}}  \Big\} .
\end{align}
The function $\btil$ is continuous in $m$ because
$G_p(x) e^{mx_1}$ decays exponentially when
$m < \mp$ by \eqref{eq:G_exp_decay}.
By \eqref{eq:untilted_G}, we have $\btil(0) \le 1$.

\begin{proposition} \label{prop:boot}
Let $d > 8$.
There are constants $L_3 \ge L_2$, $\sigma \in (0,1)$,
and $\delta_3 = \delta_3(L) > 0$ such that,
when $L \ge L_3$, $p \in [p_c - \delta_3, p_c)$, and $m \in [0, \sigma \mp]$,
we have
\begin{align}
\btil(m) \le 3
\; \implies \;
\btil(m) \le 2 .
\end{align}
\end{proposition}

\begin{proof}[Proof of Theorem~\ref{thm:near_critical}]
Let $\delta_L = \delta_3(L)$.
For $p \in [p_c - \delta_L, p_c)$,
since the function $\btil$ is continuous,
it follows from Proposition~\ref{prop:boot} and $\btil(0) \le 1$
that $\btil(\sigma \mp) \le 2$, which is equivalent to
\begin{align}
G_p(x)
\le  \frac{ 2 \gp K_S }{ L^{1-\eta} \abs x^{d-2} } e^{-\sigma \mp x_1}
	\qquad (x\ne 0) .
\end{align}
Since $G_p$ is $\Zd$-symmetric,
we can assume $x_1 = \norm x_\infty$.
The desired bound then follows from $\norm x_\infty \ge d^{-1/2} \abs x$,
by taking the constants in \eqref{eq:near_critical} to be $c = d^{-1/2} \sigma$ and $C_L = 2 \gpc K_S L^{-1 + \eta}$.
\end{proof}

\subsection{Proof of Proposition~\ref{prop:boot}}
\label{sec:boot-pf}

To prove Proposition~\ref{prop:boot},
we adapt the deconvolution argument of \cite{LS24b}
using ideas from \cite{Liu24, Slad23_wsaw}.

Recall $\beta = \const\, L^{-1+\eta}$.
Under the bootstrap hypothesis $\btil(m) \le 3$,
we have $G_p \supm (x) \lesssim \beta \gp \abs x^{-(d-2)}$ for all $x \ne 0$.
Unlike in Section~\ref{sec:mass},
this bootstrap hypothesis allows the use of convolution estimates
to bound $\Pi_p \supm(x)$, exactly as in the $m=0$ case in \cite{HHS03}.
For example,
we bound the tilt of $p D * G_p$ by
\begin{align}
p (D\supm * G_p \supm) (x)
\le e^{mL} p (D *  G_p \supm) (x)
\le e^{mL} \frac{ C \beta }{ \nnnorm x^{d-2} } ,
\end{align}
as in \cite[(4.45)]{HHS03}.
Assuming $mL \le 1$ so that we can use $e^{mL}\le e$,
the bootstrap assumption $\btil(m) \le 3$ implies, as in the proof of
\cite[(4.43)]{HHS03} (with $q=d-2$ there), that
\begin{align} \label{eq:Pim_decay}
\abs{ \Pi_p \supm (x) } \le O(\beta) \delta_{0,x} + \frac{ O(\beta^2) } { \nnnorm x^{2d-6} }
\end{align}
for all $L \ge L_3$ with $L_3$ sufficiently large.
The decay of $\Pi_p \supm$ in \eqref{eq:Pim_decay} gives estimates on derivatives of $\hat \Pi_p \supm$ through the Fourier transform, via \cite[Lemma~2.6]{LS24a}, namely
\begin{align} \label{eq:Pi_gamma}
\bignorm{ \grad^\gamma \hat \Pi_p \supm }_q
\lesssim  \beta \1_{\gamma = 0} + \beta^2
	\qquad \Big(\frac{ \abs \gamma - (d - 6) } d < q \inv \le 1 \Big)
\end{align}
for all multi-indices $\gamma$ with $\abs \gamma \le d-2$.
Here and in what follows, the Fourier norms are $L^q(\T^d)$ norms on the continuum
torus $\T^d$ of period $2\pi$, with
respect to the normalised Lebesgue measure $(2\pi)^{-d} d k$.

To prove that $\btil(m) \le 2$, we analyse $G_p(x)$ in detail.
Recall from \eqref{eq:Gk} the identity
\begin{align} \label{eq:Gk2}
\hat G_p (k) = \frac{ \hat h_p (k) }{\hat F_p (k)} .
\end{align}
We decompose the right-hand side of \eqref{eq:Gk2} into a random walk term and a remainder.
The parameter $\mu$ of the random walk is chosen carefully to make the remainder small.
This isolation of leading term idea originated in \cite{HHS03} and has subsequently
been used repeatedly \cite{Slad23_wsaw, LS24a, LS24b, Liu24}, but our decomposition is not exactly the same.

For $\mu \in [0,1]$,
we define
\begin{equation}
\label{eq:Adef}
    A_\mu = \delta - \mu D,
\end{equation}
so that $\hat A_\mu = 1 - \mu \hat D$.
For any $\lambda\in \R$, we can write
\begin{align} \label{eq:G_decomp}
\hat G_p
= \frac{ \hat h_p }{\hat F_p}
= \lambda \frac 1 {\hat A_\mu} + \frac{ \hat h_p \hat A_\mu - \lambda \hat F_p }{ \hat A_\mu \hat F_p} .
\end{align}
Let
\begin{equation}
\label{eq:Edef}
    E_{p,\lambda,\mu} = h_p * A_\mu - \lambda F_p.
\end{equation}
We choose the parameters $\lambda, \mu$ to ensure that
\begin{align} \label{eq:E_condition}
\sum_{x\in \Z^d} E_{p,\lambda,\mu}(x) = \sum_{x\in \Z^d} \abs x^2 E_{p,\lambda,\mu}(x) = 0.	
\end{align}
This is a system of two linear equations in $\lambda,\mu$, with solution
\begin{align} \label{eq:lambda-mu}
\lambda_p
= \frac {  \hat h_p(0) } { \hat F_p(0) - \sigma_L^{-2}\sum_{x\in\Z^d} \abs x^2 \big[ F_p(x) - \frac{\hat F_p(0)}{ \hat h_p(0) } \Pi_p(x) \big] } ,	
\qquad
\mu_p
= 1 - \frac{ \lambda_p } \chip .
\end{align}
Some properties of $\lambda_p, \mup$ are stated in the following lemma.
We write $E_p = E_{p,\lambda_p, \mu_p}$ henceforth.

\begin{lemma} \label{lem:m_S}
Let $L \ge L_3$.
Then $\lambda_p = \gp + O(\beta)$, and as $p\to p_c^-$ we have $\mup \to 1$ and
\begin{align}
\mp \sim m_S(\mup) .
\end{align}
\end{lemma}

\begin{proof}
Since $F_p = \delta - pD * h_p $
with $h_p = \gp\delta + \Pi_p$, and since $\Pi_p$ is of order $\beta$
by \eqref{eq:untilted_decay}, we have
\begin{align}
\hat F_p(0) &= 1 - p\gp - p \hat \Pi_p(0) = 1 - p\gp + O(\beta), \\
\sum_{x\in \Zd} \abs x^2 F_p(x)
	&= -p \gp \sigma_L^2 - p\sum_{x\in \Zd} \abs x^2 (D*\Pi_p)(x)
	= - p \gp \sigma_L^2 + O(\sigma_L^2 \beta) .
\end{align}
We insert this into the definition of $\lambda_p$ and obtain
\begin{align} \label{eq:lambda_pf}
\lambda_p
&= \frac{ \gp + O(\beta) } { 1 - p\gp + O(\beta) + p\gp + O(\beta) }
= \gp + O(\beta) .
\end{align}
Since $\chip$ diverges as $p\to p_c^-$, whereas $\lambda_p$ remains bounded,
we also have  $\mup   \to 1$.

In conjunction with $\mup\to 1$,  Proposition~\ref{prop:so-nn} gives
\begin{align}
m_S(\mup)^2
\sim \frac{2d}{\sigma_L^2} (1 - \mup)
\sim \frac{2d}{\sigma_L^2} \frac {\lambda\crit} \chip .
\end{align}
Since $\hat F\crit (0) = 0$, \eqref{eq:lambda-mu} and the $\Zd$-symmetry give
\begin{align}
\lambda\crit
= \frac { \hat h\crit(0) } { - \sigma_L^{-2}  \sum_{x\in \Zd} \abs x^2 F\crit(x)  }
= \frac { \sigma_L^2} d \frac{ \hat h\crit(0) }{ \del_1^2 \hat F\crit(0) } .
\end{align}
It then follows from the asymptotic formula \eqref{eq:mp_asymp}  for $\mp$ that
\begin{align}
m_S(\mup)^2
\sim \frac{ 2 \hat h\crit(0) }{ \del_1^2 \hat F\crit(0) } \frac 1 \chip
\sim \mp^2 ,
\end{align}
and the proof is complete.
\end{proof}

With $\lambda = \lambda_p$ and $\mu = \mup$,
the inverse Fourier transform of \eqref{eq:G_decomp} yields
\begin{align} \label{eq:G_isolate}
G_p = \lambda_p S_\mup + f_p ,
\end{align}
where $f_p$ is the inverse Fourier transform of $\hat f_p = \hat E_p / (\hat A_\mup \hat F_p)$, namely
\begin{align}
f_p = S_\mup * E_p * H_p,
\qquad H_p(x) = \int_{\T^d}\frac{e^{-ik\cdot x}}{\hat F_p(k)} \frac{dk}{(2\pi)^d} .
\end{align}
Multiplication of \eqref{eq:G_isolate} by $e^{mx_1}$ gives the tilted equation
\begin{align} \label{eq:Gm_decomp}
G_p \supm = \lambda_p S_\mup \supm + f_p \supm,
\qquad
f_p \supm =  S_\mup \supm * E_p \supm * H_p \supm.
\end{align}
The next proposition provides the key estimate on the remainder $f_p \supm$, under the bootstrap hypothesis $\btil(m) \le 3$.
The hypotheses that $\mp L$ and $m_S(\mup) L$ are small will be established
by restricting to $p \in [p_c - \delta_3, p_c)$ with $\delta_3$ small,
since both $\mp, m_S(\mup) \to 0$ as $p \to p_c^-$ by Theorem~\ref{thm:mass} and Lemma~\ref{lem:m_S}.
The proof of Proposition~\ref{prop:f} is given in Section~\ref{sec:f}.

\begin{proposition} \label{prop:f}
Let $d>8$ and $L \ge L_3$.
Suppose that $m \in [0, \half \mp \wedge \half m_S(\mup)]$,
that $\btil(m) \le 3$,
and that $\mp L$ and $m_S(\mup)L$ are sufficiently small.
Then
$\hat f_p \supm$
is $d-2$ times weakly differentiable,
and there exists $c > 0$ such that
\begin{align}
\bignorm { \grad^\alpha \hat f_p \supm }_{1}
\lesssim \beta (L^{-c} + \beta)
	\qquad (\abs \alpha = d-2) ,
\end{align}
with the constants independent of $p,m,L$.
\end{proposition}

See \cite[Appendix~A]{LS24a} for an introduction to the weak derivative,
which extends the classical notion of derivative to (locally) integrable functions, using the usual integration by parts formula.
We combine Proposition~\ref{prop:f} with the classical fact that smoothness in the Fourier space corresponds to decay in the physical space.
A precise version of this principle is stated next.

\begin{lemma} \label{lem:fourier}
Let $n,d>0$ be positive integers.
Suppose that $\hat \psi : \T^d \to \C$ is  $n$ times weakly differentiable,
and let $\psi : \Z^d \to \C$
be the inverse Fourier transform of $\hat \psi$.
There is a constant $c_{d,n}$ depending only on the dimension $d$ and the
order $n$ of differentiation, such that
\begin{equation}
\label{eq:Graf}
\abs{ \psi(x) }  \le  c_{d,n} \frac{1}{ \abs x^n}
	\max_{ \abs \alpha = n } \bignorm{ \grad^\alpha \hat \psi }_1
	\qquad (x\ne 0) .
\end{equation}
\end{lemma}

\begin{proof}
This is a version of \cite[Corollary~3.3.10]{Graf14} with its estimate
$\abs{ \psi(0) } \le \norm{ \hat \psi }_1$ omitted.
The proof uses only integration by parts and applies to weak derivatives too.
\end{proof}

\begin{proof}[Proof of Proposition~\ref{prop:boot} assuming
Propositions~\ref{prop:so-nn} and \ref{prop:f}]

Let $\sigma \in (0, a_S \wedge \half)$,
where $a_S$ is from Proposition~\ref{prop:so-nn}.
Using Lemma~\ref{lem:m_S},
we can choose $\delta_3 > 0$ small enough that
\begin{align} \label{eq:sigma_choice}
\sigma \mp \le (a_S \wedge \tfrac 1 2) m_S(\mup),
\end{align}
and also that $\mp L$ and $m_S(\mup) L$ are small enough to
apply Proposition~\ref{prop:f}
for all $p \in [p_c - \delta_3, p_c)$
and $m\in [0, \sigma \mp]$.
Then, under the assumption that $\btil(m) \le 3$,
we can use
Proposition~\ref{prop:f}, Lemma~\ref{lem:fourier}
(with $\psi = f_p \supm$ and $n = d-2$), and the decomposition \eqref{eq:Gm_decomp} of $G_p \supm$ to get
\begin{align}
G_p \supm(x) = \lambda_p S_\mup \supm(x)
	+ O\( \frac{ \beta (L^{-c} + \beta) }{ \abs x^{d-2} } \)
\qquad (x\ne 0) .
\end{align}
Since $m \le \sigma \mp \le a_S m_S(\mup)$ by \eqref{eq:sigma_choice}, Proposition~\ref{prop:so-nn} gives
\begin{align}
S_\mup \supm(x)
= S_\mup(x) e^{mx_1}
\le \frac{ K_S }{ L^{1-\eta} \abs x^{d-2} } e^{- (a_S m_S(\mup) - m )\norm x_\infty }
\le \frac{ K_S }{ L^{1-\eta} \abs x^{d-2} } .
\end{align}
With $\lambda_p = \gp + O(\beta)$ from Lemma~\ref{lem:m_S},
and with $\beta = \const\, L^{-1+\eta}$
from \eqref{eq:beta},
we obtain
\begin{align}
G_p \supm(x) &\le (\gp + O(\beta)) \frac{ K_S }{ L^{1-\eta} \abs x^{d-2} }
	+ O\( \frac{ \beta (L^{-c} + \beta) }{ \abs x^{d-2} } \) 	\nl
&=  \big(1 + O(\beta) + O(L^{-c}) \big)  \frac{ \gp K_S }{ L^{1-\eta} \abs x^{d-2} }
\qquad\qquad\qquad (x\ne 0) .
\end{align}
By choosing the constant $L_3$ larger to ensure that
$1 + O(\beta) + O(L^{-c}) \le 2$ for all $L \ge L_3$,
we obtain $\btil(m) \le 2$, and the proof is complete.
\end{proof}

\subsection{Proof of Proposition~\ref{prop:f}}
\label{sec:f}

\subsubsection{A key lemma}

To prove Proposition~\ref{prop:f}, we must show that the function
\begin{equation}
\label{eq:fhatm}
    \hat f_p \supm = \frac{\hat E_p \supm}{\hat A_\mup \supm \hat F_p \supm }
\end{equation}
is $d-2$ times weakly differentiable, with estimates on its derivatives.
We first reduce the proof of Proposition~\ref{prop:f}
to Lemma~\ref{lem:AFEbds}, whose proof is given Section~\ref{sec:AFE-pf}.

\begin{lemma} \label{lem:AFEbds}
Let $d>8$ and $L \ge L_3$.
Suppose that $m \in [0, \half \mp \wedge \half m_S(\mup)]$,
that $\btil(m) \le 3$,
and that $\mp L$ and $m_S(\mup)L$ are sufficiently small.
Then the functions $\hat F_p \supm, \hat A_\mup \supm, \hat E_p \supm$ are $d-2$ times weakly differentiable.
In addition, there exists $\eps \in (0,1)$ such that,
for any multi-index $\gamma$ with $\abs \gamma \le d-2$
and any choice of $q_1, q_2$ satisfying
\begin{equation} \label{eq:q1q2}
    \frac{ \abs \gamma } d < q_1\inv < 1,
    \qquad
    \frac{ 2 - \eps + \abs \gamma } d < q_2\inv < 1 ,
\end{equation}
the following hold with constants independent of $p,m,L$:
\begin{align} \label{eq:AFEbds_SO}
\biggnorm{ \frac{ \grad^\gamma \hat A_\mup \supm }{ \hat A_\mup \supm } }_{q_1}, \;
\biggnorm{ \frac{ \grad^\gamma \hat F_p \supm }{ \hat F_p \supm } }_{q_1}
	\lesssim 1 , \qquad
\biggnorm{ \frac{ \grad^\gamma \hat E_p \supm }{ \hat A_\mup \supm \hat F_p \supm} }_{q_2}
\lesssim \beta ,
\end{align}
If $\abs \gamma \ne 0$, the estimates are improved to
\begin{align} \label{eq:AFEbds_SO2}
\biggnorm{ \frac{ \grad^\gamma \hat A_\mup \supm }{ \hat A_\mup \supm } }_{q_1}, \;
\biggnorm{ \frac{ \grad^\gamma \hat F_p \supm }{ \hat F_p \supm } }_{q_1}
	\lesssim L^{ \abs \gamma - d/q_1 } 	, \quad
\biggnorm{ \frac{ \grad^\gamma \hat E_p \supm }{ \hat A_\mup \supm \hat F_p \supm} }_{q_2}
	\lesssim \beta ( L^{\abs \gamma - d/q_2} + \beta ),
\end{align}
with constants independent of $p,m,L$.
\end{lemma}

\begin{proof}[Proof of Proposition~\ref{prop:f} assuming Lemma~\ref{lem:AFEbds}]
For differentiability of
$\hat f_p \supm = \hat E_p \supm / ( \hat A_\mup \supm \hat F_p \supm )$,
we apply the usual product and quotient rules,
and we show that the result is a (locally) integrable function
\cite[Lemmas~A.2--A.3]{LS24a}.
Let $\alpha$ be any multi-index with $\abs \alpha \le d-2$.
By the product and quotient rules, $\grad^\alpha \hat f_p \supm$
is a linear combination of terms of the form
(omitting subscripts $p, \mup$ and superscript $(m)$)
\begin{equation} \label{eq:f_decomp}
\( \prod_{n=1}^i \frac{ \grad^{\delta_n} \hat A }{\hat A } \)
\( \frac{ \grad^{\alpha_2} \hat E }{ \hat A \hat F } \)
\( \prod_{l=1}^j \frac{ \grad^{\gamma_l} \hat F  }{ \hat F } \),
\end{equation}
where $\alpha$ is decomposed as
$\alpha = \alpha_1 + \alpha_2 + \alpha_3$,
$ 0 \le i \le \abs {\alpha_1}$, $0 \le j \le \abs {\alpha_3 }$,
$\sum_{n=1}^i \delta_n = \alpha_1$, and $\sum_{l=1}^j \gamma_l = \alpha_3$.
By Lemma~\ref{lem:AFEbds} and
H\"older's inequality, the $L^r$ norm of \eqref{eq:f_decomp} is bounded by
\begin{align}
\( \prod_{n=1}^i \biggnorm{ \frac{ \grad^{\delta_n} \hat A }{\hat A } }_{r_n} \)
\biggnorm{ \frac{ \grad^{\alpha_2} \hat E }{ \hat A \hat F } }_q
\( \prod_{l=1}^j \biggnorm{ \frac{ \grad^{\gamma_l} \hat F  }{ \hat F } }_{q_l} \),
\end{align}
where $r$ can be any positive number that satisfies
\begin{align}
\frac 1 r &= \( \sum_{n=1}^i \frac 1 {r_n} \) + \frac 1 q + \( \sum_{l=1}^j \frac 1 {q_l} \)
> \frac{ \sum_{n=1}^i \abs{\delta_n}  } d  + \frac{ 2 - \eps + \abs{\alpha_2}   } d + \frac{ \sum_{l=1}^j \abs{\gamma_l}  } d
= \frac{  \abs \alpha  + 2 - \eps } d.
\end{align}
In particular, $r=1$ is allowed, because $\abs \alpha \le d-2$ and $\eps > 0$.
This shows that the product \eqref{eq:f_decomp} is in $L^1$ and
that $\hat f_p \supm$ is $d-2$ times weakly differentiable.
In addition,
since derivatives of $\hat f_p \supm$ inherit estimates of \eqref{eq:f_decomp},
from \eqref{eq:AFEbds_SO} we get
\begin{align} \label{eq:fbeta}
\bignorm { \grad^\alpha \hat f_p \supm }_{r} \lesssim \beta
	\qquad (r\inv > \frac{ \abs \alpha + 2 - \eps }d ) ,
\end{align}
with the constant independent of $p,m,L$.
When $\abs \alpha \ne 0$, since there is at least one derivative taken,
we can use the stronger \eqref{eq:AFEbds_SO2} in one of the factors.
This gives
\begin{align} \label{eq:fbeta2}
\bignorm { \grad^\alpha \hat f_p \supm }_{r} \lesssim \beta (L^{-c} + \beta)
	\qquad (r\inv > \frac{ \abs \alpha + 2 - \eps}d )
\end{align}
for some $c>0$ produced by the strict inequalities in \eqref{eq:q1q2}.
The proof is concluded by taking $r=1$.
\end{proof}

\subsubsection{Massive infrared bound}

The next lemma extends the tilted infrared bound
of Corollary~\ref{cor:tilted_infrared} to include a mass term when $k=0$.
It also treats the random walk case.

\begin{lemma}
\label{lem:massive_infrared}
~
\begin{enumerate}
\item[(i)]
Let $d > 8$ and $L \ge L_2$.
If $\mp L$ is sufficiently small, then
\begin{align}
\abs{ \hat F_p \supm (k) }
\gtrsim L^2 m^2 + L^2 \abs k^2 \wedge 1
	\qquad (m \in [0, \tfrac 1 2 \mp]),
\end{align}
with the constant independent of $p, m, L$.

\item[(ii)]
Let $d > 2$ and $L \ge L_2$.
If $m_S(\mu) L$ is sufficiently small, then
\begin{align} \label{eq:Am_lb}
\abs{ \hat A_\mu \supm (k) }
\gtrsim L^2 m^2 + L^2 \abs k^2 \wedge 1
	\qquad (m \in [0, \tfrac 1 2 m_S(\mu)]),
\end{align}
with the constant independent of $\mu,m,L$.
\end{enumerate}
\end{lemma}

\begin{proof}
(i)
We begin with
\begin{align} \label{eq:Fm_decomp}
\abs{ \hat F_p\supm (k) }
\ge \Re \hat F_p\supm (k)
= \hat F_p \supm(0) + \Re [ \hat F_p \supm(k) - \hat F_p \supm (0) ] .
\end{align}
By Corollary~\ref{cor:tilted_infrared}, when $\mp L$ is sufficiently small,
the second term on the right-hand side is bounded from below by $\half \KIR  (L^2 \abs k^2 \wedge 1)$.
Therefore, it suffices to prove that $\hat F_p \supm(0) \gtrsim L^2 m^2$.

For this, we use a strategy similar to that used for $\mp$ in the proof of Theorem~\ref{thm:mass} (cf. \eqref{eq:m_asymp_pf}).
Let $M = \mp$.
Since $\hat F_p \supM(0) = 0$,
\begin{align} \label{eq:Fm0_lb}
\hat F_p \supm(0)
= \hat F_p\supm(0) - \hat F_p \supM (0)
= - \sum_{x\in \Z^d} F_p(x) (\cosh Mx_1 - \cosh mx_1).
\end{align}
Let $\eps \in (0,1)$.
By the inequality $\abs{ \cosh t - (1 + \half t^2) } \lesssim \abs t^{2+\eps} \cosh t $ (if $\eps \le 2$), and since $m \le M$,
\begin{align}
\cosh Mx_1 - \cosh mx_1
= \half ( M^2 - m^2 ) x_1^2 + O(M^{2+\eps}) \abs{x_1}^{2+\eps} \cosh Mx_1 .
\end{align}
We insert this into \eqref{eq:Fm0_lb},
and use the $a = 2+\eps$ moment of $F_p\supM(x)$ bounded in \eqref{eq:F_moments},
to obtain
\begin{align}
\hat F_p \supm(0)
&= - \half (M^2 - m^2) \sum_{x\in \Z^d} x_1^2 F_p(x)
	- O(M^{2+\eps})  \sum_{x\in \Z^d} \abs {x_1}^{2+\eps} \abs{ F_p(x) } \cosh Mx_1 \nl
&= \half (M^2 - m^2) \del_1^2 \hat F_p(0)
	- O(M^{2+\eps})  \sum_{x\in \Z^d} \abs {x_1}^{2+\eps} \abs{ F_p\supM(x) } \nl
&\ge \frac 3 4 \mp^2 L^2 \bigg( \frac{ \del_1^2 \hat F_p(0) }{2 L^2}
	- O(\mp^\eps L^\eps e^{\mp L})  \bigg) ,
\end{align}
where the last inequality uses $m \le \half M$.
Since $\del_1^2 \hat F_p(0) \gtrsim L^2$ by the infrared bound \eqref{eq:FIR} and Taylor's Theorem,
when $\mp L$ is sufficiently small we obtain
\begin{align}
\hat F_p \supm(0) \gtrsim \mp^2 L^2 \ge 4 m^2 L^2,
\end{align}
as desired. This concludes the proof.

\smallskip \noindent
(ii)
The proof of part~(i) relies only on the untilted infrared bound \eqref{eq:FIR},
on the moment estimate $\sum_{x\in \Zd} \abs x^{2+\eps} \abs{ F_p \supm(x) } \lesssim L^{2+\eps} e^{mL}$ from \eqref{eq:F_moments},
and on the fact that $\hat F_p \supmp(0) = 0$ by \eqref{eq:Gm0}.
The same proof works when $F_p$ is replaced by $A_\mu = \delta - \mu D$,
using the infrared bound \eqref{eq:AIR}, the moment estimate
\eqref{eq:A_moments}, and \eqref{eq:chiSinf}.
\end{proof}

\subsubsection{Proof of Lemma~\ref{lem:AFEbds}}
\label{sec:AFE-pf}

We now complete the proof of Theorem~\ref{thm:near_critical} by proving
Lemma~\ref{lem:AFEbds}.

Recall from \eqref{eq:Edef} that $E_p = h_p * A_\mup - \lambda_p F_p$.
We first note that
\begin{align} \label{eq:EPi}
E_p(x)  = O(\beta) \delta_{0,x}  + \Pi_p(x) + O(\beta) (D * h_p)(x) .
\end{align}
Too see this, we recall from \eqref{eq:FG} that $\hat F_p(0)=1-p\hat h_p(0)$,
so $\mup$ can be rewritten as
\begin{align}
\label{eq:mu-identity}
\mu_p
= 1 - \frac{ \lambda_p } \chip
= 1 - \lambda_p \frac{ \hat F_p(0) }{ \hat h_p(0) }
= 1 - \frac{ \lambda_p }{ \hat h_p(0) } + p \lambda_p .
\end{align}
Inserting into the definition of $E_p$, we have
\begin{align}
E_p
&= h_p * (\delta - \mu_p D) - \lambda_p (\delta - pD*h_p) \nl
&= h_p * \bigg(\delta - \Big(1 - \frac{ \lambda_p }{ \hat h_p(0) }\Big) D\bigg)   - \lambda_p \delta	\nl
&= ( \gp - \lambda_p ) \delta + \Pi_p - \Big(1 - \frac{ \lambda_p }{ \hat h_p(0) }\Big) (h_p * D ) ,
\end{align}
and \eqref{eq:EPi} follows from $\lambda_p = \gp + O(\beta)$
and $\hat h_p(0) = \gp + \hat \Pi_p(0) = \gp + O(\beta)$.

In Lemma~\ref{lem:AFEbds},
we assume that $m \in [0, \half \mp \wedge \half m_S(\mup)]$,
and that $\mp L$ and $m_S(\mup)L$ are sufficiently small,
so the massive infrared bounds of Lemma~\ref{lem:massive_infrared} apply.
Since $\abs k^2 + m^2 \ge \half (\abs k + m)^2$,
the latter gives
\begin{equation} \label{eq:massive_infrared}
\abs{ \hat F_p \supm (k) }, \;\; \abs{ \hat A_\mup \supm (k) }
\gtrsim \begin{cases}
L^2 (\abs k + m)^2		& (\abs k < L\inv) \\
1					&(\abs k \ge L\inv) ,
\end{cases}
\end{equation}
with constant independent of $p,m,L$.
For a small number of derivatives  of the ratio \eqref{eq:fhatm} for $\hat f_p \supm $, we seek cancellations of powers of $\abs k + m$. This is
achieved using the next lemma.

\begin{lemma} \label{lem:E}
Let $d > 8$, $L \ge L_3$, and $\mp L$ be sufficiently small.
For any $j \in \{1, \dots, d\}$, we have
\begin{align}
\label{eq:A_linear}
\abs{ \grad_j \hat A_\mu \supm(k) } &\lesssim L^2 (\abs k + m ) e^{mL}
	\qquad (m \ge 0) , \\
\label{eq:F_linear}
\abs{ \grad_j \hat F_p \supm(k) } &\lesssim L^2 (\abs k + m )  e^{mL}
	\qquad (m \in [0, \mp]) .
\end{align}
Also, given any $\eps \in (0,1]$ such that $\eps <d-8$,
if $m \in [0, \mp]$ and $\btil(m) \le 3$ then
\begin{align} \label{eq:E_gamma_ub}
\abs{ \grad^\gamma \hat E_p \supm (k) }
\lesssim  \beta L^{2+\eps} (\abs k  + m)^{2 + \eps - \abs \gamma }  e^{mL}
	\qquad (\abs \gamma \le 2).
\end{align}
All constants are independent of $\mu,p,m,L$.
\end{lemma}

\begin{proof}
The proof proceeds by expansion near $(k,m) = (0,0)$. Powers of $L$ arise from moments of $A_\mu, F_p$, or $E_p$.

We prove \eqref{eq:F_linear} first. By symmetry,
\begin{align}
\hat F_p \supm(k) =
	\sum_{x\in \Z^d} \cos (k\cdot x) \cosh (mx_1) F_p(x)
	+ i \sum_{x\in \Z^d} \sin (k\cdot x) \sinh (mx_1) F_p(x) .
\end{align}
To estimate its derivative, we use the inequalities
\begin{align}
\abs{ \grad_j \cos (k\cdot x) } \le \abs k \abs x^2,  \quad
\abs{ \grad_j \sin (k\cdot x) }  \le \abs x, \quad
\abs{ \sinh (mx_1) } \le m \abs{ x_1} \cosh (mx_1).
\end{align}
By the moment estimate \eqref{eq:F_moments} for $F_p \supm$, these give
\begin{align}
\abs{ \grad_j \hat F_p \supm(k) }
\le  (\abs k + m) \sum_{x\in \Z^d} \abs x^2 \abs{ F_p \supm(x) }
\lesssim (\abs k + m) L^2 e^{mL},
\end{align}
which is \eqref{eq:F_linear}.
The proof of \eqref{eq:A_linear} is the same,
using \eqref{eq:A_moments} with $a=2$.

For \eqref{eq:E_gamma_ub}, with $\eps$ as specified in the
statement, by symmetry we have
\begin{equation}
    \hat E_p\supm(k) = \sum_{x\in\Z^d}E_p(x)
    \bigl( \cosh(mx_1)\cos(k \cdot x) + i \sinh(mx_1)\sin(k\cdot x) \big).
\end{equation}
It is then an exercise in Taylor expansion in $k$ and $m$
(details can be found in \cite[Lemma~2.9]{Liu24}), using the
vanishing moments \eqref{eq:E_condition} of $E_p(x)$, to see that
for any $\abs \gamma \le 2$ and any $\eps \in (0,1]$,
\begin{align}
\label{eq:Egambd}
\abs{ \grad^\gamma \hat E_p \supm (k) }
\lesssim (\abs k  + m)^{2 + \eps - \abs \gamma } \sum_{x\in \Zd} \abs x^{2+\eps} \abs{ E_p \supm (x) } .
\end{align}
It therefore suffices to show that $\sum_{x\in \Zd} \abs x^{2+\eps} \abs{ E_p \supm (x) } \lesssim \beta L^{2+\eps} e^{mL}$.
Since $\eps < d-8$ and $\btil(m) \le 3$, the decay of $\Pi_p \supm$ from \eqref{eq:Pim_decay} implies that $\sum_{x\in \Zd} \abs x^{2+\eps} \abs{ \Pi_p \supm (x) } \le O(\beta)$.
By tilting equation \eqref{eq:EPi}, this leads to
\begin{align}
\sum_{x\in \Zd} \abs x^{2+\eps} \abs{ E_p \supm (x) }
&\le \sum_{x\in \Zd} \abs x^{2+\eps} \Big( \abs{ \Pi_p \supm (x) }
	+ O(\beta) ( D \supm * h_p \supm )(x)   \Big) 	\nl
&\lesssim  \beta  +
\beta \(
	\bignorm{ \abs x^{2+\eps} D \supm(x) }_1 \norm{ h_p \supm }_1
	+ \norm{ D \supm }_1  \bignorm{ \abs x^{2+\eps} h_p \supm }_1	\) 	\nl
&\lesssim \beta  +  \beta \big( L^{2+\eps} e^{mL} ( \gp + \beta ) + e^{mL} \beta \big)	\nl
&\lesssim  \beta L^{2+\eps} e^{mL} .
\end{align}
This completes the proof.
\end{proof}

\begin{proof}[Proof of Lemma~\ref{lem:AFEbds}]
We first show that $\hat F_p \supm, \hat A_\mup \supm, \hat E_p \supm$ are $d-2$ times weakly differentiable.
By \cite[Lemma~A.4]{LS24a},
the Fourier transform $\hat g$ of a function $g: \Zd \to \R$ is $d-2$ times weakly differentiable if $\abs x^{d-2} g(x) \in \ell^2(\Zd)$.
For $g = A_\mup \supm$, this condition follows from the fact that $A_\mu$
has finite support.
For $g = F_p \supm$ or $E_p \supm$,
we first calculate the moment of $h_p \supm = \gp \delta + \Pi_p \supm$.
Using $\btil(m) \le 3$, \eqref{eq:Pim_decay}, and $d>8$, we have
\begin{equation} \label{eq:Pi_d-2}
\sum_{x\in \Z^d} \Bigl( \abs x^{d-2}  \abs{\Pi_p \supm(x)} \Bigr)^2
\le O(\beta^4).
\end{equation}
This implies that
\begin{equation}
\bignorm{ \abs x^{d-2} h_p \supm(x) }_2
= \bignorm{ \abs x^{d-2} \Pi_p \supm(x) }_2
\lesssim \beta^2 .
\end{equation}
Therefore,
by the triangle inequality and Young's convolution inequality,
the moment of $F_p\supm = \delta - p D\supm * h_p \supm$ obeys
\begin{align}
\bignorm{ \abs x^{d-2} F_p \supm(x) }_2
&= p \bignorm{ \abs x^{d-2} ( D \supm * h_p \supm )(x) }_2			\nl
&\le 2^{d-2} p_c  \(
	\bignorm{ \abs x^{d-2} D \supm(x) }_1 \norm{ h_p \supm }_2
	+ \norm{ D \supm }_1  \bignorm{ \abs x^{d-2} h_p \supm (x) }_2 \)  	\nl
&\lesssim L^{d-2} e^{mL} .
\end{align}
Similarly, by \eqref{eq:EPi},
\begin{align}
\bignorm{ \abs x^{d-2} E_p \supm(x) }_2
\le  \bignorm{ \abs x^{d-2} \Pi_p \supm(x) }_2
	+ O(\beta) \bignorm{ \abs x^{d-2} ( D \supm * h_p \supm )(x) }_2
\lesssim \beta L^{d-2} e^{mL} .
\end{align}
This shows that each of $\hat F_p \supm, \hat A_\mup \supm, \hat E_p \supm$
is $d-2$ times weakly differentiable.

For the remaining claims \eqref{eq:AFEbds_SO}--\eqref{eq:AFEbds_SO2},
we omit the subscripts $p,\mup$ in
$\hat F_p \supm, \hat A_\mup \supm, \hat E_p \supm$,
and we instead use subscripts to denote derivatives.
For example, given a multi-index $\gamma$, we write $\hat A \supm _ \gamma = \grad^\gamma \hat A_\mup \supm$.
We bound factors of $(1+mL)$ or $e^{mL}$ by constants,
using the hypotheses that
$m \in [0, \half \mp \wedge \half m_S(\mup)]$
and that $\mp L$ and $m_S(\mup)L$ are small.

\medskip\noindent
\emph{Bound on $\grad^\gamma \hat A_\mup \supm / \hat A_\mup \supm$.}
Let $\gamma$ be any multi-index.
There is nothing to prove for $\abs \gamma = 0$ since the ratio is then 1.
We will prove that, for $\abs \gamma \ge 1$,
\begin{align} \label{eq:AA_L}
\biggnorm{ \frac{ \hat A \supm _\gamma }{ \hat A \supm } }_{q}
\lesssim L^{\abs \gamma - d/q }
\qquad \Big(  \frac{ \abs \gamma \wedge 2 } d < q\inv < 1 \Big).
\end{align}
This is stronger than the desired \eqref{eq:AFEbds_SO2} by allowing more values of $q = q_1$.
It also implies $\norm{  \hat A \supm _\gamma  /  \hat A \supm  }_{q} \lesssim 1$ when we restrict to $ \abs \gamma / d < q\inv < 1$, as claimed in \eqref{eq:AFEbds_SO}.

Let $\BL = \{ \abs k < 1/L \}$.
By the massive infrared bound \eqref{eq:massive_infrared},
\begin{equation} \label{eq:AAIR}
\biggabs{ \frac{ \hat A \supm _\gamma }{ \hat A \supm }(k) }
\lesssim \frac{ \abs{ \hat A \supm _\gamma (k) } }{ L^2 (\abs k + m)^2 } \1_\BL
	+ \abs{ \hat A \supm _\gamma (k) } .
\end{equation}
Since $ \hat A \supm _\gamma = - \mup  \hat D_\gamma \supm$ when $\abs \gamma \ge 1$, by Lemma~\ref{lem:D_L}
the $L^q$ norm of the second term on the right-hand side is bounded by $L^{\abs \gamma - d/q}$ (dropping a factor of $(1+mL)e^{mL}$), as required.
For the first term,
if $\abs \gamma = 1$ we use Lemma~\ref{lem:E} to bound the numerator.
Since factors of $L^2$ cancel, the $L^q$ norm of the quotient is bounded by a multiple of
\begin{align} \label{eq:linear_pf}
e^{mL} \biggnorm{ \frac 1 { \abs k + m } \1_\BL }_q
\le e^{mL} \biggnorm{ \frac 1 { \abs k } \1_\BL }_q
\lesssim L^{1 - d/q}
	\qquad \Big(q\inv > \frac 1 d \Big) ,
\end{align}
as desired.
If $\abs \gamma \ge 2$, we use H\"older's inequality and Lemma~\ref{lem:D_L} to
see that
\begin{align}
\biggnorm{ \frac{ \abs{ \hat A \supm _\gamma (k) } }{ L^2 (\abs k + m)^2 } \1_\BL }_q
\le \frac 1 {L^2} \norm{ \hat A \supm _\gamma }_r
	\biggnorm{ \frac 1 {\abs k^2} \1_\BL }_s
\lesssim \frac 1 {L^2} (L^{\abs \gamma - d/r} ) L^{2 - d/s}
= L^{\abs \gamma - d/q}
\end{align}
for $q\inv = r \inv + s \inv$ with $ 0\le r\inv < 1$ and $s \inv > 2/d$.
In particular, this holds for any $q\inv > 2/d$, and this completes the proof of \eqref{eq:AA_L}.

\medskip\noindent
\emph{Bound on $\grad^\gamma \hat F_p \supm / \hat F_p \supm$.}
The $\abs \gamma = 0$ case is again trivial.
We will prove for $1 \le \abs \gamma \le d-2$ that
\begin{align} \label{eq:FF_L}
\biggnorm{ \frac{ \hat F \supm _\gamma }{ \hat F \supm } }_{q}
\lesssim L^{\abs \gamma - d/q }
\qquad \Big(  \frac{ \abs \gamma } d < q\inv < 1 \Big) .
\end{align}
We again use the massive infrared bound \eqref{eq:massive_infrared} to obtain
\begin{equation} \label{eq:FF_decomp}
\biggabs{ \frac{ \hat F \supm _\gamma }{ \hat F \supm }(k) }
\lesssim \frac{ \abs{ \hat F \supm _\gamma (k) } }{ L^2 (\abs k + m)^2 } \1_\BL
	+ \abs{ \hat F \supm _\gamma (k) } .
\end{equation}
We will show that each term obeys the required bound \eqref{eq:FF_L}.

We begin with the second term, for which
we in fact improve the range of $q$
in \eqref{eq:FF_L} values to the larger range
\begin{equation}
\label{eq:more-q}
    \frac{|\gamma|-2-\theta}{d} \vee 0  \le q\inv < 1,
\end{equation}
for any fixed $\theta \in [0,d-8)$ (we will use this improvement for \eqref{eq:6.1pf2}).
Let $q$ satisfy \eqref{eq:more-q}.
Using $\hat F \supm = 1 - p \hat D \supm  \hat h_p \supm $,
$\abs \gamma \ge 1$, the product rule, H\"older's inequality,
and Lemma~\ref{lem:D_L},
\begin{align} \label{eq:F_product_rule}
\norm{ \hat F \supm _\gamma }_q
&\lesssim   \sum_{\alpha_1 + \alpha_2 =  \gamma}
	\bignorm{  \hat D_{\alpha_1} \supm }_{r_{\alpha_1}}
	\bignorm{ \hat h_{\alpha_2}\supm }_{ r_{\alpha_2} }
    \lesssim   \sum_{\alpha_1 + \alpha_2 =  \gamma}
	L^{|\alpha_1|-d/r_{\alpha_1}}
	\bignorm{   \hat h_{\alpha_2} \supm }_{ r_{\alpha_2} } ,
\end{align}
where $r_{\alpha_1}\inv \in [0, 1)$ and $r_{\alpha_1}^{-1}+r_{\alpha_2}^{-1}=q^{-1}$.
Since $\hat h_p \supm = \gp + \Pi_p \supm$, by \eqref{eq:Pi_gamma} we have
\begin{align}
\label{eq:6.1pf}
    \bignorm{   \hat h_{\alpha_2} \supm }_{ r_{\alpha_2} }
    & \lesssim
    \1_{\alpha_2=0} + \beta^2 \lesssim 1
    \qquad
    \Big( \frac{|\alpha_2|-d+6}{d}  < r_{\alpha_2}^{-1} \le 1,\; r_{\alpha_2}^{-1} \ge 0 \Big) .
\end{align}
We make the choice
\begin{equation}
    r_{\alpha_2}^{-1} = \frac{|\alpha_2|-2-\theta}{d} \vee 0,
\end{equation}
which satisfies the restriction in \eqref{eq:6.1pf} because
$\theta < d-8$.
Since $q\inv \ge r_{\alpha_2}\inv$, this enforces \eqref{eq:more-q}.
Inserting \eqref{eq:6.1pf} into \eqref{eq:F_product_rule},
and using $r_{\alpha_2} \inv \le \abs {\alpha_2} / d$,
we obtain
\begin{align}
\label{eq:Fratio-bd}
\norm{ \hat F \supm _\gamma }_q
&\lesssim   \sum_{\alpha_1 + \alpha_2 =  \gamma}
	L^{|\alpha_1|-d/r_{\alpha_1}}
= \sum_{\alpha_1 + \alpha_2 =  \gamma}
	L^{|\alpha_1|-d/q + d/ r_{\alpha_2}}
\le \sum_{\alpha_1 + \alpha_2 =  \gamma}
	L^{|\gamma|-d/q}  ,
\end{align}
which is the required bound \eqref{eq:FF_L}.

For the first term of \eqref{eq:FF_decomp},
let $\abs \gamma / d<q\inv < 1$.
If $\abs \gamma = 1$ we use Lemma~\ref{lem:E} and get the result as in \eqref{eq:linear_pf}.
If $\abs \gamma \ge 2$, we write $q\inv = r \inv + s \inv$ with
$r\inv = (\abs \gamma - 2 )/d$ and $s \inv > 2/d$.
Then we use H\"older's inequality and \eqref{eq:Fratio-bd} to obtain
\begin{align}
\label{eq:6.1pf2}
\biggnorm{ \frac{ \abs{ \hat F \supm _\gamma (k) } }{ L^2 (\abs k + m)^2 } \1_\BL }_q
\le \frac 1 {L^2} \norm{ \hat F \supm _\gamma }_r
	\biggnorm{ \frac 1 {\abs k^2} \1_\BL }_s
\lesssim \frac 1 {L^2} (L^{\abs \gamma - d/r} ) L^{2 - d/s}
= L^{ \abs \gamma - d/q}.
\end{align}
This concludes the proof of \eqref{eq:FF_L}.

\medskip\noindent
\emph{Bound on $\grad^\gamma \hat E_p \supm / ( \hat A_\mup \supm \hat F_p \supm)$.}
Let $\abs \gamma \le d-2$,
and let $\eps \in (0,1]$ with $\eps <d-8$
so that we can apply
Lemma~\ref{lem:E}.
We will prove that
\begin{align} \label{eq:EAF_L}
\biggnorm{ \frac{ \hat E \supm _\gamma }{ \hat A \supm \hat F \supm }}_q
\lesssim \beta \1_{|\gamma|=0} + \beta^2 + \beta L^{\abs \gamma - d/q}
	\qquad \Big( \frac{ 2 - \eps + \abs \gamma } d < q\inv < 1 \Big) .
\end{align}
This establishes \eqref{eq:AFEbds_SO2} immediately,
and also establishes \eqref{eq:AFEbds_SO} once we observe that
$\beta^2 \le \beta$ and $L^{\abs \gamma - d/q} \le L^{\eps - 2} \le 1$ for these values of $q\inv$.
To start the proof,
we use both of the massive infrared bounds in \eqref{eq:massive_infrared}, to get
\begin{equation} \label{eq:EAF_decomp}
\biggabs{ \frac{ \hat E \supm _\gamma }{ \hat A \supm \hat F \supm } (k) }
\lesssim \frac{ \abs{ \hat E \supm _\gamma (k) } }{ L^4 (\abs k + m)^4 } \1_\BL
	+ \abs{ \hat E \supm _\gamma (k) } .
\end{equation}

We again begin with the second term.
By tilting \eqref{eq:EPi} and taking the Fourier transform,
\begin{align}
\hat E \supm = O(\beta) 1 + \hat \Pi_p \supm + O(\beta) \hat D \supm \hat h_p \supm .
\end{align}
We calculate derivatives of $\hat E \supm$ using the product rule.
Derivatives of $\hat \Pi_p \supm$ are estimated in \eqref{eq:Pi_gamma},
and derivatives of $\hat D \supm \hat h_p \supm$ have already been estimated below \eqref{eq:F_product_rule}.
The result is
\begin{align} \label{eq:E_gamma_L}
\norm{ \hat E_\gamma \supm }_q
\lesssim \beta \1_{\gamma = 0} + \beta^2
	+ \beta L^{\abs \gamma - d/q}
\qquad \Big( \frac { \abs \gamma - 2 -\theta} d  \vee 0  \le q\inv < 1 \Big) ,
\end{align}
which holds, in particular, for
$( 2 - \eps + \abs \gamma) / d < q\inv < 1$,
since  $2 - \eps \ge 1 \ge -2 -\theta$.

For the first term of \eqref{eq:EAF_decomp},
if $\abs \gamma \le 2$ we use Lemma~\ref{lem:E} to bound the numerator.
After a cancellation of powers of $L$ and of $(\abs k + m)$, the $L^q$ norm of the quotient is bounded by a multiple of
\begin{align}
\frac \beta { L^{2-\eps} } \biggnorm{ \frac 1 { (\abs k + m)^{2-\eps + \abs \gamma} } \1_\BL }_q
\lesssim \frac \beta { L^{2-\eps} }  L^{2-\eps + \abs \gamma - d/q}
= \beta L^{\abs \gamma - d/q}
	\qquad \Big(q\inv > \frac {2-\eps + \abs \gamma} d\Big) ,
\end{align}
as desired.
If $\abs \gamma \ge 3$,
we let  $r\inv = (\abs \gamma - 2 - \eps)/d$ and use \eqref{eq:E_gamma_L}.  Here we need $\eps \le \theta$,
and we can choose $\theta$
to satisfy this inequality because $\eps$ is a number in $(0,1]$ with $\eps < d-8$,
and we may choose any $\theta \in [0,d-8)$.
Then, by H\"older's inequality,
\begin{align}
\label{eq:6.2pf2}
\biggnorm{ \frac{ \abs{ \hat E \supm _\gamma (k) } }{ L^4 (\abs k + m)^4 } \1_\BL }_q
&\le \frac 1 {L^4} \norm{ \hat E \supm _\gamma }_r
	\biggnorm{ \frac 1 {\abs k^4} \1_\BL }_s
\nl &
\lesssim \frac 1 {L^4} \beta(\beta+L^{\abs \gamma - d/r}) L^{4 - d/s}
\le \beta^2 + \beta L^{ \abs \gamma - d/q}
\end{align}
for $q\inv = r \inv + s \inv$ with $s \inv > 4/d$.
Since $r\inv = (\abs \gamma - 2 - \eps) / d$,
the bound holds for any $q\inv > ( 2 - \eps + \abs \gamma ) / d$.
This concludes the proof of \eqref{eq:EAF_L} and completes the proof of Lemma~\ref{lem:AFEbds}.
\end{proof}

\appendix

\section{Tilted diagrammatic estimates: proof of Lemma~\ref{lem:Pi_moments}}
\label{app:diagram}

In this appendix, we prove Lemma~\ref{lem:Pi_moments}.
The proof uses the following elementary estimate for the transition probability.

\begin{lemma} \label{lem:D_L}
Let $L\ge L_2$ with $L_2$ sufficiently large, let $m \ge 0$, and let
$\alpha$ be any multi-index.
Then
\begin{align}
\label{eq:D_alpha}
\bignorm { \grad^\alpha \hat D \supm }_{q } &\lesssim L^{\abs\alpha - d/q} (1+mL)e^{mL}	
\qquad ( 0 \le q\inv < 1 ).
\end{align}
\end{lemma}

\begin{proof}
We define $v:\R^d\to \R$ by $v(y) = 2^{-d} \1\{ 0 < \norm y_\infty \le 1 \}$.
Then $\int_{[-1, 1]^{d}}v(y)  d y = 1$, and we can write $D\supm$ as
\begin{equation}
D \supm(x) = D(x) e^{mx_1}
= \frac{v(x/L) e^{mL (x_1/L)}}{\sum_{x\in \Z^d}v(x/L)}.
\end{equation}
With this rewriting, we see from a minor extension of \cite[(5.34)]{HS90a} that,
for all $\alpha$ and for any choice of the order $I \subset \{1, \dots, d\}$
for the spatial derivative $\del^I$,
\begin{align}
\abs{ \grad^\alpha \hat D \supm (k) }
\le 2 L^{\abs \alpha} \bignorm{ \del^I [ y^\alpha v(y) e^{mL y_1}  ] }_1
	\prod_{ \nu \in I} \abs{ 2L \sin(k_\nu / 2) }\inv .
\end{align}
By the product rule and $e^{mL y_1} \le e^{mL}$,
we have $\bignorm{ \del^I [ y^\alpha v(y) e^{mL y_1}  ] }_1 \lesssim (1 + mL)e^{mL}$.
The inequality \eqref{eq:D_alpha} then follows exactly as in
the proof of its $m=0$ case in \cite[(B.3)]{LS24b}.
\end{proof}

\begin{proof}[Proof of Lemma~\ref{lem:Pi_moments}]
Let $d>8$, $p \in [\tfrac12 p_c, p_c)$, and $m<m(p)$.
We assume that $b(m) \le 4$ and $mL \le 1$, and we will show that if $L \ge L_2$ with
$L_2$ sufficiently large then
\begin{align}
\label{eq:squaresmall-pf}
\square \supm_p \lesssim L^{-d+1} \le \frac 1 {100},
\end{align}
and also that for any $a \in [ 0, \frac{ d-4} 2 )$ there is a constant $K_a$ such that
\begin{align}
\label{eq:Pi_moments-pf}
\sum_{x\in \Z^d} \abs x^a  \abs{\Pi_p \supm(x)}   \le K_a .
\end{align}
In addition, we will prove
around \eqref{eq:square_pc}--\eqref{eq:Pi_moments_crit}
that if $b(m)\le 4$ for all $m<m(p)$
 and if $\mp L \le 1$, then the bounds
\eqref{eq:squaresmall-pf}--\eqref{eq:Pi_moments-pf} also hold when $m=m(p)$.

The proof uses diagrammatic estimates for $\Pi_p$ obtained in \cite{HHS03}
(which are simpler than the original diagrams of \cite{HS90b}),
where our $\Pi_p$ is denoted by $\psi_p$.
The function $\Pi_p$ is given by an absolutely convergent series $\Pi_p(x) = \sum_{N=0}^\infty (-1)^N \Pi_p\supN(x)$ (the superscript here is \emph{not} an exponential tilt), where each $\Pi_p\supN$ is non-negative and $\Z^d$-symmetric.
The $N=0$ term is absent for lattice trees.
We write $\Pi_p\supNm$ for the exponential tilt of $\Pi_p\supN$,
and we will prove the stronger statement that
\begin{align} \label{eq:Pi_moments_strong}
\sum_{N=0}^\infty  \sum_{x\in \Z^d} \abs x^a  \Pi_p\supNm (x)  \le K_a
	\qquad ( 0 \le a < \frac{d-4} 2 ) .
\end{align}

The upper bounds for $\Pi_p \supN(x)$
are given in \cite[(4.42)]{HHS03} (for trees) and \cite[(4.51) and (4.56)]{HHS03} (for animals), as explicitly defined diagrams.
The diagrams for the $N=4$ case are depicted in Figure~\ref{figure:MN}.
For general $N$, we use the triangle inequality to distribute the factor $|x|^a$
along the top of the diagram, so that in each term only one top line carries
the moment factor.  We distribute the tilt $e^{mx_1}$ along the bottom of the diagram,
so that each line in the bottom path from $0$ to $x$ contains a tilt.
We then bound the resultant diagrams by decomposing into subdiagrams, as is
usual with the lace expansion.

\begin{figure}[h]
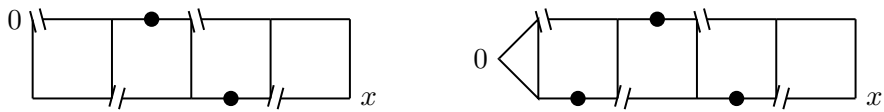

\center{
\LTPiFour   \qquad   \LAPiFour
\caption{
One example of diagrams for $\Pi_p \supk 4(x)$ for
lattice trees (left) and for lattice animals (right).
There are $2^{N-1}$ similar diagrams for $\Pi_p \supN(x)$ for $N\ge 1$.
}
\label{figure:MN}
}
\end{figure}

\smallskip \noindent \emph{Diagrammatic estimate for lattice trees.}
For lattice trees, the relevant subdiagrams are the tilted squares
\begin{equation} \label{eq:tilted_squares}
\begin{aligned}
\square_{1,p} \supm &=
	\norm{ p D \supm * G_p \supm * G_p * G_p * G_p }_\infty 	, \\
\square_{2,p} \supm &=
	\norm{ p D * G_p * G_p * G_p \supm * G_p \supm }_\infty 	,
\end{aligned}
\end{equation}
and the tilted and weighted triangles
\begin{equation}
\label{eq:tilted_triangles}
\begin{aligned}
\triangle_{1,a,p} \supm &=  \norm{ p( \abs x^a D ) * G_p * G_p \supm * G_p \supm }_\infty , \\
\triangle_{2,a,p} \supm &=  \norm{ pD * ( \abs x^a G_p ) * G_p \supm * G_p \supm }_\infty , \\
\triangle_{3,a,p} \supm &=  \norm{ pD \supm * G_p \supm * G_p * ( \abs x^a G_p ) }_\infty .
\end{aligned}
\end{equation}
We write
$\tilde \square_p \supm = \max \{ \square_{1,p} \supm , \square_{2,p} \supm \}$
and $\triangle_{a,p} \supm = \max\{ \triangle_{1,a,p} \supm, \triangle_{2,a,p} \supm, \triangle_{3,a,p} \supm \}$. Then,
as in \cite[Lemma~2.3]{HS90b},
\begin{equation} \label{eq:PiN_diagram}
\sum_{x\in \Z^d} \abs x^a  \Pi_p\supNm (x)
\le 2^{N-1} (2N)^{a+1} \triangle_{a,p} \supm (\tilde \square_p \supm )^{N-1}
	\qquad (N\ge1).
\end{equation}
We will show that, under the hypotheses $b(m) \le 4$ and $mL \le 1$, we have
\begin{align} \label{eq:Pi_moments_claim}
\tilde \square_p \supm \lesssim L^{-d+1},	\qquad
\triangle_{a,p} \supm \lesssim 1 .
\end{align}
By requiring that $L_2$ be large enough that $2 \tilde \square_p \supm < 1$,
we can then sum \eqref{eq:PiN_diagram} over $N$ to get the desired \eqref{eq:Pi_moments_strong},
since $\Pi_p \supzerom(x) = 0$ for lattice
trees.\footnote{With more effort, it could be shown that the triangle is small---not merely
bounded---leading to a small parameter on the right-hand side of
\eqref{eq:Pi_moments-pf}.  We do not need the small parameter so do not make the
extra effort.}
Also, for the tilted square $\square_p \supm$ defined in \eqref{eq:squaremax},
since $\gp^4$ is subtracted from the convolutions in \eqref{eq:squaremax},
for a nonzero contribution
at least one of the four functions in the convolutions must take at least one step.
Thus, by combining the observation
$G_p(x) \le (pD * G_p)(x)$ for $x\ne 0$
with the fact that $ e^{-mL} D(x) \le D \supm(x) \le e^{mL} D(x)$,
we have
\begin{align}
\square_p \supm \le \max \Bigl\{
	(2e^{mL} + 2) \square_{2,p} \supm,\,
	(1 + 3 e^{mL}) \square_{1,p} \supm
	\Bigr\}
\le 4e^{mL} \tilde \square_p .
\end{align}
The claim \eqref{eq:squaresmall-pf} then follows from
\eqref{eq:Pi_moments_claim}, by taking  $L_2$ larger.

We now prove \eqref{eq:Pi_moments_claim}.
The hypothesis $b(m) \le 4$ is equivalent to
\begin{align} \label{eq:Gmk}
\abs{ \hat G_p \supm (k) } \le \frac {4\gpc} \KIR \frac 1 { L^2 \abs k^2 \wedge 1 } .
\end{align}
For the squares,
using $D\supm(x) \le e^{mL} D(x)$, $mL \le 1$, and \eqref{eq:Gk_bound}, we have
\begin{align}
\square_{1,p}\supm
&\le p e^{mL} \norm{ \hat D \hat G_p \supm ( \hat G_p )^3 }_1
\le p_c e \frac {32 \gpc^4 } {\KIR^4}  \Bignorm{ \frac{ \hat D(k) }{ L^8 \abs k^8 \wedge 1 } }_1 ,
\\
\square_{2,p}\supm
&\le p \norm{ \hat D ( \hat G_p )^2 ( \hat G_p \supm )^2 }_1
\le p_c \frac {64 \gpc^4 } {\KIR^4}  \Bignorm{ \frac{ \hat D(k) }{ L^8 \abs k^8 \wedge 1 } }_1 .
\end{align}
Let $\BL = \{k: \abs k < 1/L \}$, then
\begin{align}
\Bigabs{ \frac { \hat D(k) } { L^8 \abs k^8 \wedge 1 } }
\le \frac{ \norm{ \hat D } _ \infty }{ L^8 \abs k^8 } \1_\BL    +  \abs{ \hat D(k) } .
\end{align}
We take the $L^1$ norm of the first term.  For the second term,
we set $q=\frac{d}{d-1}$, use $\|\hat D\|_1 \le \|\hat D\|_q$, and apply
Lemma~\ref{lem:D_L} with $\alpha = 0$ and $m=0$.  This gives
\begin{align}
\Bignorm{  \frac{  \hat D(k)  } {  L^8 \abs k^8 \wedge 1  }   }_1
\lesssim  \frac 1 {L^8} L^{8-d}   +  L^{-d + 1},
\end{align}
and therefore $\tilde \square_p \supm \lesssim L^{-d+1}$.

For the triangles,
we decompose using Young's convolution inequality.
Since $G_p(x) \le G\crit (x) \le O(\nnnorm x^{-(d-2)})$ by
\eqref{eq:eta_zero},
\begin{alignat}2 \label{eq:G_norm}
\sup_{p \le p_c}  \norm{ G_p }_q &< \infty
	\qquad && \Big( \frac 1 q < 1 - \frac 2 d \Big) , \\
\sup_{p \le p_c} \bignorm{ \abs x^a G_p(x) }_2 &< \infty
	\qquad && \Big( 0 \le a < \frac{d-4} 2 \Big) .
\end{alignat}
Therefore,
by the Parseval relation and \eqref{eq:Gmk}, 
the second triangle in \eqref{eq:tilted_triangles} obeys
\begin{equation}
\triangle_{2,a,p} \supm
\le p \norm D_1 \bignorm{ \abs x^a G_p(x) }_2 \bignorm{ G_p \supm * G_p \supm }_2
\lesssim  \bignorm{ ( \hat G_p \supm )^2 }_2
\lesssim \Bignorm{ \frac{ 1 }{ L^4 \abs k^4 \wedge 1 } }_2
\lesssim 1  ,
\end{equation}
and the third obeys
\begin{align}
\triangle_{3,a,p} \supm
&\le  \norm{ pD \supm * G_p \supm * G_p }_2 \bignorm{ \abs x^a G_p(x) }_2 	\nl
&\lesssim  \norm{  (pD \supm * G_p \supm * G_p) *  (pD \supm * G_p \supm * G_p)  }_\infty^{1/2} 	\nl
&\le e^{mL}  \norm{ pD }_1^{1/2} \norm{ pD * G_p * G_p * G_p \supm * G_p \supm }_\infty ^{1/2}	\nl
&= e^{mL} (p \square_{2,p} \supm)^{1/2}
\lesssim L^{ (-d+1)/2} \le 1 .
\end{align}
Finally, for the first triangle,
we apply Young's inequality twice to get
\begin{align}
    \triangle_{1,a,p} \supm
    &\le
    \norm{ p( \abs x^a D ) * G_p}_2 \norm{ G_p \supm * G_p \supm }_2
    \nnb & \le
     p \bignorm{  \abs x^a D(x) }_s \norm{G_p}_q \norm{ G_p \supm * G_p \supm }_2,
\end{align}
which holds for $q,s\ge 1$ such that $1+\frac 12 = \frac 1s + \frac 1q$.
We choose $q$ via $\frac 1q = \frac a d + \half$,
which satisfies $\frac 1q < 1 - \frac 2 d$ since $a < \frac{d-4}2$ by hypothesis.
Then \eqref{eq:G_norm} applies.  Also, our choice of $q$ implies that
$\frac ad =1-\frac 1s$, so
\begin{align}
    \bignorm { \abs x^a D(x) }_s
\lesssim L^{-d} \bignorm {\abs x^a \1_{\norm x_\infty \le L} }_s
\asymp L^{-d} L^{ a + d/s  }
= L^{a - (1 - \frac 1 s) d} =1.
\end{align}
Hence, by the massive infrared bound and the Parseval relation,
\begin{align}
\triangle_{1,a,p} \supm
&
\lesssim
\Bignorm{ \frac{ 1 }{ L^4 \abs k^4 \wedge 1 } }_2
\lesssim 1
.
\end{align}
This concludes the proof of \eqref{eq:Pi_moments-pf} for lattice trees.

\smallskip \noindent \emph{Diagrammatic estimate for lattice animals.}
For lattice animals,
as in \cite[Section~4.4]{HHS03},
the diagrammatic estimates differ from those of lattice trees only in that they contain an extra triangle at the origin.
We need additionally the tilted triangle
\begin{align}
T_p \supm = \norm{ G_p \supm * G_p * G_p }_\infty
\lesssim
\Bignorm{ \frac{ 1 }{ L^6 \abs k^6 \wedge 1 } }_1
\lesssim 1
\end{align}
and the tilted and weighted bubble
\begin{align}
B_{a,p} \supm = \bignorm{ ( \abs x^a G_p) * G_p \supm }_\infty
&\le \bignorm{ \abs x^a G_p(x) }_2 \bignorm{ G_p \supm }_2 	
\lesssim
\Bignorm{ \frac{ 1 }{ L^2 \abs k^2 \wedge 1 } }_2
\lesssim 1
.
\end{align}
For $N =0$ these lead to
$\sum_{x\in \Zd} \abs x^a \Pi_p \supzerom (x) \le B_{a,p} \supm $,
and for $N\ge1$ they give
\begin{equation}
\sum_{x\in \Z^d} \abs x^a  \Pi_p\supNm (x)
\le 2^{N-1} (2N+1)^{a} \Big( (2N) T_p \supm \triangle_{a,p} \supm (\tilde \square_p \supm )^{N-1}
	+ B_{a,p} \supm (\tilde \square_p \supm )^{N}   \Big).
\end{equation}
Summation of the above over $N$ again converges since $2 \tilde \square_p \supm$
is small.

\smallskip \noindent \emph{Estimates at $m = \mp$.}
If $b(m) \le 4$ for all $m < \mp$ and if $\mp L \le 1$,
then the estimates in \eqref{eq:Pi_moments_claim} hold uniformly in $m < \mp$.
Therefore, by Fatou's lemma, the tilted square diagram at $m = \mp$ also satisfies
\begin{equation}
\label{eq:square_pc}
\tilde \square_p \supmp
\le \liminf_{m\to \mp^-} \tilde \square_p \supm
\lesssim L^{-d+1}.
\end{equation}
Using the same argument to bound the triangle and bubble diagrams at $m=\mp$,
we obtain
\begin{equation} \label{eq:Pi_moments_crit}
\sum_{x\in \Z^d} \abs x^a  \sum_{N=0}^\infty \Pi_p^{(N,\mp)} (x)  \le K_a ,
\end{equation}
as desired.  This completes the proof.
\end{proof}

\section{Near-critical spread-out random walk}
\label{app:S}

In this appendix, we prove Proposition~\ref{prop:so-nn}, restated
here as Proposition~\ref{prop:S}
(with $z$ in place of $\mu$).
The constant $L_2$ is the same constant as in Lemma~\ref{lem:D_L}.

\begin{proposition} \label{prop:S}
Let $d>2$,
$\eta >0$,
$L \ge L_2$,
and $z\in [\half,1)$.
Then
\begin{align} \label{eq:mass_S}
m_S(z)^2 \sim \frac{2d}{\sigma_L^2} (1-z)
	\qquad \text{as $z\to 1$,}
\end{align}
and there are constants $K_S > 0$, $a_S \in (0,1)$ (independent of $L$)
and $\delta_2 = \delta_2(L) > 0$ such that
\begin{equation} \label{eq:S1bd}
S_z(x) \le \delta_{0,x} +
	\frac {K_S}{ L^{1-\eta} \nnnorm x^{d-2} } e^{- a_S m_S(z) \norm x_\infty }
	\qquad (x\in \Zd)
\end{equation}
uniformly in $z \in [1 - \delta_2, 1)$.
\end{proposition}

\subsection{Asymptotic formula for the mass}

\begin{proof}[Proof of \eqref{eq:mass_S}]
As in \eqref{eq:Adef}, we set
$A_z = \delta - zD$ so that $A_z * S_z = \delta$.
We adapt the proof of Theorem~\ref{thm:mass}, with $z$ now playing
the role of $p$.
By \eqref{eq:chiSinf},
we have $\hat A_z \supmS (0)= 0$ and thus
\begin{align} \label{eq:mass_S_pf}
1 - z
= \hat A_z (0)
= \hat A_z (0) - \hat A_z \supmS (0) .
\end{align}
As in \eqref{eq:m_asymp_pf},
we use the moment estimate \eqref{eq:A_moments} for $D\supm$ to rewrite
the right-hand side of \eqref{eq:mass_S_pf} as
\begin{align}
1 - z
= m_S(z)^2 L^2 \bigg( \frac{ \del_1^2 \hat A_z(0) }{2 L^2}
	- O(m_S(z)^\eps L^\eps e^{m_S(z) L})  \bigg)
\end{align}
for any $\eps \in (0,2]$.
The mass $m_S(z)$ goes to zero as $z\to 1$, since otherwise $\sum_{x\in \Zd} S_z(x)$
would be uniformly bounded in $z<1$, and it is not because it equals $(1-z)^{-1}$.
Since $\del_1^2 \hat A_z (0)
= z \sigma_L^2 / d$,
we get
\begin{align}
1 - z
\sim  m_S(z)^2  \frac{ z \sigma_L^2 }{ 2d }
	\qquad\text{as $z \to 1$.}
\end{align}
The claim \eqref{eq:mass_S} then follows.
\end{proof}

\subsection{Massive upper bound}

The proof of the massive upper bound \eqref{eq:S1bd}
follows the strategy of the proof of Theorem~\ref{thm:near_critical},
but now we compare $S_z$ to the nearest-neighbour random walk.

Let $d>2$ and $P = \frac 1 {2d} \1_{\abs x = 1}$.
For $\mu \in [0,1]$, the nearest-neighbour random walk two-point function
$C_\mu$ is the solution to the convolution equation
\begin{align}
(\delta - \mu P)* C_\mu = \delta
\end{align}
given by
\begin{equation}
C_\mu(x)= \int_{\T^d}\frac{e^{-ik\cdot x}}{1-\mu\hat P(k)} \frac{dk}{(2\pi)^d}
.
\end{equation}
The mass $m_0(\mu)$ of $C_\mu$ is the unique non-negative solution
to
\begin{align} \label{eq:m0_def}
\cosh m_0(\mu) = 1 + d \Big( \frac{ 1 - \mu }{ \mu } \Big)
\end{align}
(see \cite{MS22}).
By definition, $m_0(\mu) \to 0$ as $\mu \to 1$,
and, by
expanding $\cosh^{-1}$,
\begin{align} \label{eq:m0}
m_0(\mu)^2 = 2d \Big( \frac{ 1 - \mu }{ \mu } \Big)
	+ O \Big( \frac{ 1 - \mu }{ \mu } \Big)^2
	\qquad \text{as $\mu \to 1$.}
\end{align}
The following near-critical estimate of $C_\mu$ is obtained in \cite[Proposition~2.1]{Slad23_wsaw}.

\begin{lemma} \label{lem:C}
For $d>2$, there are contants $a_0 > 0$ and $a_1 \in (0,1)$ such that for all $\mu \in (0, 1]$,
\begin{align}
\label{eq:Cmu_bd}
C_\mu(x) \le \frac{ a_0 }{ \nnnorm x^{d-2} } e^{-a_1 m_0(\mu) \norm x _ \infty}
	\qquad(x\in \Zd).
\end{align}
\end{lemma}

We now write $\Ann_\mu = \delta - \mu P$ and again write $A_z = \delta - zD$.
Similar to \eqref{eq:G_isolate},
we isolate the leading decay of $S_z$ as
\begin{equation}
S_z = \lambda_z C_\muz + f_z ,
\qquad f_z = C_\muz  * E_z  * S_z,
\qquad E_z = \Ann_\muz - \lambda_z A_z,
\end{equation}
with the choices
\begin{align}
\label{eq:lambdaz}
\lambda_z &= \frac{1} { \hat A_z(0) - \sum_{x\in \Z^d} \abs x^2 A_z(x) }
	= \frac 1 { (1-z) + z\sigma_L^2 }, \\
\label{eq:muz}
\mu_z &= 1 - \lambda_z \hat A_z(0)
	= \frac {z\sigma_L^2} { (1-z) + z\sigma_L^2 } .
\end{align}
These choices cause $E_z$ to have vanishing zeroth and second moments.
We further isolate the zero-step contribution to $S_z$
using $S_z = \delta + z D * S_z $
(as done in \cite[Appendix~A]{LS24b}), to obtain
\begin{equation}
S_z
 = \delta + z D * ( \lambda_z C_\muz + f_z ) 	 = \delta + z\lambda_z C_{\mu_z} + \varphi_z,
\end{equation}
with
\begin{equation}
\varphi_z
 = - z \lambda_z A_1 * C_\muz + zD * f_z
\end{equation}
(note that we used $A_1=\delta -D$, \emph{not} $A_z$).
Multiplication by $e^{mx_1}$ gives
\begin{align} \label{eq:S_decomp}
S_z \supm = \delta + z\lambda_z C_\muz \supm + \varphi_z \supm .
\end{align}
We estimate $C_\muz \supm$ with Lemma~\ref{lem:C},
and estimate $\varphi_z \supm$ with the next lemma.

\begin{lemma} \label{lemma:phi}
Let $d>2$, $\eta > 0$, and $L \ge L_2$.
There is a $\delta_2 = \delta_2(L) > 0$ such that
for all $z \in [1 - \delta_2, 1)$
and $m \in [0, \half m_S(z) \wedge \half m_0(\muz) ]$,
$\hat \varphi_z \supm$ is $d-2$ times weakly differentiable, and
\begin{align} \label{eq:phibd}
\bignorm{ \grad^\alpha \hat \varphi_z \supm }_1 \lesssim L^{-(1-\eta)}
	\qquad (\abs \alpha \le d-2) ,
\end{align}
with the constant independent of $z,m,L$ but dependent on $\eta$.
\end{lemma}

\begin{proof}[Proof of \eqref{eq:S1bd} assuming Lemma~\ref{lemma:phi}]
By \eqref{eq:m0} and \eqref{eq:muz},
\begin{align} \label{eq:m0_asymp}
m_0(\muz)^2
\sim 2d \Big( \frac{ 1 - \mu_z }{ \mu_z } \Big)
= \frac{ 2d }{ \sigma_L^2 } \Big( \frac{ 1 - z }{ z } \Big)
	\qquad \text{as $z\to 1$.}
\end{align}
By comparing with the asymptotic formula for $m_S$ in \eqref{eq:mass_S},
this yields $m_0(\muz) \sim m_S(z)$ as $z\to 1$.
We can therefore find a constant $a_S \in (0,\half)$ and a possibly smaller $\delta_2$
such that
\begin{align}
a_S m_S(z) \le (a_1 \wedge \tfrac 1 2 )m_0(\mu_z)
\end{align}
for all $z \in [1 - \delta_2, 1)$,
where $a_1$ is the constant of Lemma~\ref{lem:C}.
In particular,
for $z \in [1 - \delta_2 , 1)$
and $m = a_S m_S(z)$, we can use Lemma~\ref{lem:C} to see that
\begin{align}
C_\muz \supm (x)
\le \frac{ a_0 }{ \nnnorm x^{d-2} } e^{-(a_1 m_0(\muz) - m ) \norm x _ \infty}
\le \frac{ a_0 }{ \nnnorm x^{d-2} } .
\end{align}
The bound on $\hat \varphi_z\supm$ from
Lemma~\ref{lemma:phi}, together with
Lemma~\ref{lem:fourier} (with $\psi = \varphi_z\supm$ and $n=d-2$)
and with $ \abs{ \varphi_z\supm (0) } \le \norm{ \hat \varphi_z \supm }_1$,
imply that
\begin{equation}
\abs{ \varphi_z \supm (x) }
\lesssim \frac { L^{-(1-\eta) } } { \nnnorm x^{d-2} } .
\end{equation}
We insert this into the decomposition \eqref{eq:S_decomp} of $S_z \supm$,
and use $z \lambda_z \lesssim L^{-2}$,
to obtain
\begin{align}
S_z \supm (x) - \delta_{0,x}
\lesssim  \frac{ a_0 } { L^2 \nnnorm x^{d-2} }
	+ \frac{ 1 } { L^{1-\eta} \nnnorm x^{d-2} } .
\end{align}
We can assume by $\Zd$-symmetry that $x_1 = \norm{ x }_\infty$,
and since $m = a_S m_S(z)$,
this implies that
\begin{align}
S_z(x) - \delta_{0,x}  \lesssim \frac{ 1 } { L^{1-\eta} \nnnorm x^{d-2} }
	e^{ - a_S m_S(z) \norm x_\infty },
\end{align}
as desired.
\end{proof}

\subsection{Proof of Lemma~\ref{lemma:phi}}

It remains to prove Lemma~\ref{lemma:phi}.
For this, we use the following analogue of
Lemma~\ref{lem:AFEbds}.
Its proof involves only a minor
adaptation of the proof of Lemma~\ref{lem:AFEbds}.

\begin{lemma} \label{lem:AFEbds_S}
Let $d>2$ and $L \ge L_2$.
There is a $\delta_2 = \delta_2(L) > 0$ such that
for all $z \in [1 - \delta_2, 1)$,
for all $m \in [0, \half m_S(z) \wedge \half m_0(\muz) ]$,
for all multi-indices $\gamma$ with $\abs \gamma \le d-2$,
and for all $q_1, q_2$ with
\begin{equation}
    \frac{ \abs \gamma } d < q_1\inv < 1,
    \qquad
    \frac{ 1 + \abs \gamma } d < q_2\inv < 1,
\end{equation}
the norm inequalities
\begin{equation}  \label{eq:EAF_S}
\biggnorm{ \frac{ \grad^\gamma \hat \Ann_\muz \supm }{ \hat \Ann_\muz \supm } }_{q_1}, \;
\biggnorm{ \frac{ \grad^\gamma \hat A_z \supm }{ \hat A_z \supm } }_{q_1}, \;
\biggnorm{ \frac{ \grad^\gamma \hat E_z \supm }{ \hat \Ann_\muz \supm \hat A_z \supm } }_{q_2}
\lesssim 1
\end{equation}
hold, with constants independent of $z,m,L$.
\end{lemma}

\begin{proof}
We begin with $A_z = \delta - z D$.
The claim on $\grad^\gamma \hat A_z \supm / \hat A_z \supm$
is essentially proved in Lemma~\ref{lem:AFEbds},
which handles $A_\mup = \delta - \mup D$.
That proof uses only the massive infrared bound of \eqref{eq:Am_lb} and
the Taylor expansion for \eqref{eq:A_linear}.
It can be repeated provided $L$ is large
and $m_S(z) L$ is small.
Since $m_S(z) \to 0$ as $z\to 1$,
we obtain the result for $z \in [1- \delta_2,1)$ with $\delta_2$ small.

The proof for $\hat\Ann_\muz$ follows similar steps to the proof for $\hat A_z$ in
Lemma~\ref{lem:AFEbds}, but it is simpler because there is no need to track
$L$-dependence.  The details are given in
\cite[(2.58)--(2.61)]{Liu24},
for $m \in [0, \half m_0(\muz)]$,
using only smallness of $m_0(\muz)$.
Since $m_0(\muz) \to 0$ as $z \to 1$ by \eqref{eq:m0_asymp},
we obtain the result by taking a possibly smaller $\delta_2$.

For the remaining claim on $\grad^\gamma \hat E_z \supm / ( \hat \Ann_\muz \supm \hat A_z \supm )$,
the proof is a straightforward adaptation of the proof of Lemma~\ref{lem:AFEbds},
as in \cite[Lemma~A.2]{LS24b}.
We therefore only highlight the difference
from the proof of \cite[Lemma~A.2]{LS24b} here.
We let $\eps = 1$ play the role of $\tau$ in that proof.
The massive infrared bounds
\begin{align}
\hat \Ann_\muz \supm (k) \gtrsim ( \abs k + m )^2,
\qquad
\hat A_z \supm (k) \gtrsim \begin{cases}
L^2 (\abs k + m)^2		& (\abs k < L\inv) \\
1					&(\abs k \ge L\inv) ,
\end{cases}
\end{align}
hold for $m \in [0, \half m_S(z) \wedge \half m_0(\muz) ]$;
the first is elementary and is proved, e.g., in \cite[(2.19)]{Slad23_wsaw}, and the
second is \eqref{eq:massive_infrared}.
We also use $\lambda_z \lesssim L^{-2}$
to see that
\begin{align} \label{eq:E_S}
\abs{ E_z \supm } = \abs{ (\delta - \muz P \supm) - \lambda_z (\delta - z D \supm) }
\lesssim  (1 + L^{-2}) \delta + P \supm + L^{-2} D \supm .
\end{align}
With this to estimate moments of $E_z \supm$,
a Taylor expansion as in the proof of Lemma~\ref{lem:E} gives
\begin{align}
\abs{ \grad^\gamma \hat E_z \supm (k) }
\lesssim  \begin{cases}
(\abs k + m)^\eps e^{mL}		&(\abs \gamma = 0)  	\\
L^{\eps} (\abs k  + m)^{2 + \eps - \abs \gamma }  e^{mL}
	&(\abs \gamma \le 2) .
\end{cases}
\end{align}
Also, by taking the Fourier transform in \eqref{eq:E_S} and using Lemma~\ref{lem:D_L} ,
we have
\begin{align} \label{eq:E_gamma_S}
\bignorm{ \grad^\gamma \hat E_z \supm }_r
\lesssim  1 + L^{-2 + \abs \gamma - d/r }
	\qquad (0 \le r\inv < 1,\ mL\le 1) .
\end{align}
These estimates allow us to bound the two terms of
\begin{align} \label{eq:EAF_decomp_S}
\biggabs{ \frac{ \grad^\gamma \hat E_z \supm }{ \hat \Ann_\muz \supm \hat A_z \supm } (k) }
\lesssim
\frac{ \abs{ \grad^\gamma \hat E_z \supm (k) } } { L^2 ( \abs k + m )^4 } \1_{B_L}
+ \frac { \abs{ \grad^\gamma \hat E_z \supm (k) } } { ( \abs k + m )^{2} }
\end{align}
using H\"older's inequality, as in the proof of \cite[Lemma~A.2]{LS24b}.
The only difference from that proof is that we now have cancellation of powers of $(\abs k + m)$, instead of powers of $\abs k$.
\end{proof}

\begin{proof}[Proof of Lemma~\ref{lemma:phi}]
It suffices to consider small $\eta>0$.
By definition, $\hat \varphi_z \supm = z  \hat f_z \supm \hat D \supm
- z \lambda_z \hat C_\muz \supm \hat A_1 \supm$.
Since $\lambda_z \lesssim L^{-2}$ uniformly in $z\in [\half,1]$,
by the product rule it suffices to prove that, for $|\alpha_1|+|\alpha_2|=|\alpha| \le d-2$,
\begin{align}
\label{eq:Df}
\bignorm{ \grad^{\alpha_1} \hat f_z \supm
	\grad^{\alpha_2} \hat D \supm  }_1  	&\lesssim L^{-(1-\eta)} , \\
\label{eq:CF}
\bignorm{ \grad^{\alpha_1} \hat C_\muz \supm
	\grad^{\alpha_2} \hat A_1 \supm  }_1  	&\lesssim L^{\eta} \le L^{1+\eta} .
\end{align}

By Lemma~\ref{lem:AFEbds_S}, it follows exactly as in the proof of Proposition~\ref{prop:f} that
the function $\hat f_z \supm = \hat E_z \supm / (\hat \Ann_\muz \supm \hat A_z \supm )$
is $d-2$ times weakly differentiable, and,
for any $\abs \alpha \le d-2$,
\begin{align} \label{eq:f_alpha_S}
\bignorm {   \grad^\alpha \hat f_z \supm  }_{r} \lesssim 1
	\qquad (r\inv > \frac{ \abs \alpha + 1 }d),
\end{align}
with the constant independent of $z,m,L$.
The bound \eqref{eq:Df} then follows from H\"older's inequality, \eqref{eq:f_alpha_S}, and \eqref{eq:D_alpha}:
\begin{align}
\bignorm{ \grad^{\alpha_1} \hat f_z \supm
	\grad^{\alpha_2} \hat D \supm  }_1
&\le
\bignorm{ \grad^{\alpha_1} \hat f_z \supm }_{\frac d {\abs {\alpha_1}+1+\eta }}
\bignorm{ \grad^{\alpha_2} \hat D \supm   }_{\frac d { \abs{\alpha_2}+1-\eta }}
\lesssim  L^{ \abs{\alpha_2} - (\abs{\alpha_2} + 1-\eta) } = L^{-(1-\eta)}.
\end{align}

For \eqref{eq:CF}, we first recall from
\cite[Lemma~2.2]{Slad23_wsaw}
the elementary fact that
$\abs{ \grad^{\alpha_1} \hat C_\muz \supm (k) }
\lesssim ( \abs k + m )^{-(2 + \abs {\alpha_1} )}$
for all $m \in [0, \half m_0(\muz)]$ and all $\alpha_1$.
Consider first the case $\abs{\alpha_2} < \eta$, \ie, $\alpha_2 = 0$.
Since $\hat A_1 (0) = \sum_{x\in\Z^d} A_1(x) = 0$ and since $1+\eta \le 2$,
a Taylor expansion around $(k,m)=(0,0)$, similar to the proof of Lemma~\ref{lem:E}, gives
\begin{align}
\abs{ \hat A_1 \supm (k) }
= \abs{  \hat A_1 \supm (k) - \hat A_1(0) }
\lesssim L^{\eta} (\abs k + m)^{ \eta } e^{mL} .
\end{align}
We can again bound $e^{mL}$ by a constant because $mL$ is small.
It follows that
\begin{align}
\bignorm{ \grad^{\alpha} \hat C_\muz \supm \cdot \hat A_1 \supm }_1
\lesssim  L^{\eta} \biggnorm{ \frac 1 { (\abs k+m)^{ 2 + \abs{\alpha} - \eta} } }_1 \lesssim L^{\eta} ,
\end{align}
since $2 + \abs{ \alpha } - \eta \le d-\eta < d$.
If instead
$\abs{ \alpha_2} \ge \eta$, then
we use H\"older's inequality and the norm bound \eqref{eq:D_alpha} on $\hat F_1 \supm = 1 - \hat D \supm$ to
complete the proof of \eqref{eq:CF} with
\begin{align}
\bignorm{ \grad^{\alpha_1} \hat C_\muz \supm
	\grad^{\alpha_2} \hat A_1 \supm  }_1
\le  \bignorm{ \grad^{\alpha_1} \hat C_\muz \supm }_{\frac{d}{2 + \abs{ \alpha_1 } + \eta} }
	\bignorm{ \grad^{\alpha_2} \hat A_1 \supm }_{\frac{d}{ \abs{ \alpha_2 } -\eta} }
\lesssim  L^{  \abs{ \alpha_2 }    -( \abs{ \alpha_2 } -\eta )  }
= L^{\eta}.
\end{align}
This concludes the proof.
\end{proof}

\section*{Acknowledgements}
The work of both authors was supported in part by NSERC of Canada.
The authors thank the Isaac Newton Institute for Mathematical Sciences, Cambridge, for support and hospitality during the programme \emph{Stochastic systems for anomalous diffusion}, where work on this paper was undertaken; this work was supported by EPSRC grant EP/R014604/1.  The work of GS was supported in part by the Clay
Mathematics Institute.

\subsection*{Data availability}
No data is generated or analysed in our work.

\subsection*{Conflict of interest}
The authors have no relevant financial or non-financial interests to disclose.


\begin{thebibliography}{10}

\bibitem{AS80}
M.~Aizenman and B.~Simon.
\newblock Local {Ward} identities and the decay of correlations in
  ferromagnets.
\newblock {\em Commun. Math. Phys.}, {\bf 77}:137--143, (1980).

\bibitem{BBR10}
R.~Barequet, G.~Barequet, and G.~Rote.
\newblock Formulae and growth rates of high-dimensional polycubes.
\newblock {\em Combinatorica}, {\bf 30}:257--275, (2010).

\bibitem{BEHK22}
B.~Berche, T.~Ellis, Y.~Holovatch, and R.~Kenna.
\newblock Phase transitions above the upper critical dimension.
\newblock {\em SciPost Phys. Lect. Notes}, {\bf 136}:paper 60, (2022).

\bibitem{BKW12}
B.~Berche, R.~Kenna, and J.-C. Walter.
\newblock Hyperscaling above the upper critical dimension.
\newblock {\em Nucl. Phys. B}, {\bf 865} [FS]:115--132, (2012).

\bibitem{BFG86}
A.~Bovier, J.~Fr\"{o}hlich, and U.~Glaus.
\newblock Branched polymers and dimensional reduction.
\newblock In K.~Osterwalder and R.~Stora, editors, {\em Critical Phenomena,
  Random Systems, Gauge Theories}, Amsterdam, (1986). North-Holland.

\bibitem{BI03a}
D.C. Brydges and J.Z. Imbrie.
\newblock Branched polymers and dimensional reduction.
\newblock {\em Ann. Math.}, {\bf 158}:1019--1039, (2003).

\bibitem{CFHP23}
M.~Cabezas, A.~Fribergh, M.~Holmes, and E.~Perkins.
\newblock Historical lattice trees.
\newblock {\em Commun. Math. Phys.}, {\bf 401}:435--496, (2023).

\bibitem{DGGZ22}
Y.~Deng, T.M. Garoni, J.~Grimm, and Z.~Zhou.
\newblock Unwrapped two-point functions on high-dimensional tori.
\newblock {\em J. Stat. Mech: Theory Exp.}, 053208, (2022).

\bibitem{DS98}
E.~Derbez and G.~Slade.
\newblock The scaling limit of lattice trees in high dimensions.
\newblock {\em Commun.\ Math.\ Phys.}, {\bf 193}:69--104, (1998).

\bibitem{DP25b}
H.~Duminil-Copin and R.~Panis.
\newblock An alternative approach for the mean-field behaviour of spread-out
  {Bernoulli} percolation in dimensons $d>6$.
\newblock Preprint, \url{https://arxiv.org/pdf/2410.03647}, (2024).

\bibitem{DP25a}
H.~Duminil-Copin and R.~Panis.
\newblock An alternative approach for the mean-field behaviour of weakly
  self-avoiding walks in dimensions $d>4$.
\newblock Preprint, \url{https://arxiv.org/pdf/2410.03649}, (2024).

\bibitem{DP24}
H.~Duminil-Copin and R.~Panis.
\newblock New lower bounds for the (near) critical {Ising} and $\varphi^4$
  models' two-point functions.
\newblock To appear in {\it Commun.\ Math.\ Phys.} Preprint,
  \url{https://arxiv.org/pdf/2404.05700}, (2024).

\bibitem{FMPPS23}
N.G. Fytas, V.~Mart\'{i}n-Mayor, G.~Parisi, M.~Picco, and N.~Sourlas.
\newblock Finite-size scaling of the random-field {Ising} model above the upper
  critical dimension.
\newblock {\em Phys. Rev. E}, {\bf 108}:044146, (2023).

\bibitem{GP23}
A.~Georgakopoulos and C.~Panagiotis.
\newblock On the exponential growth rates of lattice animals and interfaces.
\newblock {\em Combin. Probab. Comput.}, {\bf 32}:912--955, (2023).

\bibitem{Graf14}
L.~Grafakos.
\newblock {\em Classical Fourier Analysis}.
\newblock Springer, New York, 3rd edition, (2014).

\bibitem{GEZGD17}
J.~Grimm, E.M. El\c{c}i, Z.~Zhou, T.M. Garoni, and Y.~Deng.
\newblock Geometric explanation of anomalous finite-size scaling in high
  dimensions.
\newblock {\em Phys. Rev. Lett.}, {\bf 118}:115701, (2017).

\bibitem{Hara90}
T.~Hara.
\newblock Mean field critical behaviour for correlation length for percolation
  in high dimensions.
\newblock {\em Probab. Theory Related Fields}, {\bf 86}:337--385, (1990).

\bibitem{HHS03}
T.~Hara, R.~van~der Hofstad, and G.~Slade.
\newblock Critical two-point functions and the lace expansion for spread-out
  high-dimensional percolation and related models.
\newblock {\em Ann. Probab.}, {\bf 31}:349--408, (2003).

\bibitem{HS90a}
T.~Hara and G.~Slade.
\newblock Mean-field critical behaviour for percolation in high dimensions.
\newblock {\em Commun. Math. Phys.}, {\bf 128}:333--391, (1990).

\bibitem{HS90b}
T.~Hara and G.~Slade.
\newblock On the upper critical dimension of lattice trees and lattice animals.
\newblock {\em J. Stat. Phys.}, {\bf 59}:1469--1510, (1990).

\bibitem{HS92c}
T.~Hara and G.~Slade.
\newblock The number and size of branched polymers in high dimensions.
\newblock {\em J. Stat. Phys.}, {\bf 67}:1009--1038, (1992).

\bibitem{HS92a}
T.~Hara and G.~Slade.
\newblock Self-avoiding walk in five or more dimensions. {I.} {The} critical
  behaviour.
\newblock {\em Commun.\ Math.\ Phys.}, {\bf 147}:101--136, (1992).

\bibitem{HS02}
R.~van~der Hofstad and G.~Slade.
\newblock A generalised inductive approach to the lace expansion.
\newblock {\em Probab. Theory Related Fields}, {\bf 122}:389--430, (2002).

\bibitem{Holm08}
M.~Holmes.
\newblock Convergence of lattice trees to super-{B}rownian motion above the
  critical dimension.
\newblock {\em Electr.\ J.\ Probab.}, {\bf 13}:671--755, (2008).

\bibitem{Holm16}
M.~Holmes.
\newblock Backbone scaling for critical lattice trees in high dimensions.
\newblock {\em J. Phys. A: Math. Theor.}, {\bf 49}:314001, (2016).

\bibitem{HP20}
M.~Holmes and E.~Perkins.
\newblock On the range of lattice models in high dimensions.
\newblock {\em Probab. Theory Related Fields}, {\bf 176}:941--1009, (2020).

\bibitem{HMS23}
T.~Hutchcroft, E.~Michta, and G.~Slade.
\newblock High-dimensional near-critical percolation and the torus plateau.
\newblock {\em Ann. Probab.}, {\bf 51}:580--625, (2023).

\bibitem{Jans15}
E.~J. Janse~van Rensburg.
\newblock {\em The Statistical Mechanics of Interacting Walks, Polygons,
  Animals and Vesicles}.
\newblock Oxford University Press, Oxford, 2nd edition, (2015).

\bibitem{KS24}
N.~Kawamoto and A.~Sakai.
\newblock Spread-out limit of the critical points for lattice trees and lattice
  animals in dimensions $d > 8$.
\newblock {\em Combin. Probab. Comput.}, {\bf 33}:238--269, (2024).

\bibitem{Liu24}
Y.~Liu.
\newblock A general approach to massive upper bound for two-point function with
  application to self-avoiding walk torus plateau.
\newblock Preprint, \url{https://arxiv.org/pdf/2310.17321}, (2023).

\bibitem{LPS25-Ising}
Y.~Liu, R.~Panis, and G.~Slade.
\newblock The torus plateau for the high-dimensional {Ising} model.
\newblock Preprint, \url{https://arxiv.org/pdf/2405.17353}, (2024).

\bibitem{LPS25-universal}
Y.~Liu, J.~Park, and G.~Slade.
\newblock Universal finite-size scaling in high-dimensional critical phenomena.
\newblock Preprint, \url{https://arxiv.org/pdf/2412.08814}, (2024).

\bibitem{LS24b}
Y.~Liu and G.~Slade.
\newblock Gaussian deconvolution and the lace expansion for spread-out models.
\newblock To appear in {\it Ann.\ Inst.\ H.\ Poincar\'e Probab.\ Statist.}
  Preprint, \url{https://arxiv.org/pdf/2310.07640}, (2023).

\bibitem{LS25_profile}
Y.~Liu and G.~Slade.
\newblock Critical scaling profile for trees and connected subgraphs on the
  complete graph.
\newblock Preprint, \url{https://arxiv.org/pdf/2412.05503}, (2024).

\bibitem{LS24a}
Y.~Liu and G.~Slade.
\newblock Gaussian deconvolution and the lace expansion.
\newblock {\em Probab.\ Theory Related Fields}, (2024).
\newblock \url{https://doi.org/10.1007/s00440-024-01350-9}.

\bibitem{LI79}
T.C. Lubensky and J.~Isaacson.
\newblock Statistics of lattice animals and dilute branched polymers.
\newblock {\em Phys. Rev. A}, {\bf 20}:2130--2146, (1979).

\bibitem{MS93}
N.~Madras and G.~Slade.
\newblock {\em The Self-Avoiding Walk}.
\newblock Birkh{\"a}user, Boston, (1993).

\bibitem{MS13}
Y.~Mej\'ia~Miranda and G.~Slade.
\newblock Expansion in high dimension for the growth constants of lattice trees
  and lattice animals.
\newblock {\em Combin. Probab. Comput.}, {\bf 22}:527--565, (2013).

\bibitem{MPS23}
E.~Michta, J.~Park, and G.~Slade.
\newblock Boundary conditions and universal finite-size scaling for the
  hierarchical $|\varphi|^4$ model in dimensions $4$ and higher.
\newblock Preprint, \url{https://arxiv.org/pdf/2306.00896}, (2023).

\bibitem{MS22}
E.~Michta and G.~Slade.
\newblock Asymptotic behaviour of the lattice {Green} function.
\newblock {\em ALEA, Lat. Am. J. Probab. Math. Stat.}, {\bf 19}:957--981,
  (2022).

\bibitem{MS23}
E.~Michta and G.~Slade.
\newblock Weakly self-avoiding walk on a high-dimensional torus.
\newblock {\em Probab. Math. Phys.}, {\bf 4}:331--375, (2023).

\bibitem{Olve97}
F.W.J. Olver.
\newblock {\em Asymptotics and Special Functions}.
\newblock CRC Press, New York, (1997).

\bibitem{Pani24_thesis}
R.~Panis.
\newblock {\em Applications of path expansions to statistical mechanics}.
\newblock PhD thesis, University of Geneva, (2024).

\bibitem{PS81}
G.~Parisi and N.~Sourlas.
\newblock Critical behavior of branched polymers and the {L}ee--{Y}ang edge
  singularity.
\newblock {\em Phys. Rev. Lett.}, {\bf 46}:871--874, (1981).

\bibitem{PS25}
J.~Park and G.~Slade.
\newblock Boundary conditions and the two-point function plateau for the
  hierarchical $|\varphi|^4$ model in dimensions $4$ and higher.
\newblock Preprint, \url{https://arxiv.org/pdf/2405.17344}, (2024).

\bibitem{Penr94}
M.D. Penrose.
\newblock Self-avoiding walks and trees in spread-out lattices.
\newblock {\em J. Stat. Phys.}, {\bf 77}:3--15, (1994).

\bibitem{Slad06}
G.~Slade.
\newblock {\em The Lace Expansion and its Applications.}
\newblock Springer, Berlin, (2006).
\newblock Lecture Notes in Mathematics Vol. 1879. Ecole d'Et\'{e} de
  Probabilit\'{e}s de Saint--Flour XXXIV--2004.

\bibitem{Slad23_wsaw}
G.~Slade.
\newblock The near-critical two-point function and the torus plateau for weakly
  self-avoiding walk in high dimensions.
\newblock {\em Math. Phys. Anal. Geom.}, {\bf 26}:article 6, (2023).

\bibitem{TH87}
H.~Tasaki and T.~Hara.
\newblock Critical behaviour in a system of branched polymers.
\newblock {\em Prog. Theor. Phys. Suppl.}, {\bf 92}:14--25, (1987).

\bibitem{WY14}
M.~Wittmann and A.P. Young.
\newblock Finite-size scaling above the upper critical dimension.
\newblock {\em Phys. Rev. E}, {\bf 90}:062137, (2014).

\end{thebibliography}
\end{document}